\documentclass[11pt,oneside,english]{amsbook}
\usepackage{times}
\usepackage[T1]{fontenc}
\usepackage[latin1]{inputenc}
\usepackage{geometry}
\usepackage{graphicx}
\usepackage{setspace}

\makeatletter
 \numberwithin{section}{chapter}
 \theoremstyle{plain}    
 \newtheorem{thm}{Theorem}[section]
 \numberwithin{equation}{section} 
 \numberwithin{figure}{section} 
 \theoremstyle{plain}
 \theoremstyle{plain}    
 \newtheorem{cor}[thm]{Corollary} 
 \theoremstyle{plain}    
 \newtheorem{lem}[thm]{Lemma} 
 \theoremstyle{plain}    
 \newtheorem{prop}[thm]{Proposition} 
 \theoremstyle{remark}
 \newtheorem{rem}[thm]{Remark}
 \theoremstyle{plain}    
 \newtheorem{fact}[thm]{Fact} 

\newcommand{\ZZ}{\mathbb{Z}}

\newcommand{\CC}{\mathbb{C}}
\newcommand{\CCh}{\hat{\mathbb{C}}}

\newcommand{\HH}{\mathbb{H}}
\newcommand{\RR}{\mathbb{R}}

\newcommand{\DD}{\mathbb{D}}

\newcommand{\ra}{\rightarrow}
\newcommand{\FF}{\mathcal{F}}

\newcommand{\re}{\text{Re }}
\newcommand{\im}{\text{Im }}

\newcommand{\Lg}{\mathbf{g}}
\newcommand{\Lk}{\mathbf{k}}
\newcommand{\TT}{\mathcal{T}}
\newcommand{\LL}{\mathcal{L}}
\newcommand{\Mb}{\mathbf{M}}

\usepackage{babel}
\makeatother
\begin{document}
\title{Discrete Inverse Scattering Theory for NMR Pulse Design}
\author{Jeremy F. Magland\footnote{This manuscript contains my Ph.D. thesis,
    which was accepted by the Department of Mathematics of the University of PA in May 2004.}\\Department of Radiology\\ Hospital of the
    University of Pennsylvania\\
Philadelphia, PA 19104}
\maketitle
\tableofcontents{}

\chapter{\label{cha:Introduction}Introduction}

The problem of selective excitation pulse design in nuclear magnetic
resonance (NMR) corresponds, in mathematics, to the inversion of a
certain mapping, $\TT$, which we call the \textit{selective excitation
transform}. This transform, which is a kind of non-linear Fourier
transform, maps complex-valued functions of time, called pulses, to
unit $3$-vector valued functions of frequency, called magnetization
profiles. The explicit definition of $\TT$ is given in Section \ref{sec:The-selective-excitation-transform}.
The theory of inverting $\TT$ has been shown to coincide with the
theory of inverse scattering for the \textit{Zakharov-Shabat (ZS)}
$2\times2$ system (see Section \ref{sec:The-Zakharov-Shabat-system}
and for example \cite{Rourke Morris}). Numerous authors have studied
this inverse scattering problem, for example Ablowitz et al. \cite{AKNS},
and the results have been applied to NMR pulse design (see for example
\cite{Rourke Morris,Carlson}). However, no stable and efficient algorithm
has been given in the literature for generating the full space of
solutions to the inverse problem. For this and other reasons, less
exact methods for NMR pulse design, such as the \textit{Fourier transform}
method and the \textit{Shinnar-Le Roux (SLR)} method, have been used
in practice instead of the more exact and more flexible \textit{inverse
scattering (IST)} method. In this thesis, we present the \textit{discrete
inverse scattering transform (DIST) algorithm} for efficiently solving
the full inverse scattering problem relating to NMR pulse design. 

In this introductory chapter, we define the continuum and discrete
selective excitation transforms and describe the problem of selective
excitation pulse design. The theoretical results are summarized in
Sections \ref{sec:Main results continuum}, \ref{sec:main results discrete},
and \ref{sec:energy formulas}, and the main algorithms are described
in Section \ref{sec:Algorithms}.

The second chapter is devoted to proving the main theorems and deriving
the algorithms. We introduce the discrete scattering theory, which
is completely analogous to the standard continuum theory.

The third chapter describes how to apply the theory to practical NMR
pulse design.

See \cite{Epstein 2} and \cite{Epstein} for a detailed mathematical
introduction to NMR imaging and NMR pulse design.

\section{\label{sec:The-selective-excitation-transform}The selective excitation
transform}

In this section we define the selective excitation transform, which
maps a complex function of time to a unit $3$-vector valued function
of frequency.

Let $\omega:\RR\rightarrow\CC$ be a function of time (usually we
think of $\omega$ as smooth and supported on a finite interval),
and suppose that for every frequency $z\in\RR$, there is a solution
$M_{-}(z;\cdot):\RR\rightarrow\RR^{3}$ to the frequency dependent
Bloch equation (without relaxation)\begin{equation}
\frac{d}{dt}M_{-}(z;t)=M_{-}(z;t)\times\left[\begin{array}{c}
\re\omega(t)\\
\im\omega(t)\\
z\end{array}\right]\label{eq:bloch frequency}\end{equation}
 normalized by \begin{equation}
\lim_{t\rightarrow-\infty}M_{-}(z;t)=\left[\begin{array}{c}
0\\
0\\
1\end{array}\right].\label{eq:bf 1}\end{equation}
 Let us show that such a solution is unique. If $M_{1}$ and $M_{2}$
are two solutions to (\ref{eq:bloch frequency}) then $M_{1}(z;t)^{T}M_{2}(z;t)$
is independent of $t$ because \begin{eqnarray*}
\frac{d}{dt}M_{1}(z;t)^{T}M_{2}(z;t) & = & \left[X_{B}(z;t)M_{1}(z;t)\right]^{T}M_{2}(z;t)+M_{1}(z;t)^{T}\left[X_{B}(z;t)M_{2}(z;t)\right]\\
 & = & M_{1}^{T}(z;t)\left(X_{B}(z;t)^{T}+X_{B}(z;t)\right)M_{2}(z;t)\\
 & = & 0,\end{eqnarray*}
where \[
X_{B}(z;t)=\left[\begin{array}{ccc}
0 & z & -\im\omega(t)\\
-z & 0 & \re\omega(t)\\
\im\omega(t) & -\re\omega(t) & 0\end{array}\right].\]
 Therefore, (\ref{eq:bf 1}) implies that $M_{-}(z;t)^{T}M_{-}(z;t)=1$.
So $M_{-}(z;t)$ is a unit vector for all $z$ and $t$. It is unique
because, if $M_{2}(z;t)$ is any unit vector satisfying $M_{-}(z;t)^{T}M_{2}(z;t)=1$,
then the Cauchy-Schwarz inequality implies that $M_{-}(z;t)=M_{2}(z;t)$.

Suppose that $\omega$ decays sufficiently so that \[
\Mb(z):=\lim_{t\rightarrow+\infty}\left[\begin{array}{ccc}
\re e^{izt} & -\im e^{izt} & 0\\
\im e^{izt} & \re e^{izt} & 0\\
0 & 0 & 1\end{array}\right]M_{-}(z;t)\]
 exists for all $t$. For example, this limit exists whenever $\omega$
is integrable (see \cite{Epstein}). We call $\omega$ the \textit{pulse}
and $\Mb$ the \textit{resulting magnetization profile}. The map $\omega\mapsto\Mb$
is called the \textit{selective excitation transform}, and we write
$\Mb=\TT\omega$. In Section \ref{sec:Examples} we plot several examples
of pulses and their resulting magnetization profiles.

\section{The problem of selective excitation pulse design}

In selective excitation pulse design, we usually start with an ideal
magnetization profile $\Mb_{\textrm{ideal}}:\RR\rightarrow S^{2}\subset\RR^{3}$.
This is a unit 3-vector valued function of frequency which is typically
equal to $\left[\begin{array}{ccc}
0 & 0 & 1\end{array}\right]^{t}$ outside some finite interval. The problem is to find a pulse $\omega:\RR\rightarrow\CC$
such that the resulting magnetization profile $\Mb=\TT\omega$ is
a good approximation to $\Mb_{\textrm{ideal}}$. For practical applications,
$\omega$ should have finite duration. That is, it should be supported
in some finite interval $[\rho-T,\rho]\subset\RR$. The number $T$
is called the \textit{duration} of the pulse, and the value $\rho$
is called the \textit{rephasing time}. Depending on the application,
it may be important to design a pulse with the shortest possible duration
and minimal (perhaps zero) rephasing time (pulses with $\rho=0$ are
called \textit{self refocused}). It is also often necessary to limit
the \textit{energy} \[
E_{\omega}:=\int_{-\infty}^{\infty}\left|\omega(t)\right|^{2}dt\]
 of the pulse, as well as the maximum amplitude. For many application
it is important that the pulse behaves well under imperfect magnetic
field conditions. That is, $\frac{\partial\Mb}{\partial\omega}$ should
not be too large.

The standard example of an ideal magnetization profile is \[
M_{ideal}(z)=\begin{cases}
\left[\begin{array}{c}
0\\
\sin\theta_{0}\\
\cos\theta_{0}\end{array}\right] & \textrm{if }|z-z_{0}|<c_{0}\\
\\\left[\begin{array}{c}
0\\
0\\
1\end{array}\right] & \textrm{if }|z-z_{0}|>c_{0}.\end{cases}\]
 Here, the magnetization is rotated $\theta_{0}$ radians around the
$x$-axis for frequencies in the interval $|z-z_{0}|<c_{0}$. Outside
of this frequency interval, the magnetization is kept at equilibrium.
The angle $\theta_{0}$ is called the \textit{flip angle}. See the
next section and Chapter \ref{cha:Pulses-with-finite} for plots of
pulses producing magnetization profiles that approximate such a profile.

\section{Examples\label{sec:Examples}}

When $\int_{-\infty}^{\infty}|\omega(t)|dt$ is small, the resulting
magnetization profile $\Mb=\TT\omega$ very closely resembles the
Fourier transform of $\omega$ (see for example \cite{Epstein}).
This fact is illustrated in Figures \ref{cap:const} and \ref{cap:sinc}.
Figure \ref{cap:const} shows a pulse which is constant over a finite
interval. Notice that at zero frequency, the magnetization is rotated
$\frac{\pi}{2}$ radians around the $x$-axis, and away from zero
the transverse magnetization resembles a $\textrm{sinc}$ function.
Figures \ref{cap:sinc} and \ref{cap:sinc120} show truncated sinc
pulses. Such pulses are often used in practice to produce a magnetization
profile which is approximately constant within some finite interval,
and approximately in equilibrium outside the interval. Notice that
as the flip angle increases from $90^{\circ}$ to $120^{\circ}$,
the resulting magnetization profile does a worse job approximating
the ideal.

\begin{figure}

\caption{\label{cap:const}A constant pulse and the resulting magnetization:
$T=1,$ $\rho=\frac{1}{2}$. }

\begin{center}\includegraphics[scale=0.7]{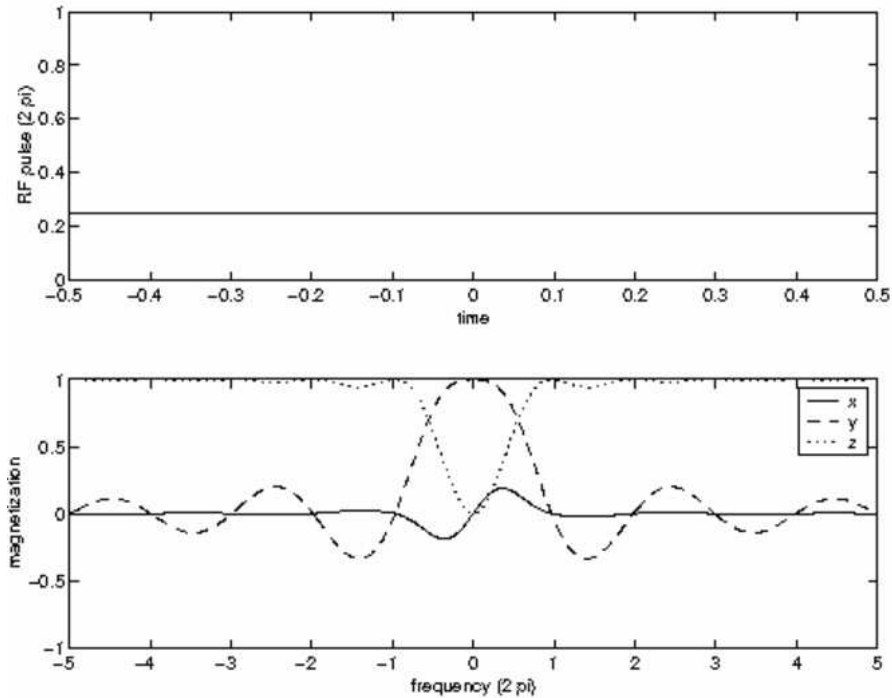}\end{center}
\end{figure}

\begin{figure}

\caption{\label{cap:sinc}A $90^{\circ}$ sinc pulse and the resulting magnetization:
N=10, $\rho\approx5$.}

\begin{center}\includegraphics[%
  scale=0.7]{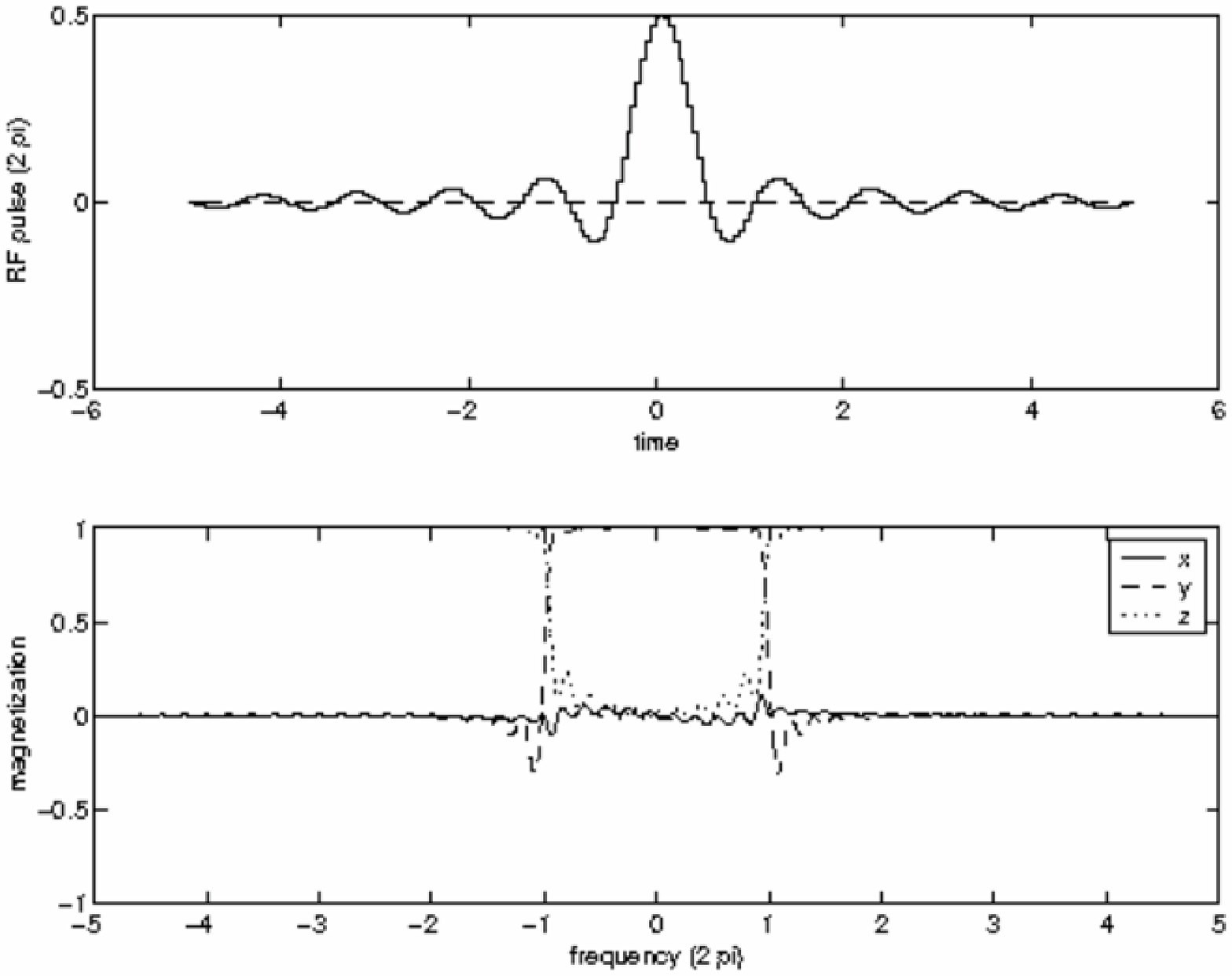}\end{center}
\end{figure}

\begin{figure}

\caption{\label{cap:sinc120}A $120^{\circ}$ sinc pulse and the resulting
magnetization: N=10, $\rho\approx5$.}

\begin{center}\includegraphics[%
  scale=0.7]{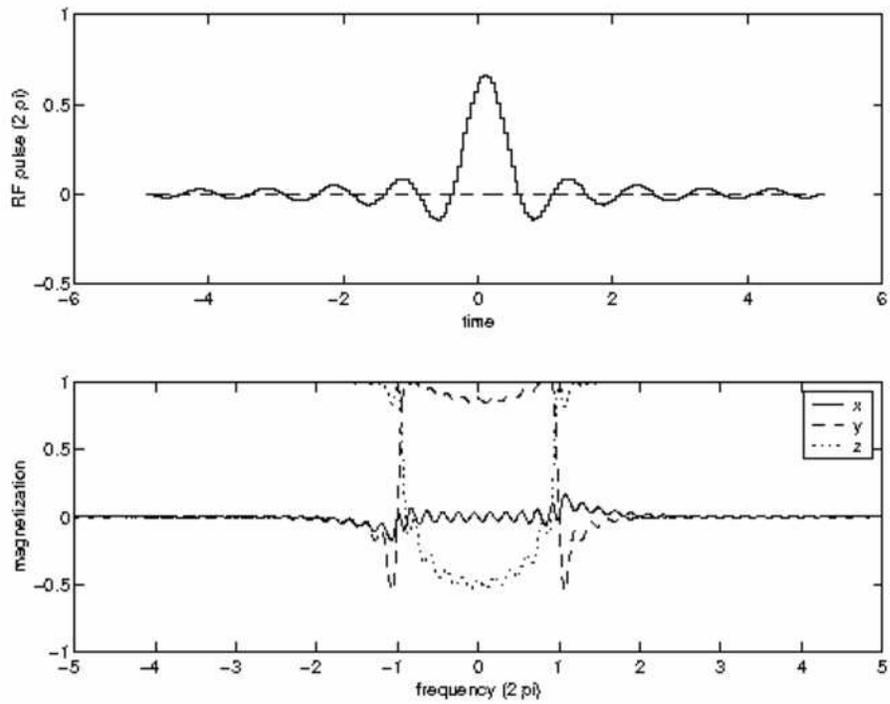}\end{center}
\end{figure}

To obtain a more accurate magnetization profile, one should invert
the selective excitation transform directly rather than use the Fourier
transform approximation. Figures \ref{cap:min120}, \ref{cap:sr120},
and \ref{cap:srb120} show three very different pulses which produce
magnetization profiles which accurately approximate a $120^{\circ}$
profile. The fact that the inversion of the selective excitation transform
is highly non-unique reflects the nonlinearity of the map $\omega\mapsto\TT\omega$. 

\begin{figure}

\caption{\label{cap:min120}A $120^{\circ}$ minimum energy pulse and the
resulting magnetization.}

\begin{center}\includegraphics[%
  scale=0.7]{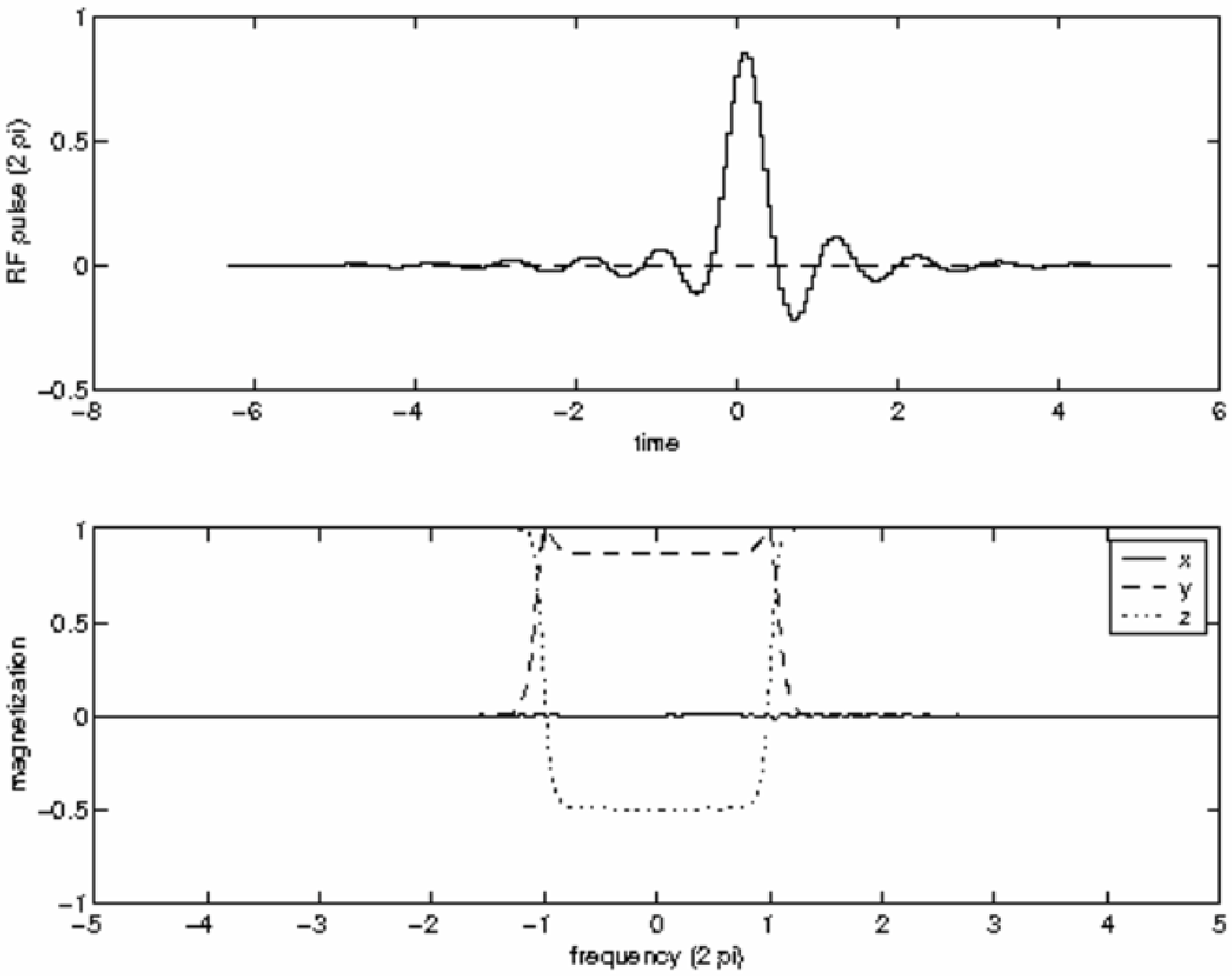}\end{center}
\end{figure}

\begin{figure}

\caption{\label{cap:sr120}A $120^{\circ}$ self refocused pulse.}

\begin{center}\includegraphics[%
  scale=0.7]{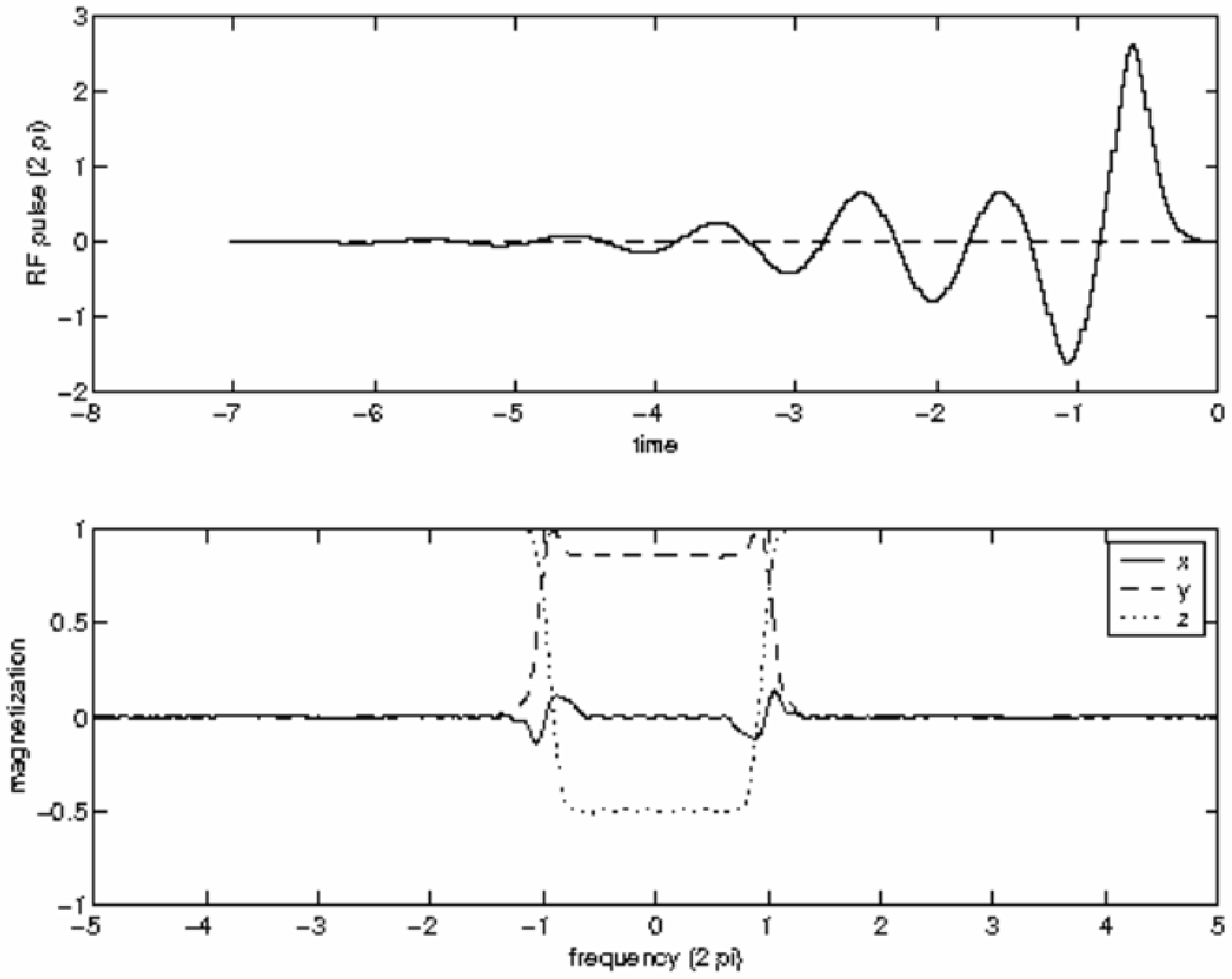}\end{center}
\end{figure}

\begin{figure}

\caption{\label{cap:srb120}A $120^{\circ}$ self refocused pulse.}

\begin{center}\includegraphics[%
  scale=0.7]{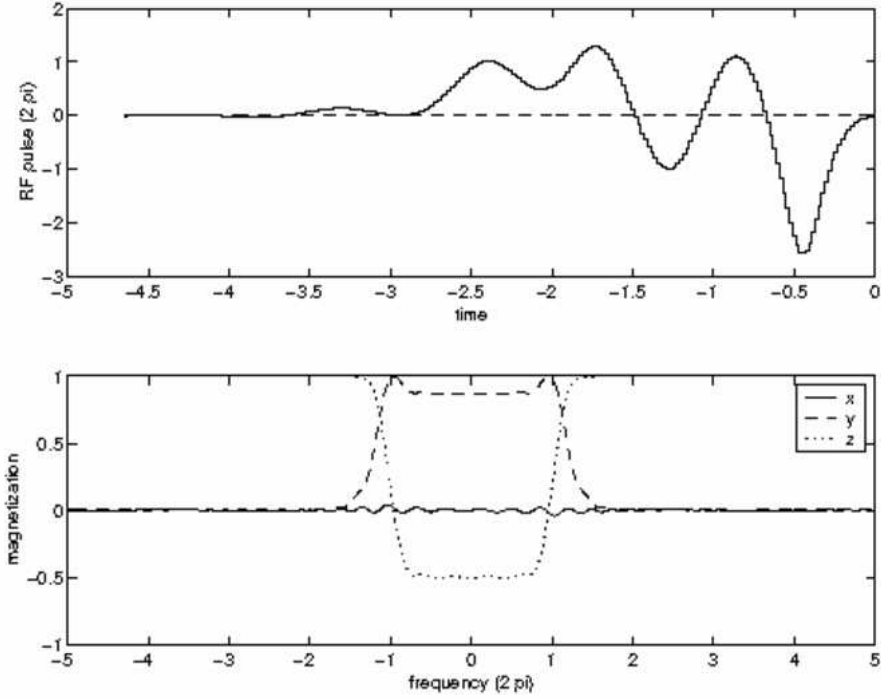}\end{center}
\end{figure}

\section{The discrete selective excitation transform (hard pulse approximation)\label{sec:The-discrete-selective-excitation-transform}}

For practical purposes it is useful to consider pulses of the form
\[
\Omega(t)=\sum_{j=-\infty}^{\infty}\omega_{j}\delta(t-j\Delta).\]
 For example, this is the type of pulse designed by the SLR algorithm
(see \cite{PLNM}). Such a pulse is called a \textit{hard pulse}.
If $\omega_{j}$ vanishes for all $j\geq\rho\in\ZZ$, then $\Omega$
is said to have $\rho$ \textit{rephasing time steps}. In practical
applications, a softened version, e.g., \begin{equation}
\Omega_{\epsilon}(t)=\sum_{j=-\infty}^{\infty}\frac{\omega_{j}}{\epsilon}\chi_{[0,\epsilon)}(t-j\Delta)\label{eq:softened}\end{equation}
 is typically used. The energy of the softened pulse is \begin{equation}
E_{\Omega_{\epsilon}}=\epsilon^{-1}\sum_{j=-\infty}^{\infty}|\omega_{j}|^{2},\label{eq:softened energy}\end{equation}
 which tends to infinity as $\epsilon\rightarrow0$. We think of the
hard pulse as an ideal non-physical pulse corresponding to $\epsilon=0$.
In Appendix \ref{sub:The-error-from-softening-a-pulse}, we explore
the relationship between the hard and softened pulses.

For a pulse of the form (\ref{eq:softened}), it is possible solve
the Bloch equation (\ref{eq:bloch frequency}) explicitly. In the
limit, as $\epsilon$ approaches $0$, the Bloch equation can be replaced
by the recursion \[
M_{-}(z;(j+1)\Delta)=P_{\Delta}R_{\omega_{j}}M_{-}(z;j\Delta),\]
 where $P_{\Delta}$ corresponds to a certain frequency-dependent
rotation around the z-axis, and $R_{\omega_{j}}$ corresponds to a
certain frequency-independent rotation around the $\left[\begin{array}{c}
\re\omega_{j}\\
\im\omega_{j}\\
0\end{array}\right]$-axis. The precise definitions of the operators $P_{\Delta}$ and
$R_{\omega_{j}}$ are given in Section \ref{sec:Discrete-Theory}.

For hard pulses, we evaluate the magnetization $M_{-}$ only at the
discrete time points $j\Delta$ for $j\in\ZZ$. It is easy to see
that, at such time points, $M_{-}$ is a periodic function of the
frequency -- it is a function of $w=e^{i\Delta z}$. We suppress the
dependence on $\Delta$ and $z$ and write $M_{-}=M_{-}(w;j)$ as
a function of $w\in S^{1}$ and $j\in\ZZ$. The resulting magnetization
profile is \[
\Mb(w):=\lim_{j\rightarrow\infty}P_{\Delta}^{-j}M_{-}(w;j),\]
 a function on $S^{1}$. We define the \textit{discrete selective
excitation transform} by \[
\TT^{\textrm{disc}}\Omega:=\Mb=\lim_{j\rightarrow\infty}P_{\Delta}^{-j}M_{-}(\cdot;j).\]
 A condition on $\Omega$ that guarantees the existence of this limit
is given in Section \ref{sec:Forward-discrete-scattering}. The transform
$\TT^{\textrm{disc}}$ maps hard pulses (sequences of complex numbers)
to periodic magnetization profiles ($S^{2}$ valued functions on $S^{1}$).

\section{Main results: continuum theory}

\label{sec:Main results continuum}The problem of inverting the selective
excitation transform has been solved using inverse scattering theory.
We prove in Section \ref{sec:Proof-main-continuum}

\begin{thm}
\label{thm:main continuum}Let $\Mb=\left[\begin{array}{c}
\Mb_{x}\\
\Mb_{y}\\
\Mb_{z}\end{array}\right]:\RR\rightarrow S^{2}\subset\RR^{3}$ be a magnetization profile, and set\[
r(\xi)=\frac{\Mb_{x}(2\xi)+i\Mb_{y}(2\xi)}{1+\Mb_{z}(2\xi)}.\]
Assume that $r$ and $\xi r$ are in $H^{1}(\RR)$. \\
(a) There are infinitely many pulses $\omega$ such that $\TT\omega=\Mb$.
These pulses are parameterized by their bound state data (see below).\\
(b) If $r$ has the form $r(\xi)=e^{-2i\xi\rho}r_{0}(\xi)$ where
$r_{0}$ has a meromorphic extension to the upper half plane with
finitely many poles, and if\[
\lim_{|\xi|\rightarrow\infty}r_{0}(\xi)=0,\]
 then there exists a pulse $\omega$ with rephasing time $\rho$ such
that $\TT\omega=\Mb$.\\
(c) Suppose that $r$ has the form $r(\xi)=e^{-2i\xi\rho}\frac{B(\xi)}{A(\xi)}$,
for $A-1$ and $B$ in $H_{+}(\RR)$ (see Section \ref{sub:proj})
satisfying \[
|A(\xi)|^{2}+|B(\xi)|^{2}=1\textrm{ for all }\xi\in\RR,\]
 and suppose that the Fourier transform of $B$ is supported on the
interval $[0,2T]\subset\RR$. Also assume that $A$ has an analytic
extension to the entire complex plane. Then there exists a pulse $\omega$
with duration $T$ and rephasing time $\rho$ such that $\TT\omega=\Mb$.
\end{thm}
The function $r:\RR\rightarrow\CC$ is called the \textit{reflection
coefficient}. The magnetization profile $\Mb$ is uniquely determined
by the reflection coefficient:\begin{equation}
\Mb(z)=\left[\begin{array}{c}
\frac{2\re r}{1+|r|^{2}}\\
\frac{2\im r}{1+|r|^{2}}\\
\frac{1-|r|^{2}}{1+|r|^{2}}\end{array}\right](\frac{z}{2}).\label{eq:Mr}\end{equation}

The bound state data mentioned in part (a) of Theorem \ref{thm:main continuum}
is defined in Section \ref{sec:Continuum-theory}. In the basic and
generic case, this bound state data takes the form \[
D=(\xi_{1},\xi_{2},\dots,\xi_{m};C_{1}^{\prime},C_{2}^{\prime},\dots,C_{m}^{\prime}),\]
 where $\xi_{1},\dots,\xi_{m}$ are distinct complex numbers in the
upper half plane called \textit{energies}, and $C_{1}^{\prime},\dots,C_{m}^{\prime}$
are non-zero complex numbers called \textit{norming constants}. The
main result is that to each such $r$ and $D$, there corresponds
a unique pulse. The special case of no bound states ($m=0$) corresponds
to a pulse called the \textit{minimum energy pulse} (see \cite{Epstein}).

\section{Main results: discrete theory}

\label{sec:main results discrete}There is a completely analogous
theory for the discrete selective excitation transform. The following
Theorem is proved in Section \ref{sec:Proof-main-discrete}.

\begin{thm}
\label{thm:main discrete}Let $\Mb=\left[\begin{array}{c}
\Mb_{x}\\
\Mb_{y}\\
\Mb_{z}\end{array}\right]:S^{1}\rightarrow S^{2}\subset\RR^{3}$ be a periodic magnetization profile, and set\[
r=\frac{\Mb_{x}+i\Mb_{y}}{1+\Mb_{z}}.\]
Assume that $r$ is in $H^{1}(S^{1})$. \\
(a) There are infinitely many hard pulses $\Omega$ such that $\TT^{\textrm{disc}}\Omega=\Mb$.
These hard pulses are parameterized by their bound state data (see
below). \\
(b) If $r$ has the form $r(w)=w^{-\rho}r_{0}(w)$ where $r_{0}$
has a meromorphic extension to the unit disk which vanishes at the
origin, then there exists a hard pulse $\Omega$ with $\rho$ rephasing
time steps such that $\TT^{\textrm{disc}}\Omega=\Mb$.\\
(c) If $r$ has the form $r(w)=w^{-\rho}\frac{B(w)}{A(w)}$, where
$A$ and $B$ are polynomials of degree $T$ which satisfy\[
|A(w)|^{2}+|B(w)|^{2}=1\textrm{ for all }w\in S^{1},\]
 then there exists a hard pulse $\Omega$ with duration $T$ and $\rho$
rephasing time steps such that $\TT^{\textrm{disc}}\Omega=\Mb$.
\end{thm}
Again, $r:S^{1}\rightarrow\CC$ is called the reflection coefficient,
and the magnetization profile $\Mb$ can be obtained from the reflection
coefficient using\begin{equation}
\Mb(w)=\left[\begin{array}{c}
\frac{2\re r}{1+|r|^{2}}\\
\frac{2\im r}{1+|r|^{2}}\\
\frac{1-|r|^{2}}{1+|r|^{2}}\end{array}\right](w).\label{eq:Mrd}\end{equation}

The bound state data mentioned in part (a) of Theorem \ref{thm:main discrete}
is defined in Section \ref{sec:Discrete-Theory}. In the basic and
generic case, this bound state data takes the form\[
D=(w_{1},w_{2},\dots,w_{m};c_{1}^{\prime},c_{2}^{\prime},\dots,c_{m}^{\prime}),\]
 where $w_{1},\dots,w_{m}$ are distinct complex numbers in the unit
disk called \textit{ene}rgies, and $c_{1}^{\prime},\dots,c_{m}^{\prime}$
are non-zero complex numbers called \textit{norming constants}. The
main result is that to each such $r$ and $D$, there corresponds
a unique hard pulse. The special case of no bound states ($m=0$)
corresponds to a hard pulse called the \textit{minimum energy hard
pulse}. In Section \ref{sec:Applying-the-discrete}, we discuss the
relationship between the continuum bound state data and the discrete
bound state data.

\section{Energy Formulas}

\label{sec:energy formulas}

The following theorem is proved in \cite{Faddeev}. The corollary
is immediate (see Section \ref{sec:Proof-main-discrete}).

\begin{thm}
\label{thm:energy}Let $r$ be a reflection coefficient as in Theorem
\ref{thm:main continuum}, and let $D=(\xi_{1},\xi_{2},\dots,\xi_{m};C_{1}^{\prime},C_{2}^{\prime},\dots,C_{m}^{\prime}),$
be bound state data. Then the energy of the corresponding pulse $\omega:\RR\rightarrow\CC$
is \[
E_{\omega}=\int_{-\infty}^{\infty}|\omega(t)|^{2}dt=\frac{4}{\pi}\int_{-\infty}^{\infty}\log(1+|r(\xi)|^{2})d\xi+16\sum_{k=1}^{m}\im\xi_{k}.\]

\end{thm}
\begin{cor}
\label{cor:energy}Let $r(\xi)=e^{-2i\xi\rho}r_{0}(\xi)$ be a reflection
coefficient as in part (b) of Theorem \ref{thm:main continuum}, and
let $\omega$ be the corresponding pulse with rephasing time $\rho$.
Then the energy of $\omega$ is \begin{equation}
E_{\omega}=\int_{-\infty}^{\infty}|\omega(t)|^{2}dt=\frac{4}{\pi}\int_{-\infty}^{\infty}\log(1+|r(\xi)|^{2})d\xi+16\sum_{k=1}^{m}d_{k}\im\xi_{k},\label{eq:energy cor}\end{equation}
 where $\xi_{1},\dots,\xi_{m}$ are the poles of $r_{0}$ in the upper
half plane, and $d_{1},\dots,d_{m}$ are the corresponding multiplicities.
\end{cor}
The following theorem and its corollary are proved in Section \ref{sub:The-energy-formula-discrete}.

\begin{thm}
\label{thm:energy discrete}Let $r$ be a reflection coefficient as
in Theorem \ref{thm:main discrete}, and let $D=(w_{1},w_{2},\dots,w_{m};c_{1}^{\prime},c_{2}^{\prime},\dots,c_{m}^{\prime})$
be discrete bound state data. Let $\Omega(t)=\sum_{j=-\infty}^{\infty}\omega_{j}\delta(t-j\Delta)$
be the corresponding hard pulse. Then \[
\sum_{j=-\infty}^{\infty}\log(1+\tan^{2}\frac{|\omega_{j}|}{2})=\frac{1}{2\pi}\int_{0}^{2\pi}\log(1+|r(e^{i\theta})|^{2})d\theta-2\sum_{k=1}^{m}\log|w_{k}|.\]
 
\end{thm}
\begin{cor}
\label{cor:energy discrete}Let $r(w)=w^{-\rho}r_{0}(w)$ be a reflection
coefficient as in part (b) of Theorem \ref{thm:main discrete}, and
let $\Omega(t)=\sum_{j=-\infty}^{\rho}\omega_{j}\delta(t-j\Delta)$
be the corresponding pulse with $\rho$ rephasing time steps. Then
\[
\sum_{j=-\infty}^{\infty}\log(1+\tan^{2}\frac{|\omega_{j}|}{2})=\frac{1}{2\pi}\int_{0}^{2\pi}\log(1+|r(e^{i\theta})|^{2})d\theta-2\sum_{k=1}^{m}d_{k}\log|w_{k}|,\]
 where $w_{1},\dots,w_{m}$ are the poles of $r_{0}$ in the unit
disk, and $d_{1},\dots,d_{m}$ are the corresponding multiplicities.
\end{cor}

\section{Algorithms}

\label{sec:Algorithms}There are three main algorithms for producing
hard pulses: the SLR algorithm, the finite rephasing time algorithm,
and the DIST recursion.

\subsection{The SLR algorithm}

The SLR algorithm, discovered independently by M. Shinnar, and his
co-workers, and P. Le Roux, can only be used to design hard pulses
of finite duration, as in part (c) of Theorem \ref{thm:main discrete}.
The input to the algorithm consists of the number of rephasing time
steps, $\rho$, and the two polynomials $A$ and $B$ which must satisfy
\begin{equation}
|A|^{2}+|B|^{2}=1\label{eq:A2B2}\end{equation}
 on the unit circle. The designed magnetization profile is determined
using equation (\ref{eq:Mr}) with the reflection coefficient \[
r(w)=w^{-\rho}\frac{B(w)}{A(w)}.\]
The main drawback of this method is that one does not have direct
control over the magnetization profile (more precisely, the phase
of the reflection coefficient). On the other hand, one has complete
control of the duration, which coincides with the degree of the polynomial
$B$. The SLR technique uses the fact that the magnitude of $B$ is
related to the $z$-component of the resulting magnetization profile
$\Mb=\left[\begin{array}{c}
\Mb_{x}\\
\Mb_{y}\\
\Mb_{z}\end{array}\right]$ by the formula \[
\Mb_{z}=\frac{1-|r|^{2}}{1+|r|^{2}}=|A|^{2}-|B|^{2}=1-2|B|^{2}.\]
 The polynomial $B$ is designed so that $1-2|B|^{2}$ approximates
the ideal $z$-magnetization, and then $A$ is chosen to satisfy equation
(\ref{eq:A2B2}) . The phase of the transverse magnetization $\Mb_{x}+i\Mb_{y}$
is not directly specified by this procedure, and ad hoc methods must
be used to obtain a good approximation of the desired phase. This
disadvantage is discussed further in Section \ref{sub:Comparison-SLR-IST}
where we compare SLR pulses to inverse scattering pulses.

\subsection{The finite rephasing time algorithm}

The finite rephasing time algorithm resembles the SLR algorithm, but
it is more flexible. It produces pulses of infinite duration, and
finite rephasing time. Similar algorithms have been presented by various
authors. See for example \cite{Yagle,Buonocore}. 

The finite rephasing time algorithm can be used to produce the pulse
from part (b) of Theorem \ref{thm:main discrete}. The input to this
algorithm is the number of rephasing time steps, $\rho$, and a rational
function $r_{0}=\frac{P}{Q}$, where $P$ and $Q$ are polynomials
(possibly of very high degree). These polynomials do not need to satisfy
any kind of equation like (\ref{eq:A2B2}). The designed magnetization
profile is determined using equation (\ref{eq:Mr}) with the reflection
coefficient \[
r(w)=w^{-\rho}\frac{P(w)}{Q(w)}.\]
With this algorithm, one has direct control of the entire magnetization
profile (not just the $z$-component as in the SLR algorithm), but
control on the pulse duration is sacrificed. In fact, the designed
pulse almost always has technically infinite duration. However, we
will see in Chapter \ref{cha:Pulses-with-finite} that for many common
applications, the \textit{effective} duration of pulses designed with
this method is within a reasonable range. The loss of direct control
on the duration causes no disadvantage in practice.

This algorithm is more general than the SLR algorithm, and equally
efficient. In Chapter \ref{cha:Pulses-with-finite} we give a simple
derivation of the algorithm, and we describe several applications
in NMR pulse design.

\subsection{The DIST recursion}

The discrete inverse scattering transform (DIST) algorithm, which
we introduce in this thesis, efficiently handles the most general
case of hard pulse design. It is motivated by the inverse scattering
method of pulse design (see \cite{Epstein}). The input to the algorithm
is called the \textit{scattering data} and consists of:

(a) An arbitrary reflection coefficient, $r:S^{1}\rightarrow\CC$;

(b) Arbitrary bound state data (see Section \ref{sec:Discrete-Theory}). 

\noindent For technical reasons, we assume that $r$ is in $H^{1}(S^{1})$.

The derivation of this algorithm, and the proof that the output pulse
has the correct scattering data is the main result of this thesis.
This derivation and proof can be found in Chapter \ref{cha:Scattering-theory}.

\chapter{Scattering theory\label{cha:Scattering-theory}}

\section{Preliminaries}

\subsection{The projection operators $\Pi_{+}$ and $\Pi_{-}$\label{sub:proj}}

The Paley-Wiener theorem states that every function $f\in L^{2}(\RR)$
can be written uniquely in the form \[
f=f_{+}+f_{-},\]
 where $f_{+}\in L^{2}(\RR)$ has an analytic extension to the upper
half plane, and $f_{-}\in L^{2}(\RR)$ has an analytic extension to
the lower half plane satisfying \[
\lim_{|\xi|\rightarrow\infty}f_{\pm}(\xi)=0.\]
We define the projection operators $\Pi_{+},\Pi_{-}:L^{2}(\RR)\rightarrow L^{2}(\RR)$
by \begin{eqnarray*}
(\Pi_{+}f)(\xi) & = & f_{+}(\xi)=\frac{1}{2\pi}\int_{0}^{\infty}\hat{f}(y)e^{i\xi y}dy\\
(\Pi_{-}f)(\xi) & = & f_{-}(\xi)=\frac{1}{2\pi}\int_{-\infty}^{0}\hat{f}(y)e^{i\xi y}dy.\end{eqnarray*}
Let $H_{+}(\RR)$ and $H_{-}(\RR)$ denote the ranges of $\Pi_{+}$
and $\Pi_{-}$, respectively. 

Similarly, every function $f\in L^{2}(S^{1})$ can be written uniquely
in the form \[
f=f_{0}+f_{+}+f_{-},\]
 where $f_{0}\in\CC$ is a constant, $f_{+}\in L^{2}(S^{1})$ has
an analytic extension to the unit disk, vanishing at the origin, and
$f_{-}$ has an analytic extension to $\CCh\setminus\DD$, vanishing
at $\infty$. We define the operators $\Pi_{+},\Pi_{-},\tilde{\Pi}_{+},\tilde{\Pi}_{-}:L^{2}(S^{1})\rightarrow L^{2}(S^{1})$
by \begin{eqnarray*}
(\Pi_{+}f)(w) & = & f_{+}(w)=\sum_{j=1}^{\infty}\hat{f}(j)w^{j}\\
(\Pi_{-}f)(w) & = & f_{-}(w)=\sum_{j=-\infty}^{-1}\hat{f}(j)w^{j}\\
\tilde{\Pi}_{+}f & = & f_{+}+\frac{1}{2}\hat{f}(0)\\
\tilde{\Pi}_{-}f & = & f_{-}+\frac{1}{2}\hat{f}(0).\end{eqnarray*}
Let $H_{+}$, $H_{-}$, $\tilde{H}_{+}$, and $\tilde{H}_{-}$ denote
the ranges of $\Pi_{+}$, $\Pi_{-}$, $\tilde{\Pi}_{+}$, and $\tilde{\Pi}_{-}$,
respectively. 

The following lemma will be needed later in the chapter.

\begin{lem}
\label{fac:proj}Suppose that $f$ is in $L^{2}(\RR)$ and that $\xi f$
is in $L^{2}(\RR)$. Then \begin{equation}
i\xi\Pi_{+}f-\Pi_{+}i\xi f=-\frac{1}{2\pi}\hat{f}(0),\label{eq:proj 1}\end{equation}
 and \begin{equation}
i\xi\Pi_{-}f-\Pi_{-}i\xi f=\frac{1}{2\pi}\hat{f}(0).\label{eq:proj 2}\end{equation}

\end{lem}
\begin{proof}
First note that the hypothesis implies that $\hat{f}(0)$ is well
defined. We will assume sufficient regularity and decay for $f$ so
that the below integrals make sense. The general result then follows
by continuity.

Integrating by parts gives\[
\frac{1}{2\pi}\int_{0}^{\infty}\hat{f}^{\prime}(t)e^{it\xi}dt+\frac{i\xi}{2\pi}\int_{0}^{\infty}\hat{f}(t)e^{it\xi}dt=-\frac{1}{2\pi}\hat{f}(0).\]
Notice that the first term on the left is $-\Pi_{+}i\xi f$ and the
second term is $i\xi\Pi_{+}f$ . This proves (\ref{eq:proj 1}). A
similar computation can be used to prove (\ref{eq:proj 2}).
\end{proof}

\subsection{The reflection coefficient}

Later in the chapter we will be working with functions of the form
$r=\frac{b}{a}$ where $a$ and $b$ satisfy $|a|^{2}+|b|^{2}=1$
on $\RR$. In this section we mention some results which are well
known in the inverse scattering literature.

\begin{prop}
\label{pro:refl CC}Let $r\in H^{1}(\RR)$, let $\xi_{1},\dots,\xi_{m}\in\HH$,
and let $d_{1},\dots,d_{m}$ be positive integers. Then there exist
unique continuous functions $a,b:\RR\rightarrow\CC$ such that\\
(i) $r=\frac{b}{a}$ on $\RR$;\\
(ii) $|a(\xi)|^{2}+|b(\xi)|^{2}=1$ for all $\xi\in\RR$;\\
(iii) $a$ has an analytic extension to the upper half plane with
zeros at $\xi_{1},\dots,\xi_{m}$ of orders $d_{1},\dots,d_{m}$;\\
(iv) $\lim_{|\xi|\rightarrow\infty}a(\xi)=1$.
\end{prop}
\begin{proof}
We need only consider the case where $a$ is non-vanishing ($m=0$),
because adding zeros simply amounts to multiplying $a$ and $b$ by
a common Blaschke product. Therefore, $\log a$ is an analytic function
which, by condition (iv), must tend to $0$ at $\infty$. We know
that \[
\re\log a=\log|a|=-\frac{1}{2}\log(1+|r|^{2}).\]
Notice that $\log(1+|r|^{2})$ is in $H^{1}(\RR)$. Therefore, we
must have\[
\log a=-\Pi_{+}(1+|r|^{2}).\]
The general solution is \[
a=\prod_{k=1}^{m}\left(\frac{\xi-\xi_{k}}{\xi-\xi_{k}^{*}}\right)^{d_{k}}\cdot\exp(-\Pi_{+}(1+|r|^{2}))\]
 and \[
b=ra.\]

\end{proof}
\begin{prop}
\label{pro:refl DC}Let $r\in H^{1}(S^{1})$, let $w_{1},\dots,w_{m}\in\DD$,
and let $d_{1},\dots,d_{m}$ be positive integers. Then there exist
unique continuous functions $a,b:S^{1}\rightarrow\CC$ such that\\
(i) $r=\frac{b}{a}$ on $S^{1}$;\\
(ii) $|a|^{2}+|b|^{2}=1$ on $S^{1}$;\\
(iii) $a$ has an analytic extension to the unit disk with zeros at
$w_{1},\dots,w_{m}$ of orders $d_{1},\dots,d_{m}$;\\
(iv) $a(0)>0$.
\end{prop}
\begin{proof}
The proof is essentially the same as the proof of Proposition \ref{pro:refl CC}.
The general solution is \[
a=\prod_{k=1}^{m}\left(\frac{w_{k}^{*}}{|w_{k}|}\frac{w_{k}-w}{1-w_{k}^{*}w}\right)^{d_{k}}\cdot\exp(-\tilde{\Pi}_{+}(1+|r|^{2}))\]
 and \[
b=ra.\]

\end{proof}

\subsection{Banach derivatives\label{sub:Banach-derivatives}}

In this section we define the derivative of a curve in a Banach space,
and prove some results which will be needed in Section \ref{sec:Proof-main-continuum}.

Let $X$ be a Banach space. A function $u:\RR\rightarrow X$ is said
to be \textit{differentiable} at $t_{0}\in\RR$ if there exists $u_{t_{0}}^{\prime}\in X$
such that \[
\lim_{t\rightarrow t_{0}}\left\Vert \frac{u_{t}-u_{t_{0}}}{t-t_{0}}-u_{t_{0}}^{'}\right\Vert =0.\]
 In this case $u_{t_{0}}^{\prime}$ is called the \textit{derivative}
of $u$ at $t_{0}$. This derivative is necessarily unique.

Let $X$ and $Y$ be Banach spaces, and let $\LL(X,Y)$ denote the
space of bounded linear maps from $X$ to $Y$. Since $\LL(X,Y)$
is a Banach space (with the operator norm), we can also speak of the
derivative of $A:\RR\rightarrow\LL(X,Y)$.

The following two lemmas can be proved directly by applying the above
definition of derivative.

\begin{lem}
\label{lem:bd 1}Let $A_{t}:\RR\rightarrow\LL(X,Y)$ and $u_{t}:\RR\rightarrow X$
. If $A_{t}$ and $u_{t}$ are differentiable at $t_{0}$, then so
is $v_{t}=A_{t}u_{t}:\RR\rightarrow Y$, and \[
v_{t_{0}}^{\prime}=A_{t_{0}}^{\prime}u_{t_{0}}+A_{t_{0}}u_{t_{0}}^{\prime}.\]

\end{lem}
$\;$

\begin{lem}
\label{lem:bd 2}Suppose $A_{t}:\RR\rightarrow\LL(X,Y)$ is invertible
near $t_{0}\in\RR$, and assume that both $A_{t}$ and $A_{t}^{-1}$
are uniformly bounded in a neighborhood of $t_{0}$. If $A_{t}$ is
differentiable at $t_{0}$, then so is $A_{t}^{-1}$, and \[
(A^{-1})_{t_{0}}^{\prime}=-A_{t_{0}}^{-1}A_{t_{0}}^{\prime}A_{t_{0}}^{-1}.\]

\end{lem}
The following propositions will be used in Section \ref{sub:The-Marchenko-Equation}.

\begin{prop}
\label{pro:bd 3}Let $X$ be a complex Hilbert space. Suppose that
$A_{t}:\RR\rightarrow\LL(X,X)$ and $v_{t}:\RR\rightarrow X$ are
differentiable at $t_{0}$. If $A_{t}$ has the form \[
A_{t}=1+K_{t},\]
 where $K_{t}$ is a bounded, self adjoint, positive operator in a
neighborhood of $t_{0}$, then $A_{t}$ is invertible in a neighborhood
of $t_{0}$, and the function $u_{t}:\RR\rightarrow X$ given by \[
u_{t}=A_{t}^{-1}v_{t}\]
 is also differentiable at $t_{0}$.
\end{prop}
\begin{proof}
Since $A$ is differentiable at $t_{0}$, it is certainly uniformly
bounded near $t_{0}$. Therefore, by Proposition \ref{pro:1+A invertible},
and its proof, $A_{t}$ is invertible, and $A_{t}^{-1}$ is uniformly
bounded in a neighborhood of $t_{0}$. The result now follows from
Lemmas \ref{lem:bd 1} and \ref{lem:bd 2}.
\end{proof}
\begin{prop}
\label{pro:bd 4}Let $(X,\left\Vert \cdot\right\Vert )$ be a Banach
space of complex-valued functions on some set $S$, and suppose that
there is a constant $C>0$ such that $|f(p)|\leq C\left\Vert f\right\Vert $
for all $f\in X$ and $p\in S$. If $u:\RR\rightarrow X$ is differentiable
at $t_{0}$, then for each point $p\in S$, the function $t\mapsto u_{t}(p)$
is differentiable at $t_{0}$, and its derivative is given by $u_{t_{0}}^{\prime}(p)$.
\end{prop}
\begin{proof}
For simplicity, assume $t_{0}=0$. We want to show that \[
\lim_{t\rightarrow0}\left|\frac{u_{t}(p)-u_{0}(p)}{t}-u_{0}^{'}(p)\right|=0.\]
But this follows immediately from the definition of derivative, and
the hypothesis that $|f(p)|\leq C\left\Vert f\right\Vert $ for all
$f\in X$ and $p\in S$.
\end{proof}
\begin{prop}
\label{pro:bd 5}Let $r\in H^{1}(\RR)$, and suppose that $\xi r$
is in $H^{1}(\RR)$. Then the curve $t\mapsto re^{i\xi t}\in H^{1}(\RR)$
is differentiable at every $t_{0}\in\RR$, and its derivative at $t_{0}$
is given by $i\xi re^{i\xi t_{0}}.$
\end{prop}
\begin{proof}
Without loss of generality, we can assume that $t_{0}=0$. We need
to show that \[
\lim_{t\rightarrow0}\left\Vert \frac{e^{i\xi t}r-r}{t}-i\xi r\right\Vert _{H^{1}}=\lim_{t\rightarrow0}\left\Vert \left(\frac{e^{i\xi t}-1}{\xi t}-i\right)\xi r\right\Vert _{H^{1}}=0.\]
Set $\phi_{t}(\xi)=\left(\frac{e^{i\xi t}-1}{\xi t}-i\right)$. One
can show that there is a constant $M>0$ such that for all $0<t<1$
and for all $\xi\in\RR$, we have $|\phi_{t}(\xi)|<M$ and $|\phi_{t}^{\prime}(\xi)|<M$.
Given $\epsilon>0$, we can choose $\xi_{\epsilon}>0$ large enough
so that \[
\left\Vert \chi_{(-\infty,-\xi_{\epsilon}]\cup[\xi_{\epsilon},\infty)}\xi r\right\Vert _{L^{2}},\left\Vert \chi_{(-\infty,-\xi_{\epsilon}]\cup[\xi_{\epsilon},\infty)}\partial_{\xi}(\xi r)\right\Vert _{L^{2}}<\epsilon,\]
and $0<t<1$ small enough so that\[
\left\Vert \chi_{[-\xi_{\epsilon},\xi_{\epsilon}]}\phi_{t}\right\Vert _{L^{1}},\left\Vert \chi_{[-\xi_{\epsilon},\xi_{\epsilon}]}\phi_{t}^{\prime}\right\Vert _{L^{1}}<\epsilon.\]
 We have \begin{eqnarray}
\left\Vert \phi_{t}\xi r\right\Vert _{H^{1}}^{2} & = & \left\Vert \phi_{t}\xi r\right\Vert _{L^{2}}^{2}+\left\Vert \phi_{t}^{\prime}\xi r+\phi_{t}\partial_{\xi}(\xi r)\right\Vert _{L^{2}}^{2}\nonumber \\
 & \leq & \left\Vert \phi_{t}\xi r\right\Vert _{L^{2}}^{2}+2\left\Vert \phi_{t}^{\prime}\xi r\right\Vert _{L^{2}}^{2}+2\left\Vert \phi_{t}\partial_{\xi}(\xi r)\right\Vert _{L^{2}}^{2}.\label{eq:bd 5 1}\end{eqnarray}
We estimate each of the three terms in (\ref{eq:bd 5 1}):\begin{eqnarray*}
\left\Vert \phi_{t}\xi r\right\Vert _{L^{2}}^{2} & = & \left\Vert \chi_{(-\infty,-\xi_{\epsilon}]\cup[\xi_{\epsilon},\infty)}\phi_{t}\xi r\right\Vert _{L^{2}}^{2}+\left\Vert \chi_{[-\xi_{\epsilon},\xi_{\epsilon}]}\phi_{t}\xi r\right\Vert _{L^{2}}^{2}\\
 & \leq & M^{2}\epsilon^{2}+\epsilon^{2}\left\Vert \xi r\right\Vert _{L^{2}}^{2}\end{eqnarray*}
\begin{eqnarray*}
\left\Vert \phi_{t}^{\prime}\xi r\right\Vert _{L^{2}}^{2} & = & \left\Vert \chi_{(-\infty,-\xi_{\epsilon}]\cup[\xi_{\epsilon},\infty)}\phi_{t}^{\prime}\xi r\right\Vert _{L^{2}}^{2}+\left\Vert \chi_{[-\xi_{\epsilon},\xi_{\epsilon}]}\phi_{t}^{\prime}\xi r\right\Vert _{L^{2}}^{2}\\
 & \leq & M^{2}\epsilon^{2}+\epsilon^{2}\left\Vert \xi r\right\Vert _{L^{2}}^{2}\end{eqnarray*}
\begin{eqnarray*}
\left\Vert \phi_{t}\partial_{\xi}(\xi r)\right\Vert _{L^{2}}^{2} & = & \left\Vert \chi_{(-\infty,-\xi_{\epsilon}]\cup[\xi_{\epsilon},\infty)}\phi_{t}\partial_{\xi}(\xi r)\right\Vert _{L^{2}}^{2}+\left\Vert \chi_{[-\xi_{\epsilon},\xi_{\epsilon}]}\phi_{t}\partial_{\xi}(\xi r)\right\Vert _{L^{2}}^{2}\\
 & \leq & M^{2}\epsilon^{2}+\epsilon^{2}\left\Vert \partial_{\xi}(\xi r)\right\Vert _{L^{2}}^{2}.\end{eqnarray*}
Therefore we have $\left\Vert \phi_{t}\xi r\right\Vert _{H^{1}}^{2}\leq5(M^{2}+\left\Vert \xi r\right\Vert _{H^{1}}^{2})\epsilon^{2}$. 
\end{proof}

\subsection{The Marchenko equation\label{sub:The-Marchenko-Equation}}

Later in the chapter we work with the Marchenko equation:\begin{equation}
(1+\Pi_{+}r^{*}\Pi_{-}r)L=-\Pi_{+}r^{*},\label{eq:Mar}\end{equation}
 where $r$ is a complex function on $\Lambda=\RR\textrm{ or }S^{1}$.
The Marchenko equation comes from a system of equations:\begin{eqnarray*}
K & = & 1+\Pi_{+}r^{*}L^{*}\\
L & = & -\Pi_{+}r^{*}K^{*}.\end{eqnarray*}
In the literature $\Lambda$ is $\RR$, and the system is typically
written in the Fourier domain: \begin{eqnarray*}
\FF(K-1)(x) & = & \int_{0}^{\infty}f(-x-y)^{*}\hat{L}(y)^{*}dy\\
\hat{L}(x) & = & -\int_{0}^{\infty}f(-x-y)^{*}\hat{K}(y)^{*}dy,\end{eqnarray*}
where $f=\hat{r}$. In this section, we discuss hypotheses on $r$
which guarantee that there is a unique solution $L\in H_{+}^{1}(\Lambda)$
to equation (\ref{eq:Mar}). 

\begin{lem}
\label{lem:TME 1}If $r\in H^{1}(\Lambda)$, then $A:=\Pi_{+}r^{*}\Pi_{-}r\Pi_{+}$
is a bounded, positive, self adjoint operator from $H^{1}(\Lambda)$
to itself.
\end{lem}
\begin{proof}
By Fact \ref{fac:Sob 4}, multiplication by $r$ or $r^{*}$is a bounded
operator from $H^{1}(\Lambda)$ to itself. It is clear that $\Pi_{-}$
and $\Pi_{+}$ are also bounded operators on $H^{1}(\Lambda)$. Therefore
$A$ is itself such an operator. We can use Fact \ref{fac:adjoint pair}
to show that $A$ is positive and self adjoint once we establish that
$(\Pi_{+}r^{*}\Pi_{-},\Pi_{-}r\Pi_{+})$ forms an adjoint pair. This
can easily be shown by proving that $(\Pi_{-},\Pi_{+})$ and $(r,r^{*})$
are each adjoint pairs.
\end{proof}
\begin{prop}
\label{pro:TME 2}If $r\in H^{1}(\Lambda)$, then there is a unique
solution $L\in H_{+}^{1}(\Lambda)$ to the Marchenko equation (\ref{eq:Mar}).
The norm of this solution satisfies the estimate\[
\left\Vert L\right\Vert _{H^{1}}\leq\left\Vert r\right\Vert _{H^{1}}.\]

\end{prop}
\begin{proof}
The Marchenko equation can be written \[
(1+A)L=-\Pi_{+}r^{*}\]
 for $A$ defined in Lemma \ref{lem:TME 1}. The Lemma tells us that
$A$ is a bounded, positive, self adjoint operator from $H^{1}(\Lambda)$
to itself. The desired result then follows immediately from Proposition
\ref{pro:1+A invertible}.
\end{proof}
\begin{prop}
Let $r_{0}\in H^{1}(\RR)$, and let $A:\RR\rightarrow\LL(H^{1}(\RR),H^{1}(\RR))$
be given by \[
A_{t}=\Pi_{+}r_{0}^{*}e^{-2i\xi t}\Pi_{-}r_{0}e^{2i\xi t}.\]
If $\xi r_{0}$ is in $H^{1}(\RR)$, then the solution $L:\RR\rightarrow H^{1}(\RR)$
to the Marchenko equation\[
(1+A_{t})L_{t}=-\Pi_{+}r_{0}^{*}e^{-2i\xi t}\]
 is differentiable (in the sense of Section \ref{sub:Banach-derivatives})
at every $t_{0}\in\RR$.
\end{prop}
\begin{proof}
By Proposition \ref{pro:bd 3}, we just need to show that $A$ is
differentiable at $t_{0}$. It is sufficient to show that multiplication
by $r_{0}e^{2i\xi t}$ is differentiable at $t_{0}.$ This can easily
be shown using Proposition \ref{pro:bd 5}.
\end{proof}

\section{The Zakharov-Shabat system\label{sec:The-Zakharov-Shabat-system}}

The selective excitation transform has been shown to be equivalent
to the scattering transform for the Zakharov-Shabat (ZS) system of
equations. In this section we give the details of this relationship.
We mainly follow the notation of \cite{Epstein}.

The Bloch equation (\ref{eq:bloch frequency}) can be written as \[
\partial_{t}M(z;t)=\left[\begin{array}{ccc}
0 & z & -\im\omega(t)\\
-z & 0 & \re\omega(t)\\
\im\omega(t) & -\re\omega(t) & 0\end{array}\right]M(z;t).\]
We can think of the $3\times3$ matrix in this equation as an element
of the Lie algebra of $SO_{3}\RR$ mapping $M\in S^{2}$ to an element
of the tangent space $T_{M}S^{2}$. If we lift to the universal cover,
$SU_{2}\CC$, this equation becomes \begin{equation}
\partial_{t}\left[\begin{array}{c}
\psi_{1}(\xi;t)\\
\psi_{2}(\xi;t)\end{array}\right]=\left[\begin{array}{cc}
-i\xi & q(t)\\
-q^{*}(t) & i\xi\end{array}\right]\left[\begin{array}{c}
\psi_{1}(\xi;t)\\
\psi_{2}(\xi;t)\end{array}\right]\label{eq:ZS1}\end{equation}
 where \begin{equation}
M(z;t)=\left[\begin{array}{c}
2\re\psi_{1}^{*}\psi_{2}\\
2\im\psi_{1}^{*}\psi_{2}\\
|\psi_{1}|^{2}-|\psi_{2}|^{2}\end{array}\right](\frac{z}{2};t)\label{eq:ZS2}\end{equation}
 and \begin{equation}
q(t)=-\frac{i}{2}\omega^{*}(t).\label{eq:ZS3}\end{equation}
 The function $q$ is called the \textit{potential} for the ZS-system.

\section{Continuum theory\label{sec:Continuum-theory}}

In this section we outline the scattering theory for the ZS-system.
Many of the formulas can be found in \cite{Epstein}. 

Let $q:\RR\rightarrow\CC$ be an integrable potential for the ZS-system.
Then there exist solutions $\psi_{1\pm}=\left[\begin{array}{c}
\psi_{11\pm}\\
\psi_{12\pm}\end{array}\right]$ to the differential equation \begin{equation}
\partial_{t}\psi_{1\pm}(\xi;t)=\left[\begin{array}{cc}
-i\xi & q(t)\\
-q^{*}(t) & i\xi\end{array}\right]\psi_{1\pm}(\xi;t)\label{eq:ct1}\end{equation}
 satisfying \begin{equation}
\lim_{t\rightarrow\pm\infty}e^{i\xi t}\psi_{1\pm}(\xi;t)=\left[\begin{array}{c}
1\\
0\end{array}\right]\;\;\;\textrm{for all }\xi\in\RR.\label{eq:ct2}\end{equation}
The matrix \begin{equation}
\left[\begin{array}{cc}
a(\xi) & -b^{*}(\xi)\\
b(\xi) & a^{*}(\xi)\end{array}\right]=\left[\begin{array}{cc}
\psi_{11+}^{*}(\xi;t) & \psi_{12+}^{*}(\xi;t)\\
-\psi_{12+}(\xi;t) & \psi_{11+}(\xi;t)\end{array}\right]\left[\begin{array}{cc}
\psi_{11-}(\xi;t) & -\psi_{12-}^{*}(\xi;t)\\
\psi_{12-}(\xi;t) & \psi_{11-}^{*}(\xi;t)\end{array}\right]\label{eq:ct3}\end{equation}
 is independent of $t$, and is called the \textit{scattering matrix}.
Let us define $A_{\pm}$ and $B_{\pm}$ by \begin{equation}
\left[\begin{array}{c}
A_{\pm,t}(\xi)e^{-i\xi t}\\
B_{\pm,t}(\xi)e^{-i\xi t}\end{array}\right]=\psi_{1\pm}(\xi;t),\label{eq:ct4}\end{equation}
 so we have\begin{equation}
\lim_{t\rightarrow\pm\infty}\left[\begin{array}{c}
A_{\pm,t}(\xi)\\
B_{\pm,t}(\xi)\end{array}\right]=\left[\begin{array}{c}
1\\
0\end{array}\right]\;\;\;\textrm{for all }\xi\in\RR\label{eq:ct5}\end{equation}
 and \begin{equation}
\left[\begin{array}{cc}
a(\xi) & -b^{*}(\xi)e^{-2i\xi t}\\
b(\xi)e^{2i\xi t} & a^{*}(\xi)\end{array}\right]=\left[\begin{array}{cc}
A_{+,t}^{*}(\xi) & B_{+,t}^{*}(\xi)\\
-B_{+,t}(\xi) & A_{+,t}(\xi)\end{array}\right]\left[\begin{array}{cc}
A_{-,t}(\xi) & -B_{-,t}^{*}(\xi)\\
B_{-,t}(\xi) & A_{-,t}^{*}(\xi)\end{array}\right]\label{eq:ct6}\end{equation}
 for all $t$ and $\xi$. One can show that the resulting magnetization
profile, $\Mb=\TT\omega$, from Section \ref{sec:The-selective-excitation-transform}
is given by equation (\ref{eq:Mr}) for the reflection coefficient
\[
r(\xi)=\frac{b(\xi)}{a(\xi)}.\]
 The following is an outline of the main elements of the scattering
theory (see \cite{Epstein,Faddeev}). The Marchenko equations below
do not appear in their typical forms. The derivations of these Marchenko
equations are given in Sections \ref{sec:Derivation-right-continuum}
and \ref{sec:Derivation-left-continuum}. 

\begin{itemize}
\item The functions $a$ and $b$ satisfy \begin{equation}
|a(\xi)|^{2}+|b(\xi)|^{2}=1\;\;\;\textrm{for }\xi\in\RR\label{eq:6.5}\end{equation}
 and \begin{equation}
\lim_{|\xi|\rightarrow\infty}\left[\begin{array}{c}
a(\xi)\\
b(\xi)\end{array}\right]=\left[\begin{array}{c}
1\\
0\end{array}\right].\label{eq:cc6.7}\end{equation}

\item For every $t,\xi\in\RR$, we have \begin{equation}
|A_{\pm,t}(\xi)|^{2}+|B_{\pm,t}(\xi)|^{2}=1.\label{eq:ct7}\end{equation}

\item For each $t$, the functions $A_{+,t}^{*}$, $B_{+,t}^{*}$, $A_{-,t}$,
and $B_{-,t}$ have analytic extensions to the upper half plane. 
\item The function $a=A_{+,t}^{*}A_{-,t}+B_{+,t}^{*}B_{-,t}$ has an analytic
extension to the upper half plane. We assume that $a$ has finitely
many zeros $\left\{ \xi_{1},\dots,\xi_{m}\right\} $ in the upper
half plane, which are all simple. For each zero $\xi_{k}$, of $a$,
there is a constant $C_{k}^{\prime}$ such that \begin{equation}
\left[\begin{array}{c}
A_{-,t}(\xi_{k})\\
B_{-,t}(\xi_{k})\end{array}\right]=C_{k}^{\prime}\left[\begin{array}{c}
-B_{+,t}^{*}(\xi_{k})e^{2i\xi_{k}t}\\
A_{+,t}^{*}(\xi_{k})e^{2i\xi_{k}t}\end{array}\right]\textrm{ for all }t\in\RR.\label{eq:ct8}\end{equation}
Set\begin{equation}
C_{k}=\frac{C_{k}^{\prime}}{a^{\prime}(\xi_{k})}\label{eq:ct9}\end{equation}
and \begin{equation}
\tilde{C}_{k}=\frac{(C_{k}^{\prime})^{-1}}{a^{\prime}(\xi_{k})}=\frac{-1}{C_{k}\left[a^{\prime}(\xi_{k})\right]^{2}}.\label{eq:ct10}\end{equation}
 
\item The data $(a,b;\xi_{1},\dots,\xi_{m};C_{1}^{\prime},\dots,C_{m}^{\prime})$
is called the \textit{scattering data} for the potential $q$. 
\item The data $(r;\xi_{1},\dots,\xi_{m};C_{1},\dots,C_{m})$ is called
the \textit{reduced scattering data} for the potential $q$. The functions
$a$ and $b$ can be determined from the reduced scattering data by
the formulas \begin{eqnarray*}
a & = & \prod_{k=1}^{m}\left(\frac{\xi-\xi_{k}}{\xi-\xi_{k}^{*}}\right)\cdot\exp(-\Pi_{+}(1+|r|^{2}))\\
b & = & ra.\end{eqnarray*}

\item The function $B_{+,t}^{*}$ can be determined from the reduced scattering
data. It is the unique solution (see Proposition \ref{pro:TME 2})
to the Marchenko equation: \begin{equation}
(1+\Pi_{+}r_{t}^{*}\Pi_{-}r_{t})B_{+,t}^{*}=-\Pi_{+}r_{t}^{*}\label{eq:ct11}\end{equation}
where \begin{equation}
r_{t}(\xi)=\Pi_{-}re^{2i\xi t}-\sum_{k=1}^{m}\frac{C_{k}e^{2i\xi_{k}t}}{\xi-\xi_{k}}.\label{eq:ct12}\end{equation}

\item The function $B_{-,t}$ can be determined from the \textit{left reduced
scattering data} \[
\tilde{S}=(s;\xi_{1},\dots,\xi_{m};\tilde{C}_{1},\dots,\tilde{C}_{m}),\]
where \begin{equation}
s(\xi)=-\frac{b^{*}(\xi)}{a(\xi)}.\label{eq:ct13}\end{equation}
 It is the unique solution (see Proposition \ref{pro:TME 2}) to the
left Marchenko equation: \begin{equation}
(1+\Pi_{+}s_{t}^{*}\Pi_{-}s_{t})B_{-,t}=-\Pi_{+}s_{t}^{*}\label{eq:ct14}\end{equation}
where \begin{equation}
s_{t}(\xi)=\Pi_{-}s(\xi)e^{-2i\xi t}-\sum_{k=1}^{m}\frac{\tilde{C}_{k}e^{-2i\xi_{k}t}}{\xi-\xi_{k}}.\label{eq:ct15}\end{equation}

\item The potential $q$ can be recovered using \begin{equation}
q(t)=\frac{1}{\pi}\FF(B_{+,t}^{*})(0^{+})\label{eq:ct16}\end{equation}
 or \begin{equation}
-q^{*}(t)=\frac{1}{\pi}\FF(B_{-,t})(0^{+}).\label{eq:ct17}\end{equation}

\end{itemize}
The following is a restatement of part (a) of Theorem \ref{thm:main continuum}
in terms of the ZS-system framework. The proof is given in Section
\ref{sec:Proof-main-continuum}.

\begin{thm}
Let $S=(a,b;\xi_{1},\dots,\xi_{m};C_{1}^{\prime},\dots,C_{m}^{\prime})$
be arbitrary scattering data, as above, such that $r=\frac{b}{a}$
and $\xi r$ are both in $H^{1}(\RR)$. Then there is a well defined
potential $q$ for the ZS-system such that $S$ is the corresponding
scattering data. This potential can be found either by using equations
(\ref{eq:ct11}), (\ref{eq:ct12}) and (\ref{eq:ct16}), or by using
equations (\ref{eq:ct14}), (\ref{eq:ct15}) and (\ref{eq:ct17}).
\end{thm}
\begin{rem}
Typically the right Marchenko equation (\ref{eq:ct11}) is used for
the positive values of $t$, and the left Marchenko equation (\ref{eq:ct14})
is used for the negative values of $t$.
\end{rem}

\section{Discrete Theory\label{sec:Discrete-Theory}}

In this section we describe an analogous scattering theory for hard
pulses.

Consider a potential of the form\[
q(t)=\sum_{j=-\infty}^{\infty}\mu_{k}\delta(t-j\delta),\]
 such that \[
\sum_{j=-\infty}^{\infty}|\mu_{j}|<\infty.\]
We will call such a function a \textit{discrete potential}. For these
potentials, the differential equation (\ref{eq:ct1}) is replaced
by a recursion: \begin{equation}
\psi_{1\pm}(\xi;(j+1)\Delta)=\left[\begin{array}{cc}
e^{-i\Delta\xi} & 0\\
0 & e^{i\Delta\xi}\end{array}\right]\left[\begin{array}{cc}
\cos\left|\mu_{j}\right| & \frac{\mu_{j}}{|\mu_{j}|}\sin\left|\mu_{j}\right|\\
-\frac{\mu_{j}^{*}}{|\mu_{j}|}\sin\left|\mu_{j}\right| & \cos\left|\mu_{j}\right|\end{array}\right]\psi_{1\pm}(\xi;j\Delta).\label{eq:dt1}\end{equation}
 For each integer $j$, $\psi_{1-}(\xi;j\Delta)$ and $\psi_{2-}(\xi;j\Delta)$
are periodic functions of $w^{\frac{1}{2}}=e^{i\xi\Delta}$. Let us
set \begin{equation}
\Psi_{\pm,j}(w)=\left[\begin{array}{c}
\Psi_{1+,j}(w)\\
\Psi_{2+,j}(w)\end{array}\right]=\psi_{1\pm}(\xi;j\Delta).\label{eq:dt2}\end{equation}
 Then the recursion is\begin{equation}
\Psi_{\pm,j+1}(w)=\left[\begin{array}{cc}
w^{-\frac{1}{2}} & 0\\
0 & w^{\frac{1}{2}}\end{array}\right]\left[\begin{array}{cc}
\cos\left|\mu_{j}\right| & \frac{\mu_{j}}{|\mu_{j}|}\sin\left|\mu_{j}\right|\\
-\frac{\mu_{j}^{*}}{|\mu_{j}|}\sin\left|\mu_{j}\right| & \cos\left|\mu_{j}\right|\end{array}\right]\Psi_{\pm,j}(w),\label{eq:dt3}\end{equation}
and the scattering matrix is\begin{equation}
\left[\begin{array}{cc}
a & -b^{*}\\
b & a^{*}\end{array}\right]=\left[\begin{array}{cc}
\Psi_{1+,j}^{*} & \Psi_{2+,j}^{*}\\
-\Psi_{2+,j} & \Psi_{1+,j}\end{array}\right]\left[\begin{array}{cc}
\Psi_{1-,j} & -\Psi_{2-,j}^{*}\\
\Psi_{2-,j} & \Psi_{1-,j}^{*}\end{array}\right].\label{eq:dt4}\end{equation}
 Let us define \begin{equation}
\left[\begin{array}{c}
A_{\pm,j}(w)w^{-\frac{j}{2}}\\
B_{\pm,j}(w)w^{\frac{j}{2}}\end{array}\right]=\Psi_{\pm,j}(w),\label{eq:dt5}\end{equation}
 so we have\begin{equation}
\lim_{j\rightarrow\pm\infty}\left[\begin{array}{c}
A_{\pm,j}(w)\\
B_{\pm,j}(w)\end{array}\right]=\left[\begin{array}{c}
1\\
0\end{array}\right]\;\;\;\textrm{for all }w\in S^{1}\label{eq:dt6}\end{equation}
 and \begin{equation}
\left[\begin{array}{cc}
a & -b^{*}w^{-j}\\
bw^{j} & a^{*}\end{array}\right]=\left[\begin{array}{cc}
A_{+,j}^{*} & B_{+,j}^{*}\\
-B_{+,j} & A_{+,j}\end{array}\right]\left[\begin{array}{cc}
A_{-,j} & -B_{-,j}^{*}\\
B_{-,j} & A_{-,j}^{*}\end{array}\right]\label{eq:dt7}\end{equation}
 for all $j\in\ZZ$. One can show that the resulting magnetization
profile from Section \ref{sec:The-discrete-selective-excitation-transform}
is given by equation (\ref{eq:Mrd}) for the reflection coefficient
\[
r(w)=\frac{b(w)}{a(w)}.\]
The following is an outline of the main elements of the discrete scattering
theory. Some of these statements are proved in Section \ref{sec:Forward-discrete-scattering}.
The discrete Marchenko equations are derived in Section \ref{sec:Derivation-discrete-marchenko}.

\begin{itemize}
\item The functions $a$ and $b$ are in $L^{2}(S^{1})$ and satisfy \begin{equation}
|a(w)|^{2}+|b(w)|^{2}=1\;\;\;\textrm{for }w\in S^{1}\label{eq:dt9}\end{equation}
 and \begin{equation}
\hat{a}(0)>0.\label{eq:dt10}\end{equation}

\item The functions $A_{\pm}$ and $B_{\pm}$ are in $L^{2}(S^{1})$ and
satisfy\begin{equation}
|A_{\pm,j}(w)|^{2}+|B_{\pm,j}(w)|^{2}=1\textrm{ for }w\in S^{1}\label{eq:dt11}\end{equation}
 and \begin{equation}
\hat{A}_{\pm,j}(0)>0.\label{eq:dt12}\end{equation}

\item For each $j$, the functions $A_{+,t}^{*}$, $B_{+,t}^{*}$, $A_{-,t}$,
and $w^{-1}B_{-,t}$ have analytic extensions to the unit $w$-disk
$\DD$. 
\item The function $a=A_{+,j}^{*}A_{-,j}+B_{+,j}^{*}B_{-,j}$ has an analytic
extension to the unit disk. We assume that $a$ has finitely many
zeros $\left\{ w_{1},\dots,w_{m}\right\} $ in the unit disk, which
are all simple. For each zero $w_{k}$ of $a$, there is a constant
$c_{k}^{\prime}$ such that \begin{equation}
\left[\begin{array}{c}
A_{-,j}(w_{k})\\
B_{-,j}(w_{k})\end{array}\right]=c_{k}^{\prime}\left[\begin{array}{c}
-B_{+,j}^{*}(w_{k})w_{k}^{j}\\
A_{+,j}^{*}(w_{k})w_{k}^{j}\end{array}\right]\textrm{ for all }j\in\ZZ.\label{eq:dt13}\end{equation}
Set\begin{equation}
c_{k}=\frac{c_{k}^{\prime}}{a^{\prime}(w_{k})}\label{eq:dt14}\end{equation}
and \begin{equation}
\tilde{c}_{k}=\frac{-(c_{k}^{\prime})^{-1}w_{k}^{-1}}{a^{\prime}(w_{k})}=\frac{-w_{k}^{-1}}{c_{k}\left[a^{\prime}(w_{k})\right]^{2}}.\label{eq:dt15}\end{equation}
 
\item The data $(a,b;w_{1},\dots,w_{m};c_{1}^{\prime},\dots,c_{m}^{\prime})$
is called the \textit{discrete scattering data} for the potential
$q$.
\item The data $(r;w_{1},\dots,w_{m};c_{1},\dots,c_{m})$ is called the
\textit{reduced discrete scattering data} for the potential $q$.
The functions $a$ and $b$ can be determined from the reduced scattering
data by the formulas \begin{eqnarray}
a & = & \prod_{k=1}^{m}\left(\frac{w_{k}^{*}}{|w_{k}|}\frac{w_{k}-w}{1-w_{k}^{*}w}\right)\cdot\exp(-\tilde{\Pi}_{+}(1+|r|^{2}))\label{eq:dt16}\\
b & = & ra.\nonumber \end{eqnarray}

\item The function $\frac{wB_{+,j}^{*}}{\hat{A}_{+,j}(0)}$ can be determined
from the reduced scattering data. It is the unique solution (see Proposition
\ref{pro:TME 2}) to the Marchenko equation: \begin{equation}
(1+\Pi_{+}r_{j}^{*}\Pi_{-}r_{j})\frac{wB_{+,j}^{*}}{\hat{A}_{+,j}(0)}=-\Pi_{+}r_{j}^{*}\label{eq:dt17}\end{equation}
where \begin{equation}
r_{j}=\Pi_{-}rw^{j-1}-\sum_{k=1}^{m}\frac{c_{k}w_{k}^{j}}{w-w_{k}}.\label{eq:dt18}\end{equation}

\item The function $\frac{B_{-,j}}{\hat{A}_{-,t}(0)}$ can be determined
from the \textit{left reduced scattering data} $(s;w_{1},\dots,w_{n};\tilde{c}_{1},\dots,\tilde{c}_{N})$,
where \begin{equation}
s=-\frac{b^{*}}{a}.\label{eq:dt19}\end{equation}
 It is the unique solution (see Proposition \ref{pro:TME 2}) to the
left Marchenko equation: \begin{equation}
(1+\Pi_{+}s_{j}^{*}\Pi_{-}s_{j})\frac{B_{-,j}}{\hat{A}_{-,j}(0)}=-s_{j}^{*}\label{eq:dt20}\end{equation}
where \begin{equation}
s_{j}=\Pi_{-}sw^{-j}-\sum_{k=1}^{m}\frac{\tilde{c}_{k}w_{k}^{-j}}{w-w_{k}}.\label{eq:dt21}\end{equation}

\item The potential $q(t)=\sum_{j=-\infty}^{\infty}\mu_{j}\delta(t-j\Delta)$
can be recovered using\begin{equation}
\mu_{j}=\frac{\gamma_{j}}{|\gamma_{j}|}\arctan|\gamma_{j}|\label{eq:dt22}\end{equation}
 for\begin{equation}
\gamma_{j}=\FF(\frac{wB_{+,j}^{*}}{\hat{A}_{+,j}(0)})(1)\label{eq:dt23}\end{equation}
or\begin{equation}
-\gamma_{j}^{*}=\FF(\frac{B_{-,j+1}}{\hat{A}_{-,j+1}(0)})(1).\label{eq:dt24}\end{equation}

\item The functions $A_{+,j}$ and $B_{+,j}$ can also be approximately
computed recursively. We start by setting \begin{equation}
\left[\begin{array}{c}
A_{+,M}(w)\\
B_{+,M}(w)\end{array}\right]=\left[\begin{array}{c}
1\\
0\end{array}\right]\label{eq:dt25}\end{equation}
 for $M>>0$, and then use the recursion \begin{equation}
\left[\begin{array}{c}
A_{+,j}\\
B_{+,j}\end{array}\right]=(1+|\gamma_{j}|^{2})^{-\frac{1}{2}}\left[\begin{array}{cc}
1 & -\gamma_{j}w^{-1}\\
\gamma_{j}^{*} & w^{-1}\end{array}\right]\left[\begin{array}{c}
A_{+,j+1}\\
B_{+,j+1}\end{array}\right]\label{eq:dt26}\end{equation}
 for \begin{equation}
-\gamma_{j}^{*}=\frac{\FF(wr_{j}A_{+,j+1}^{*})(0)}{\FF(A_{+,j+1}^{*})(0)-\FF(wr_{j}(wB_{+,j+1}^{*}))(0)}.\label{eq:dt27}\end{equation}

\item Similarly, the functions $A_{-,j}$ and $B_{-,j}$ can be approximately
computed recursively. We start by setting \begin{equation}
\left[\begin{array}{c}
A_{-,-M}(w)\\
B_{-,-M}(w)\end{array}\right]=\left[\begin{array}{c}
1\\
0\end{array}\right]\label{eq:dt28}\end{equation}
 for $M>>0$, and then use the recursion\begin{equation}
\left[\begin{array}{c}
A_{-,j+1}\\
B_{-,j+1}\end{array}\right]=(1+|\gamma_{j}|^{2})^{-\frac{1}{2}}\left[\begin{array}{cc}
1 & \gamma_{j}\\
-\gamma_{j}^{*}w & w\end{array}\right]\left[\begin{array}{c}
A_{-,j}\\
B_{-,j}\end{array}\right]\label{eq:dt29}\end{equation}
for \begin{equation}
\gamma_{j}=\frac{\FF(ws_{j+1}A_{-,j})(0)}{\FF(A_{-,j})(0)-\FF(ws_{j+1}B_{-,j})(0)}.\label{eq:dt30}\end{equation}

\end{itemize}
The recursions (\ref{eq:dt26}) and (\ref{eq:dt29}) along with equations
(\ref{eq:dt27}) and (\ref{eq:dt30}) form the discrete inverse scattering
transform (DIST) algorithm. These equations are derived in Section
\ref{sec:Derivation-of-the-DIST}.

The following is a restatement of part (a) of Theorem \ref{thm:main discrete}
in terms of the ZS-system framework. The proof is given in Section
\ref{sec:Proof-main-discrete}.

\begin{thm}
Let $S=(a,b;w_{1},\dots,w_{m};c_{1},\dots,c_{m})$ be arbitrary discrete
scattering data, as above, such that $r=\frac{b}{a}$ is in $H^{1}(S^{1})$.
Then there is a well defined discrete potential $q(t)=\sum_{j=-\infty}^{\infty}\mu_{j}\delta(t-j\Delta)$
for the ZS-system such that $S$ is the corresponding discrete scattering
data. This potential can be found either by using equations (\ref{eq:dt17}),
(\ref{eq:dt18}) and (\ref{eq:dt23}), or by using equations (\ref{eq:dt20}),
(\ref{eq:dt21}) and (\ref{eq:dt24}).
\end{thm}

\section{Forward discrete scattering\label{sec:Forward-discrete-scattering}}

In this section we prove the analyticity properties of $A_{\pm}$
and $B_{\pm}$ from Section \ref{sec:Discrete-Theory}.

\begin{prop}
\label{pro:forward}Let $\gamma:\ZZ\rightarrow\CC$ be a sequence
of complex numbers such that \[
\sum_{j=-\infty}^{\infty}|\gamma_{j}|<\infty.\]
 Then there are unique solutions $A_{\pm}$ and $B_{\pm}$ to equations
(\ref{eq:dt26}), (\ref{eq:dt29}) and (\ref{eq:dt6}). Furthermore,
for each integer $j$, the functions $A_{-,j}$, $w^{-1}B_{-,j}$,
$A_{+,j}^{*}$ and $B_{+,j}^{*}$ are all in $\tilde{H}_{+}(S^{1})$. 
\end{prop}
\begin{proof}
Let $A_{0}$ and $B_{0}$ be the solutions to the recursion\begin{equation}
\small{\left[\begin{array}{cc}
A_{0,j+1} & -w^{j+1}B_{0,j+1}^{*}\\
w^{-j-1}B_{0,j+1} & A_{0,j+1}^{*}\end{array}\right]=(1+|\gamma_{j}|^{2})^{-\frac{1}{2}}\left[\begin{array}{cc}
1 & \gamma_{j}w^{j}\\
-\gamma_{j}^{*}w^{-j} & 1\end{array}\right]\left[\begin{array}{cc}
A_{0,j} & -w^{j}B_{0,j}^{*}\\
w^{-j}B_{0,j} & A_{0,j}^{*}\end{array}\right]}\label{eq:dfs1}\end{equation}
 normalized by \[
\left[\begin{array}{c}
A_{0,0}\\
B_{0,0}\end{array}\right]=\left[\begin{array}{c}
1\\
0\end{array}\right].\]
Notice that equation (\ref{eq:dfs1}) is equivalent to (\ref{eq:dt29}).
Clearly we have $|A_{0,j}|^{2}+|B_{0,j}|^{2}=1$ on $S^{1}$ for all
integers $j$. Therefore, we can estimate \[
\left|A_{0,j+1}-A_{0,j}\right|\leq\left|\gamma_{j}\right|+\left|1-\frac{1}{\sqrt{1+|\gamma_{j}|^{2}}}\right|\]
 and \[
\left|w^{-j-1}B_{0,j+1}-w^{-j}B_{0,j}\right|\leq\left|\gamma_{j}\right|+\left|1-\frac{1}{\sqrt{1+|\gamma_{j}|^{2}}}\right|,\]
which implies that the sequences $A_{0,0},A_{0,1},A_{0,2},\dots$
and $B_{0.0},w^{-1}B_{0,1},w^{-2}B_{0,2},\dots$ converge in $L^{\infty}(S^{1})$
to some functions $a_{0}$ and $b_{0}$, respectively, in $L^{\infty}(S^{1})\subset L^{2}(S^{1})$.
One can show inductively that $A_{0,j}$ and $w^{j}B_{0,j}^{*}$ are
in $\tilde{H}_{+}(S^{1})$ for all $j\geq0$. Thus, $a_{0}$ and $b_{0}^{*}$
also must be in $\tilde{H}_{+}(S^{1})$. By multiplying the matrix
recursion on the right by $\left[\begin{array}{cc}
a_{0}^{*} & b_{0}^{*}\\
-b_{0} & a_{0}\end{array}\right]$, we see that $A_{+.j}$ and $B_{+,j}$ must be equal to \begin{eqnarray*}
A_{+,j} & = & a_{0}^{*}A_{0,j}+b_{0}w^{j}B_{0,j}^{*}\\
B_{+,j} & = & a_{0}^{*}w^{-j}B_{0,j}-b_{0}A_{0,j}^{*},\end{eqnarray*}
 which implies that $A_{+,0}^{*}=a_{0}$ and $B_{+,0}^{*}=-b_{0}^{*}$
are in $\tilde{H}_{+}(S^{1})$, as desired. By similar reasoning,
$A_{+,j}^{*}$ and $B_{+,j}^{*}$ must be in $\tilde{H}_{+}(S^{1})$,
for all $j$. 
\end{proof}

\section{Derivation of the right Marchenko equation\label{sec:Derivation-right-continuum}}

In the next two sections we derive the right and left Marchenko equations
for the ZS-system. We assume that $q$ is an integrable potential
with scattering data \[
S=(a,b;\xi_{1},\dots,\xi_{m};C_{1}^{\prime},\dots,C_{m}^{\prime}).\]
 Let $A_{\pm}$ and $B_{\pm}$ be as in Section \ref{sec:Continuum-theory}.
Recall that $A_{+,t}^{*}$, $B_{+,t}^{*}$, $A_{-,t}$, and $B_{-,t}$
all have analytic extensions to the upper half $\xi$-plane. Rearranging
equation (\ref{eq:ct6}) gives

\begin{equation}
\left[\begin{array}{cc}
A_{-,t}(\xi) & -B_{-,t}^{*}(\xi)\\
B_{-,t}(\xi) & A_{-,t}^{*}(\xi)\end{array}\right]=\left[\begin{array}{cc}
A_{+,t}(\xi) & -B_{+,t}^{*}(\xi)\\
B_{+,t}(\xi) & A_{+,t}^{*}(\xi)\end{array}\right]\left[\begin{array}{cc}
a(\xi) & -b^{*}(\xi)e^{-2i\xi t}\\
b(\xi)e^{2i\xi t} & a^{*}(\xi)\end{array}\right]\label{eq:drm 0}\end{equation}
 or\begin{eqnarray*}
\frac{1}{a}A_{-,t} & = & A_{+,t}-re^{2i\xi t}B_{+,t}^{*}\\
\frac{1}{a}B_{-,t} & = & B_{+,t}+re^{2i\xi t}A_{+,t}^{*}.\end{eqnarray*}
We apply $\Pi_{-}$, conjugate the second equation, and rearrange:\begin{eqnarray}
A_{+,t} & = & 1+\Pi_{-}re^{2i\xi t}B_{+,t}^{*}+\Pi_{-}\frac{1}{a}A_{-,t}\label{eq:drm1}\\
B_{+,t}^{*} & = & -\Pi_{+}r^{*}e^{-2i\xi t}A_{+,t}+\Pi_{+}\frac{1}{a^{*}}B_{-,t}^{*}.\label{eq:drm2}\end{eqnarray}
Using the properties of $C_{1}^{\prime},\dots,C_{k}^{\prime}$ we
get\begin{eqnarray*}
\Pi_{-}\frac{1}{a}A_{-,t} & = & \sum_{k=1}^{N}\frac{A_{-,t}(\xi_{k})}{a^{\prime}(\xi_{k})}\cdot\frac{1}{\xi-\xi_{k}}\\
 & = & -\sum_{k=1}^{N}\frac{C_{k}^{\prime}e^{2i\xi_{k}t}B_{+,t}^{*}(\xi_{k})}{a^{\prime}(\xi_{k})}\cdot\frac{1}{\xi-\xi_{k}}\\
 & = & -\sum_{k=1}^{N}C_{k}e^{2i\xi_{k}t}B_{+,t}^{*}(\xi_{k})\cdot\frac{1}{\xi-\xi_{k}}\\
 & = & -\Pi_{-}Q_{+,t}B_{+,t}^{*}\end{eqnarray*}
and \begin{eqnarray*}
\Pi_{-}\frac{1}{a}B_{-,t} & = & \sum_{k=1}^{N}\frac{B_{-,t}(\xi_{k})}{a^{\prime}(\xi_{k})}\cdot\frac{1}{\xi-\xi_{k}}\\
 & = & \sum_{k=1}^{N}\frac{C_{k}^{\prime}e^{2i\xi_{k}t}A_{+,t}^{*}(\xi_{k})}{a^{\prime}(\xi_{k})}\cdot\frac{1}{\xi-\xi_{k}}\\
 & = & \sum_{k=1}^{N}C_{k}e^{2i\xi_{k}t}A_{+,t}^{*}(\xi_{k})\cdot\frac{1}{\xi-\xi_{k}}\\
 & = & \Pi_{-}Q_{+,t}A_{+,t}^{*}\end{eqnarray*}
where \[
Q_{+,t}(\xi)=\sum_{k=1}^{m}\frac{C_{k}e^{2i\xi_{k}t}}{\xi-\xi_{k}}.\]
Therefore (\ref{eq:drm1}) and (\ref{eq:drm2}) become \begin{eqnarray}
A_{+,t} & = & 1+\Pi_{-}(re^{2i\xi t}-Q_{+,t})B_{+,t}^{*}\label{eq:drm3p}\\
B_{+,t}^{*} & = & -\Pi_{+}(re^{2i\xi t}-Q_{+,t})^{*}A_{+,t}\label{eq:drm4p}\end{eqnarray}
or\begin{eqnarray}
A_{+,t} & = & 1+\Pi_{-}r_{t}B_{+,t}^{*}\label{eq:drm3}\\
B_{+,t}^{*} & = & -\Pi_{+}r_{t}^{*}A_{+,t}\label{eq:drm4}\end{eqnarray}
 where \begin{eqnarray*}
r_{t}(\xi) & = & \Pi_{-}re^{2i\xi t}-Q_{+,t}(\xi)\\
 & = & \Pi_{-}re^{2i\xi t}-\sum_{k=1}^{m}\frac{C_{k}e^{2i\xi_{k}t}}{\xi-\xi_{k}}.\end{eqnarray*}
Equations (\ref{eq:drm3}) and (\ref{eq:drm4}) can be combined into
the single Marchenko equation\begin{equation}
(1+\Pi_{+}r_{t}^{*}\Pi_{-}r_{t})B_{+,t}^{*}=-\Pi_{+}r_{t}^{*}.\label{eq:drm mar}\end{equation}

To prove equation (\ref{eq:ct16}) we rewrite the ZS-system in terms
of $A_{\pm}$ and $B_{\pm}$:\[
\frac{d}{dt}\left[\begin{array}{c}
A_{\pm,t}(\xi)\\
B_{\pm,t}(\xi)\end{array}\right]=\left[\begin{array}{cc}
0 & q(t)\\
-q^{*}(t) & 2i\xi\end{array}\right]\left[\begin{array}{c}
A_{\pm,t}(\xi)\\
B_{\pm,t}(\xi)\end{array}\right].\]
Taking the $t$-derivative of both sides of equation (\ref{eq:drm4})
gives:\begin{eqnarray*}
-q(t)A_{+,t}^{*}-2i\xi B_{+,t}^{*} & = & -\Pi_{+}(\frac{d}{dt}r_{t})^{*}A_{+,t}-q(t)\Pi_{+}r_{t}^{*}B_{+,t}\\
 & = & \Pi_{+}2i\xi r_{t}^{*}A_{+,t}-q(t)A_{+,t}^{*}+q(t).\end{eqnarray*}
Here we used the fact that \[
\Pi_{-}\frac{d}{dt}r_{t}=\Pi_{-}2i\xi r_{t}.\]
We then use equation (\ref{eq:proj 1}) from Lemma \ref{fac:proj}
to solve for the potential:\begin{eqnarray*}
q(t) & = & 2i\xi\Pi_{+}r_{t}^{*}A_{+,t}-\Pi_{+}2i\xi r_{t}^{*}A_{+,t}\\
 & = & -\frac{1}{\pi}\FF(r_{t}^{*}A_{+,t})(0)\\
 & = & \frac{1}{\pi}\FF(B_{+,t}^{*})(0^{+}).\end{eqnarray*}

\section{Derivation of the left Marchenko equation\label{sec:Derivation-left-continuum}}

A similar method can be used to derive the left Marchenko equation.
Instead of equation (\ref{eq:drm 0}), we use \[
\left[\begin{array}{cc}
A_{+,t}^{*}(\xi) & B_{+,t}^{*}(\xi)\\
-B_{+,t}(\xi) & A_{+,t}(\xi)\end{array}\right]=\left[\begin{array}{cc}
a(\xi) & -b^{*}(\xi)e^{-2i\xi t}\\
b(\xi)e^{2i\xi t} & a^{*}(\xi)\end{array}\right]\left[\begin{array}{cc}
A_{-,t}^{*}(\xi) & B_{-,t}^{*}(\xi)\\
-B_{-,t}(\xi) & A_{-,t}(\xi)\end{array}\right]\]
or

\begin{eqnarray*}
\frac{1}{a}A_{+,t}^{*} & = & A_{-,t}^{*}+\frac{b^{*}}{a}e^{-2i\xi t}B_{-,t}\\
\frac{1}{a}B_{+,t}^{*} & = & B_{-,t}^{*}-\frac{b^{*}}{a}e^{-2i\xi t}A_{-,t}.\end{eqnarray*}
Again, we apply $\Pi_{-}$ and conjugate the second equation:\begin{eqnarray}
A_{-,t}^{*} & = & 1+\Pi_{-}se^{-2i\xi t}B_{-,t}+\Pi_{-}\frac{1}{a}A_{+,t}^{*}\label{eq:dlm1}\\
B_{-,t} & = & -\Pi_{+}s^{*}e^{2i\xi t}A_{-,t}^{*}+\Pi_{+}\frac{1}{a^{*}}B_{+,t}.\label{eq:dlm2}\end{eqnarray}
This time, we have\begin{eqnarray*}
\Pi_{-}\frac{1}{a}A_{+,t}^{*} & = & \sum_{k=1}^{N}\frac{A_{+,t}^{*}(\xi_{k})}{a^{\prime}(\xi_{k})}\cdot\frac{1}{\xi-\xi_{k}}\\
 & = & \sum_{k=1}^{N}\frac{(C_{k}^{\prime})^{-1}e^{-2i\xi_{k}t}B_{-,t}(\xi_{k})}{a^{\prime}(\xi_{k})}\cdot\frac{1}{\xi-\xi_{k}}\\
 & = & \sum_{k=1}^{N}\tilde{C}_{k}e^{-2i\xi_{k}t}B_{-,t}(\xi_{k})\cdot\frac{1}{\xi-\xi_{k}}\\
 & = & -\Pi_{-}Q_{-,t}B_{-,t}\end{eqnarray*}
and \begin{eqnarray*}
\Pi_{-}\frac{1}{a}B_{+,t}^{*} & = & \sum_{k=1}^{N}\frac{B_{+,t}^{*}(\xi_{k})}{a^{\prime}(\xi_{k})}\cdot\frac{1}{\xi-\xi_{k}}\\
 & = & -\sum_{k=1}^{N}\frac{(C_{k}^{\prime})^{-1}e^{-2i\xi_{k}t}A_{-,t}(\xi_{k})}{a^{\prime}(\xi_{k})}\cdot\frac{1}{\xi-\xi_{k}}\\
 & = & -\sum_{k=1}^{N}\tilde{C}_{k}e^{-2i\xi_{k}t}A_{-,t}(\xi_{k})\cdot\frac{1}{\xi-\xi_{k}}\\
 & = & \Pi_{-}Q_{-,t}A_{-,t}\end{eqnarray*}
where \[
Q_{-,t}(\xi)=-\sum_{k=1}^{m}\frac{\tilde{C}_{k}e^{-2i\xi_{k}t}}{\xi-\xi_{k}}.\]
Therefore (\ref{eq:dlm1}) and (\ref{eq:dlm2}) become\begin{eqnarray*}
A_{-,t}^{*} & = & 1+\Pi_{-}(se^{-2i\xi t}-Q_{-,t})B_{-,t}\\
B_{-,t} & = & -\Pi_{+}(se^{-2i\xi t}-Q_{-,t})^{*}A_{-,t}^{*}\end{eqnarray*}
or\begin{eqnarray}
A_{-,t}^{*} & = & 1+\Pi_{-}s_{t}B_{-,t}\label{eq:dlm3}\\
B_{-,t} & = & -\Pi_{+}s_{t}^{*}A_{-,t}^{*}\label{eq:dlm4}\end{eqnarray}
 where \begin{eqnarray*}
s_{t}(\xi) & = & \Pi_{-}\frac{-b^{*}}{a}e^{-2i\xi t}-Q_{-,t}(\xi)\\
 & = & \Pi_{-}\frac{-b^{*}}{a}e^{-2i\xi t}+\sum_{k=1}^{m}\frac{\tilde{C}_{k}e^{-2i\xi_{k}t}}{\xi-\xi_{k}}.\end{eqnarray*}
Equations (\ref{eq:dlm3}) and (\ref{eq:dlm4}) can be combined into
the single Marchenko equation:\[
(1+\Pi_{+}s_{t}^{*}\Pi_{-}s_{t})B_{-,t}=-\Pi_{+}s_{t}^{*}.\]

The proof of equation (\ref{eq:ct17}) is identical to the above proof
of equation (\ref{eq:ct16}):

\begin{eqnarray*}
-q^{*}(t)A_{-,t}+2i\xi B_{-,t} & = & -\Pi_{+}(\frac{d}{dt}d_{t})^{*}A_{-,t}^{*}-q^{*}(t)\Pi_{+}s_{t}^{*}B_{-,t}^{*}\\
 & = & -\Pi_{+}2i\xi s_{t}^{*}A_{-,t}^{*}-q^{*}(t)A_{-,t}+q^{*}(t)\end{eqnarray*}
 \begin{eqnarray*}
q^{*}(t) & = & -2i\xi\Pi_{+}s_{t}^{*}A_{-,t}^{*}+\Pi_{+}2i\xi s_{t}^{*}A_{-,t}^{*}\\
 & = & \frac{1}{\pi}\FF(s_{t}^{*}A_{-,t}^{*})(0)\\
 & = & -\frac{1}{\pi}\FF(B_{-,t})(0^{+}).\end{eqnarray*}

\section{Derivation of the discrete Marchenko equations\label{sec:Derivation-discrete-marchenko}}

In this section we derive the right Marchenko equations for hard pulses.
We omit the derivation of the left equation, but the reader should
be able to reproduce it using the techniques from this section and
the previous two sections. 

Assume that $q$ has the form \[
q(t)=\sum_{j=-\infty}^{\infty}\mu_{j}\delta(t-j\Delta),\]
 where \[
\sum_{j=-\infty}^{\infty}|\mu_{j}|<\infty,\]
 and let $S=(a,b;w_{1},\dots,w_{m};c_{1}^{\prime},\dots,c_{m}^{\prime})$
be the corresponding discrete scattering data. Let $A_{\pm}$ and
$B_{\pm}$ be as in Section \ref{sec:Discrete-Theory}. Recall that
$A_{+,j}^{*}$, $B_{+,j}^{*}$, $A_{-,j}$, and $w^{-1}B_{-,j}$ all
have analytic extensions to unit $w$-disk. Rearranging equation (\ref{eq:dt7})
gives

\[
\left[\begin{array}{cc}
A_{-,j} & -B_{-,j}^{*}\\
B_{-,j} & A_{-,j}^{*}\end{array}\right]=\left[\begin{array}{cc}
A_{+,j} & -B_{+,j}^{*}\\
B_{+,j} & A_{+,j}^{*}\end{array}\right]\left[\begin{array}{cc}
a & -b^{*}w^{-j}\\
bw^{j} & a^{*}\end{array}\right]\]
 or

\begin{eqnarray*}
\frac{1}{a}A_{-,j} & = & A_{+,j}-rw^{j-1}(wB_{+,j}^{*})\\
\frac{1}{a}w^{-1}B_{-,j} & = & w^{-1}B_{+,j}+rw^{j-1}A_{+,j}^{*}.\end{eqnarray*}
We apply $\Pi_{-}$ and conjugate the second equation:\begin{eqnarray}
A_{+,j} & = & \hat{A}_{+,j}(0)+\Pi_{-}rw^{j-1}(wB_{+,j}^{*})+\Pi_{-}\frac{1}{a}A_{-,j}\label{eq:drmd1}\\
wB_{+,j}^{*} & = & -\Pi_{+}(rw^{j-1})^{*}A_{+,j}+\Pi_{+}\frac{1}{a^{*}}wB_{-,j}^{*}.\label{eq:drmd2}\end{eqnarray}
Using the properties of $c_{1}^{\prime},\dots,c_{m}^{\prime}$, we
get\begin{eqnarray*}
\Pi_{-}\frac{1}{a}A_{-,j} & = & \sum_{k=1}^{N}\frac{A_{-,j}(w_{k})}{a^{\prime}(w_{k})}\cdot\frac{1}{w-w_{k}}\\
 & = & -\sum_{k=1}^{N}\frac{c_{k}^{\prime}w_{k}^{j}B_{+,j}^{*}(w_{k})}{a^{\prime}(w_{k})}\cdot\frac{1}{w-w_{k}}\\
 & = & -\sum_{k=1}^{N}c_{k}w_{k}^{j}B_{+,j}^{*}(w_{k})\cdot\frac{1}{w-w_{k}}\\
 & = & -\Pi_{-}Q_{+,j}(wB_{+,j}^{*})\end{eqnarray*}
and \begin{eqnarray*}
\Pi_{-}\frac{1}{a}w^{-1}B_{-,j} & = & \sum_{k=1}^{N}\frac{w_{k}^{-1}B_{-,j}(w_{k})}{a^{\prime}(w_{k})}\cdot\frac{1}{w-w_{k}}\\
 & = & \sum_{k=1}^{N}\frac{c_{k}^{\prime}w_{k}^{j-1}A_{+,j}^{*}(w_{k})}{a^{\prime}(w_{k})}\cdot\frac{1}{w-w_{k}}\\
 & = & \sum_{k=1}^{N}c_{k}w_{k}^{j-1}A_{+,j}^{*}(w_{k})\cdot\frac{1}{w-w_{k}}\\
 & = & \Pi_{-}Q_{+,j}A_{+,j}^{*}\end{eqnarray*}
where \[
Q_{+,j}(w)=\sum_{k=1}^{m}\frac{c_{k}w_{k}^{j-1}}{w-w_{k}}.\]
So, equations (\ref{eq:drmd1}) and (\ref{eq:drmd2}) become\begin{eqnarray}
A_{+,j} & = & \hat{A}_{+,j}(0)+\Pi_{-}(rw^{j-1}-Q_{+,j})(wB_{+,j}^{*})\label{eq:drmd3p}\\
wB_{+,j}^{*} & = & -\Pi_{+}(rw^{j-1}-Q_{+,j})^{*}A_{+,j}\label{eq:drmd4p}\end{eqnarray}
 or\begin{eqnarray}
A_{+,j} & = & \hat{A}_{+,j}(0)+\Pi_{-}r_{j}(wB_{+,j}^{*})\label{eq:drmd3}\\
wB_{+,j}^{*} & = & -\Pi_{+}r_{j}^{*}A_{+,j}\label{eq:drmd4}\end{eqnarray}
 where \begin{eqnarray*}
r_{j} & = & \Pi_{-}rw^{j-1}-Q_{+,j}\\
 & = & \Pi_{-}rw^{j-1}-\sum_{k=1}^{N}\frac{c_{k}w_{k}^{j-1}}{w-w_{k}}.\end{eqnarray*}
Equations (\ref{eq:drmd3}) and (\ref{eq:drmd4}) can be combined
into the single Marchenko equation:\begin{equation}
(1+\Pi_{+}r_{j}^{*}\Pi_{-}r_{j})\frac{wB_{+,j}^{*}}{\hat{A}_{+,j}(0)}=-\Pi_{+}r_{j}^{*}.\label{eq:drmd mar}\end{equation}

To prove equations (\ref{eq:dt23}) and (\ref{eq:dt24}), we write
the recursion (\ref{eq:dt1}) in terms of $A_{\pm}$ and $B_{\pm}$:\begin{equation}
\left[\begin{array}{c}
A_{\pm,j+1}\\
B_{\pm,j+1}\end{array}\right]=(1+|\gamma_{j}|^{2})^{-\frac{1}{2}}\left[\begin{array}{cc}
1 & \gamma_{j}\\
-\gamma_{j}^{*}w & w\end{array}\right]\left[\begin{array}{c}
A_{\pm,j}\\
B_{\pm,j}\end{array}\right]\label{eq:drmd5}\end{equation}
 \begin{equation}
\left[\begin{array}{c}
A_{\pm,j}\\
B_{\pm,j}\end{array}\right]=(1+|\gamma_{j}|^{2})^{-\frac{1}{2}}\left[\begin{array}{cc}
1 & -\gamma_{j}w^{-1}\\
\gamma_{j}^{*} & w^{-1}\end{array}\right]\left[\begin{array}{c}
A_{\pm,j+1}\\
B_{\pm,j+1}\end{array}\right]\label{eq:drmd6}\end{equation}
for \[
\gamma_{j}=\frac{\mu_{j}}{|\mu_{j}|}\tan|\mu_{j}|.\]
These recursions immediately imply\[
-\gamma_{j}^{*}=\frac{\hat{B}_{-,j+1}(1)}{\hat{A}_{-,j+1}(0)}\]
 and\[
\gamma_{j}^{*}=\frac{\hat{B}_{+,j}(0)}{\hat{A}_{+,j}(0)},\]
 as desired.

\section{Derivation of the DIST recursion\label{sec:Derivation-of-the-DIST}}

To derive equation (\ref{eq:dt23}) we simply plug equation (\ref{eq:drmd4})
into the recursion (\ref{eq:drmd6}):

\begin{eqnarray*}
wB_{+,j}^{*} & = & -\Pi_{+}r_{j}^{*}A_{+,j}\\
w(\gamma_{j}A_{+,j+1}^{*}+wB_{+,j+1}^{*}) & = & -\Pi_{+}r_{j}^{*}(A_{+,j+1}-\gamma_{j}w^{-1}B_{+,j+1})\\
\gamma_{j}(wA_{+,j+1}^{*}-\Pi_{+}w^{-1}r_{j}^{*}B_{+,j+1}) & = & -\Pi_{+}r_{j}^{*}A_{+,j+1}-w^{2}B_{+,j+1}^{*}\end{eqnarray*}
Examining the coefficient of $w^{1}$, we have\[
\gamma_{j}\left(\FF(A_{+,j+1})(0)-\FF(w^{-1}r_{j}^{*}(w^{-1}B_{+,j+1}))(0)\right)=-\FF(w^{-1}r_{j}^{*}A_{+,j+1})(0).\]
This immediately gives the desired result.

A similar computation can be used to obtain equation (\ref{eq:dt24}):\begin{eqnarray*}
B_{-,j+1} & = & -\Pi_{+}s_{j+1}^{*}A_{-,j+1}^{*}\\
-\gamma_{j}^{*}wA_{-,j}+wB_{-,j} & = & -\Pi_{+}s_{j+1}^{*}A_{-,j}^{*}-\gamma_{j}^{*}\Pi_{+}s_{j+1}^{*}B_{-,j}^{*}\\
-\gamma_{j}^{*}\left(wA_{-,j}-\Pi_{+}s_{j+1}^{*}B_{-,j}^{*}\right) & = & -\Pi_{+}s_{j+1}^{*}A_{-,j}^{*}-wB_{-,j}\\
-\gamma_{j}^{*}\left(\FF(A_{-,j})(0)-\FF(w^{-1}s_{j+1}^{*}B_{-,j}^{*})(0)\right) & = & -\FF(w^{-1}s_{j+1}^{*}A_{-,j}^{*})(0).\end{eqnarray*}

\section{The discrete energy formula\label{sub:The-energy-formula-discrete}}

In this section, we prove Theorem \ref{thm:energy discrete} and Corollary
\ref{cor:energy discrete}. Let \[
S=(a,b;w_{1},\dots,w_{m};c_{1}^{\prime},\dots,c_{m}^{\prime})\]
 be discrete scattering data, and let $A_{\pm}$ and $B_{\pm}$ be
as in Section \ref{sec:Discrete-Theory}. By equation (\ref{eq:dt7}),
we know that $a$ is given by \[
a=\lim_{j\rightarrow+\infty}A_{-,j}.\]
Therefore, since $\lim_{j\rightarrow-\infty}A_{-,j}=1$, the recursion
(\ref{eq:drmd5}) tells us that \[
a(0)=\hat{a}(0)=\prod_{j=-\infty}^{\infty}(1+|\gamma_{j}|^{2})^{-\frac{1}{2}}.\]
 Let us write \[
a(w)=a_{0}(w)\cdot\prod_{k=1}^{m}\frac{w_{k}^{*}}{|w_{k}|}\frac{w-w_{k}}{1-w_{k}^{*}w},\]
 where $a_{0}$ is analytic in the unit disk, and $w_{1},\dots,w_{m}$
are the zeros of $a$ in the unit disk. Since $\log|a_{0}|$ is harmonic
in the unit disk, and since $a(0)$ is positive, we have \begin{eqnarray*}
\frac{1}{2\pi}\int_{0}^{2\pi}\log|a_{0}(e^{i\theta})|d\theta & = & \log|a_{0}(0)|\\
 & = & \log a(0)-\sum_{k=1}^{m}\log|w_{k}|\\
 & = & -\frac{1}{2}\sum_{j=-\infty}^{\infty}\log(1+|\gamma_{j}|^{2})-\sum_{k=1}^{m}\log|w_{k}|.\end{eqnarray*}
Therefore, \begin{equation}
\sum_{j=-\infty}^{\infty}\log(1+|\gamma_{j}|^{2})=-\frac{1}{\pi}\int_{0}^{2\pi}\log|a_{0}(e^{i\theta})|d\theta-2\sum_{k=1}^{m}\log|w_{k}|.\label{eq:energy 1}\end{equation}
For $w$ on the unit circle, we have $|a_{0}(w)|=|a(w)|=(1+|r(w)|^{2})^{-\frac{1}{2}}$,
where $r=\frac{b}{a}$ is the reflection coefficient. Therefore (\ref{eq:energy 1})
becomes \[
\sum_{j=-\infty}^{\infty}\log(1+|\gamma_{j}|^{2})=\frac{1}{2\pi}\int_{0}^{2\pi}\log(1+|r(e^{i\theta})|^{2})d\theta-2\sum_{k=1}^{m}\log|w_{k}|.\]
By equation (\ref{eq:ZS3}), this is \[
\sum_{j=-\infty}^{\infty}\log(1+\tan^{2}\frac{|\omega_{j}|}{2})=\frac{1}{2\pi}\int_{0}^{2\pi}\log(1+|r(e^{i\theta})|^{2})d\theta-2\sum_{k=1}^{m}\log|w_{k}|.\]
This proves Theorem \ref{thm:energy discrete}. Corollary \ref{cor:energy discrete}
follows immediately in light of the proof of part (b) of Theorem \ref{thm:main discrete}.
Notice that if the time step, $\Delta$, is small, then the magnitude
of $\omega_{j}$ is also small, and so we have\[
\sum_{j=-\infty}^{\infty}|\omega_{j}|^{2}\approx\frac{4}{2\pi}\int_{0}^{2\pi}\log(1+|r(e^{i\theta})|^{2})d\theta-8\sum_{k=1}^{m}\log|w_{k}|.\]

\section{Proof of the main result: discrete case\label{sec:Proof-main-discrete}}

In this section, we prove Theorem \ref{thm:main discrete}.

\textbf{Proof of part (a):} We start with the reduced discrete scattering
data \[
\tilde{S}=(r;w_{1},\dots,w_{m};c_{1},\dots,c_{m})\]
 where $r$ is in $H^{1}(S^{1})$. This corresponds to unique discrete
scattering data \[
S=(a,b;w_{1},\dots,w_{m};c_{1}^{\prime},\dots,c_{m}^{\prime}).\]
 Proposition \ref{prop: proof main dc} below tells us that there
is a unique hard pulse with scattering data $S$.

\textbf{Proof of part (b):} This follows from part (a), in the special
case where, in the notation of Section \ref{sec:Derivation-discrete-marchenko},
$Q_{+,\rho}=r_{0}$, because in this case we have $r_{j}=0$ for all
$j\geq\rho$. This implies that the hard pulse vanishes for time steps
$j\geq\rho$, as desired. See Appendix \ref{app:non-simple} for the
case of non-simple poles.

\textbf{Proof of part (c):} This follows from part (b). Simply note
that the scattering data in this case has the form \[
S=(A,w^{-\rho}B;w_{1},\dots,w_{m};c_{1},\dots,c_{m}).\]
 One can check that the scattering data for the time reversed pulse
$\mu_{j}=-\mu_{-j}^{*}$ is \[
\overleftarrow{S}=(A,-w^{\rho}B^{*};\tilde{w}_{1},\dots,\tilde{w}_{m};\tilde{c}_{1},\dots,\tilde{c}_{m}).\]
 Notice that $w^{T-\rho}w^{\rho}B^{*}$ is analytic in the upper half
plane, which implies that the time reversed pulse ends at time step
$T-\rho$, as desired. We omit some details about the bound state
data which need to be worked out.

\begin{lem}
\label{lem:1}Let $S=(a,b;w_{1},\dots,w_{m};c_{1},\dots,c_{m})$ be
discrete scattering data such that $r=\frac{b}{a}$ is in $H^{1}(S^{1})$.
For each $j$ there exist unique functions $A_{-,j}$, $w^{-1}B_{-,j}$,
$A_{+,j}^{*}$, and $B_{+,j}^{*}$ in $\tilde{H}_{+}^{1}(S^{1})$
such that \begin{equation}
\left[\begin{array}{cc}
A_{-,j} & -B_{-,j}^{*}\\
B_{-,j} & A_{-,j}^{*}\end{array}\right]=\left[\begin{array}{cc}
A_{+,j} & -B_{+,j}^{*}\\
B_{+,j} & A_{+,j}^{*}\end{array}\right]\left[\begin{array}{cc}
a & -b^{*}w^{-j}\\
bw^{j} & a^{*}\end{array}\right],\label{eq:lem 1}\end{equation}
\begin{equation}
\frac{1}{2\pi}\int_{0}^{2\pi}|A_{\pm,j}(e^{i\theta})|^{2}+|B_{\pm,j}(e^{i\theta})|^{2}d\theta=1,\label{eq:lem 1.1}\end{equation}
\begin{equation}
\hat{A}_{\pm,j}(0)>0,\label{eq:lem 1.2}\end{equation}
 and \begin{eqnarray}
\Pi_{-}\frac{1}{a}A_{-,j} & = & -\Pi_{-}Q_{+,j}wB_{+,j}^{*}\label{eq:lem 1Q1}\\
\Pi_{-}\frac{1}{a}w^{-1}B_{-,j} & = & \Pi_{-}Q_{+,j}A_{+,j}^{*}\label{eq:lem 1Q2}\end{eqnarray}
 where \begin{eqnarray*}
Q_{+,j}(w) & = & \sum_{k=1}^{m}\frac{c_{k}w_{k}^{j-1}}{w-w_{k}}.\end{eqnarray*}

\end{lem}
\begin{proof}
We first prove uniqueness. From Section \ref{sec:Derivation-discrete-marchenko}
we see that $\frac{wB_{+,j}^{*}}{\hat{A}_{+,j}(0)}$ must be the unique
solution to equation (\ref{eq:drmd mar}) where \[
r_{j}=\Pi_{-}rw^{j-1}-Q_{+,j}.\]
Using equation (\ref{eq:drmd3}) we can determine\[
\frac{A_{+,j}}{\hat{A}_{+,j}(0)}=1+\Pi_{-}r_{j}\frac{wB_{+,j}^{*}}{\hat{A}_{+,j}(0)}.\]
The functions $A_{+,j}$ and $B_{+,j}$ are then uniquely defined
by (\ref{eq:lem 1.1}) and (\ref{eq:lem 1.2}). Finally, $A_{-,j}$
and $B_{-,j}$ can be computed from $A_{+,j}$ and $B_{+,j}$ using
the matrix equation (\ref{eq:lem 1}). This proves uniqueness.

To prove existence, we just need to show that $A_{-,j}$, $w^{-1}B_{-,j}$,
$A_{+,j}^{*}$, and $B_{+,j}^{*}$ as defined above are all in $\tilde{H}_{+}^{1}(S^{1})$,
and satisfy equations (\ref{eq:lem 1Q1}) and (\ref{eq:lem 1Q2}).
By construction, we know that $A_{+,j}^{*}$ and $B_{+,j}^{*}$ are
in $\tilde{H}_{+}^{1}(S^{1})$. Also, by construction, we know that
equations (\ref{eq:drmd1}), (\ref{eq:drmd2}), (\ref{eq:drmd3p}),
and (\ref{eq:drmd4p}) hold. Comparing these equations immediately
gives (\ref{eq:lem 1Q1}) and (\ref{eq:lem 1Q2}). So we just need
to show that $A_{-,j}$ and $w^{-1}B_{-,j}$ are analytic in the unit
disk. But this actually follows from equations (\ref{eq:lem 1Q1})
and (\ref{eq:lem 1Q2}). Indeed, equation (\ref{eq:lem 1Q1}) implies
that $\frac{1}{a}A_{-,j}+Q_{+,j}wB_{+,j}^{*}$ is analytic in the
unit disk. Since $\frac{1}{a}$ and $Q_{+,j}$ have the same poles
and multiplicities, and since $wB_{+,j}^{*}$ is analytic in the disk,
it follows that $A_{-,j}$ must also be analytic in the unit disk.
Similarly, equation (\ref{eq:lem 1Q2}) proves the analyticity of
$w^{-1}B_{-,j}$.
\end{proof}
\begin{lem}
\label{lem:2}The functions $A_{\pm,j}$ and $B_{\pm,j}$ given in
Lemma \ref{lem:1} satisfy the recursion \[
\left[\begin{array}{c}
A_{\pm,j+1}\\
B_{\pm,j+1}\end{array}\right]=(1+|\gamma_{j}|^{2})^{-\frac{1}{2}}\left[\begin{array}{cc}
1 & \gamma_{j}\\
-\gamma_{j}^{*}w & w\end{array}\right]\left[\begin{array}{c}
A_{\pm,j}\\
B_{\pm,j}\end{array}\right]\]
 where \[
\gamma_{j}^{*}=\frac{\hat{B}_{+,j}(0)}{\hat{A}_{+,j}(0)}.\]

\end{lem}
\begin{proof}
By uniqueness of the functions we just need to show that $\tilde{A}_{\pm,j+1}$
and $\tilde{B}_{\pm,j+1}$ given by\[
\left[\begin{array}{c}
\tilde{A}_{\pm,j+1}\\
\tilde{B}_{\pm,j+1}\end{array}\right]=(1+|\gamma_{j}|^{2})^{-\frac{1}{2}}\left[\begin{array}{cc}
1 & \gamma_{j}\\
\gamma_{j}^{*}w & w\end{array}\right]\left[\begin{array}{c}
A_{\pm,j}\\
B_{\pm,j}\end{array}\right]\]
 and \[
\gamma_{j}=\frac{\FF(B_{+,j}^{*})(0)}{\hat{A}_{+,j}(0)}\]
 have the desired properties at step $j+1$. By multiplying equation
(\ref{eq:lem 1}) on the left by $(1+|\gamma_{j}|^{2})^{-\frac{1}{2}}\left[\begin{array}{cc}
1 & \gamma_{j}\\
-\gamma_{j}^{*}w & w\end{array}\right]$, it is easy to see that the matrix equation holds. It is clear that
$\tilde{A}_{-,j+1}$, $w^{-1}\tilde{B}_{-,j+1}$, $\tilde{A}_{+,j+1}^{*}$,
and $\tilde{B}_{+,j+1}^{*}$ are in $\tilde{H}^{+}(S^{1})$, because
$\gamma_{j}$ was chosen precisely so that this condition would hold.
Also it is clear that properties (\ref{eq:lem 1.1}) and (\ref{eq:lem 1.2})
are satisfied. So we just need to verify equations (\ref{eq:lem 1Q1})
and (\ref{eq:lem 1Q2}) at step $j+1$:\begin{eqnarray*}
\tiny{(1+|\gamma_{j}|^{2})^{\frac{1}{2}}\Pi_{-}(\frac{1}{a}\tilde{A}_{-,j+1}+Q_{+,j+1}w\tilde{B}_{+,j+1}^{*})} & = & \Pi_{-}(\frac{1}{a}A_{-,j}+\gamma_{j}\frac{1}{a}B_{-,j}-\gamma_{j}Q_{+,j+1}A_{+,j}^{*})\\
 &  & +\Pi_{-}Q_{+,j+1}B_{+,j}^{*}\\
 & = & \Pi_{-}\gamma_{j}w(\frac{1}{a}w^{-1}B_{-,j}-w^{-1}Q_{+,j+1}A_{+,j}^{*})\\
 &  & +\Pi_{-}(\frac{1}{a}A_{-,j}+w^{-1}Q_{+,j+1}wB_{+,j}^{*})\\
 & = & 0\\
\tiny{(1+|\gamma_{j}|^{2})^{\frac{1}{2}}\Pi_{-}(\frac{1}{a}w^{-1}\tilde{B}_{-,j+1}-Q_{+,j+1}\tilde{A}_{+,j+1}^{*})} & = & \Pi_{-}(-\gamma_{j}^{*}\frac{1}{a}A_{-,j}+\frac{1}{a}B_{-,j}-Q_{+,j+1}A_{+,j}^{*})\\
 &  & -\Pi_{-}\gamma_{j}^{*}Q_{+,j+1}B_{+,j}^{*}\\
 & = & -\Pi_{-}\gamma_{j}^{*}(\frac{1}{a}A_{-,j}+w^{-1}Q_{+,j+1}wB_{+,j}^{*})\\
 &  & +\Pi_{-}w(\frac{1}{a}w^{-1}B_{-,j}-w^{-1}Q_{+,j+1}A_{+,j}^{*})\\
 & = & 0.\end{eqnarray*}
Here we used the fact that \[
w^{-1}Q_{+,j+1}=Q_{+,j}+w^{-1}g,\]
 where $g$ is analytic in the unit disk.
\end{proof}
\begin{lem}
\label{lem:3}The functions $A_{+,j}$ and $B_{+,j}$ given in Lemma
\ref{lem:1} satisfy \[
\frac{1}{\hat{A}_{+,j}(0)}\left[\begin{array}{c}
A_{+,j}\\
B_{+,j}\end{array}\right]-\left[\begin{array}{c}
1\\
0\end{array}\right]\underrightarrow{H^{1}(S^{1})}\left[\begin{array}{c}
0\\
0\end{array}\right]\textrm{ as }j\rightarrow+\infty.\]

\end{lem}
\begin{proof}
Since $\frac{wB_{+,j}^{*}}{\hat{A}_{+,j}(0)}$ is the solution to
\[
(1+\Pi_{+}r_{j}^{*}\Pi_{-}r_{j})\frac{wB_{+,j}^{*}}{\hat{A}_{+,j}(0)}=-\Pi_{+}r_{j}^{*},\]
 we can use Proposition \ref{pro:TME 2} to estimate its $H^{1}$-norm
by \[
\left\Vert \frac{wB_{+,j}^{*}}{\hat{A}_{+,j}(0)}\right\Vert _{H^{1}}\leq\left\Vert \Pi_{+}r_{j}^{*}\right\Vert _{H^{1}}.\]
 For $j\geq0$ this is \[
\left\Vert \frac{B_{+,j}^{*}}{\hat{A}_{+,j}(0)}\right\Vert _{H^{1}}\leq\left\Vert \Pi_{+}w^{-j}r_{0}^{*}\right\Vert _{H^{1}}.\]
 Clearly the expression on the right tends to zero, which implies
that \[
\lim_{j\rightarrow+\infty}\left\Vert \frac{B_{+,j}^{*}}{\hat{A}_{+,j}(0)}\right\Vert _{H^{1}}=0.\]
 By equation (\ref{eq:drmd3}) and Fact \ref{fac:Sob 4}, we have
(for $j\geq0$)\begin{eqnarray*}
\left\Vert \frac{A_{+,j}}{\hat{A}_{+,j}(0)}-1\right\Vert _{H^{1}} & = & \left\Vert \Pi_{-}r_{j}\frac{wB_{+,j}^{*}}{\hat{A}_{+,j}(0)}\right\Vert _{H^{1}}\\
 & = & \left\Vert \Pi_{-}w^{j}r_{0}\frac{wB_{+,j}^{*}}{\hat{A}_{+,j}(0)}\right\Vert _{H^{1}}\\
 & \leq & 2\left\Vert r_{0}\right\Vert _{H^{1}}\left\Vert \frac{B_{+,j}}{\hat{A}_{+,j}(0)}\right\Vert _{H^{1}},\end{eqnarray*}
 which implies that \[
\lim_{j\rightarrow+\infty}\left\Vert \frac{A_{+,j}}{\hat{A}_{+,j}(0)}-1\right\Vert _{H^{1}}=0,\]
 as desired. 
\end{proof}
\begin{lem}
\label{lem:4}Let $A_{\pm}$ and $B_{\pm}$ be the functions given
in Lemma \ref{lem:1} and let $\gamma_{j}$ be as in Lemma \ref{lem:2}.
Then for each $j$ we have \begin{equation}
\left|\frac{A_{+,j}}{\hat{A}_{+,j}(0)}\right|^{2}+\left|\frac{B_{+,j}}{\hat{A}_{+,j}(0)}\right|^{2}=\prod_{k=j}^{\infty}(1+|\gamma_{k}|^{2})^{-1}\textrm{ on }S^{1}.\label{eq:lem 4}\end{equation}
 
\end{lem}
\begin{proof}
The recursion in Lemma \ref{lem:2} implies that\[
\hat{A}_{+,j+1}(0)=(1+|\gamma_{j}|^{2})^{-\frac{1}{2}}\hat{A}_{+,j}(0)\]
 and \[
|A_{+,j+1}|^{2}+|B_{+,j+1}|^{2}=|A_{+,j}|^{2}+|B_{+,j}|^{2}\textrm{ on }S^{1}.\]
Together, these imply that \[
\left|\frac{A_{+,j+1}}{\hat{A}_{+,j+1}(0)}\right|^{2}+\left|\frac{B_{+,j+1}}{\hat{A}_{+,j+1}(0)}\right|^{2}=(1+|\gamma_{j}|^{2})\cdot\left(\left|\frac{A_{+,j}}{\hat{A}_{+,j}(0)}\right|^{2}+\left|\frac{B_{+,j}}{\hat{A}_{+,j}(0)}\right|^{2}\right)\textrm{ on }S^{1}\]
 for all $j$. Therefore, formula (\ref{eq:lem 4}) follows from Lemma
\ref{lem:3}.
\end{proof}
\begin{lem}
\label{lem:5}The functions $A_{+,j}$ and $B_{+,j}$ given in Lemma
\ref{lem:1} satisfy \begin{equation}
\left[\begin{array}{c}
A_{+,j}-1\\
B_{+,j}\end{array}\right]\underrightarrow{H^{1}(S^{1})}\left[\begin{array}{c}
0\\
0\end{array}\right]\textrm{ as }j\rightarrow+\infty\label{eq:lem 5a}\end{equation}
 and \begin{equation}
|A_{+,j}|^{2}+|B_{+,j}|^{2}=1\textrm{ on }S^{1}.\label{eq:lem 5b}\end{equation}
 
\end{lem}
\begin{proof}
Lemma \ref{lem:4} implies that $|A_{+,j}|^{2}+|B_{+,j}|^{2}$ is
constant on $S^{1}$. The recursion in Lemma (\ref{lem:2}) implies
that this function is independent of $j$. So equation (\ref{eq:lem 5b})
follows from property (\ref{eq:lem 1.1}). 

Lemma \ref{lem:4} also implies that $\lim_{j\rightarrow\infty}\hat{A}_{+,j}(0)=1$.
Therefore equation (\ref{eq:lem 5a}) follows from Lemma \ref{lem:3}.
\end{proof}
Each of the above lemmas has an analogue for the functions $A_{-,j}$
and $B_{-,j}$, which we omit. The above lemmas together with their
analogues, give the following 

\begin{prop}
\label{prop: proof main dc}Let $S=(a,b;w_{1},\dots,w_{m};c_{1},\dots,c_{m})$
be discrete scattering data such that $r=\frac{b}{a}$ is in $H^{1}(S^{1})$.
For each $j\in\ZZ$ there exist unique functions $A_{-,j}$, $w^{-1}B_{-,j}$,
$A_{+,j}^{*}$, and $B_{+,j}^{*}$ in $\tilde{H}_{+}^{1}(S^{1})$
such that \begin{equation}
\left[\begin{array}{cc}
a & -b^{*}w^{-j}\\
bw^{j} & a^{*}\end{array}\right]=\left[\begin{array}{cc}
A_{+,j}^{*} & B_{+,j}^{*}\\
-B_{+,j} & A_{+,j}\end{array}\right]\left[\begin{array}{cc}
A_{-,j} & -B_{-,j}^{*}\\
B_{-,j} & A_{-,j}^{*}\end{array}\right],\label{eq:pm1 dc}\end{equation}
 and \begin{eqnarray}
\Pi_{-}\frac{1}{a}A_{-,j} & = & -\Pi_{-}Q_{+,j}B_{+,j}^{*}\label{eq:pm2 dc}\\
\Pi_{-}\frac{1}{a}w^{-1}B_{-,j} & = & \Pi_{-}Q_{+,j}A_{+,j}^{*}\label{eq:pm3 dc}\end{eqnarray}
 where \begin{eqnarray}
Q_{+,j}(w) & = & \sum_{k=1}^{m}\frac{c_{k}w^{j-1}}{w-w_{k}}.\label{eq:pm4 dc}\end{eqnarray}
Furthermore, if we set \[
\gamma_{j}=\frac{\FF(B_{+,j}^{*})(0)}{\hat{A}_{+,j}(0)}\]
 or \[
-\gamma_{j}^{*}=\frac{\FF(B_{-,j})(0)}{\hat{A}_{-,j}(0)},\]
 then these functions satisfy \[
\left[\begin{array}{c}
\tilde{A}_{\pm,j+1}\\
\tilde{B}_{\pm,j+1}\end{array}\right]=(1+|\gamma_{j}|^{2})^{-\frac{1}{2}}\left[\begin{array}{cc}
1 & \gamma_{j}\\
\gamma_{j}^{*}w & w\end{array}\right]\left[\begin{array}{c}
A_{\pm,j}\\
B_{\pm,j}\end{array}\right]\]
 and \[
\left[\begin{array}{c}
A_{\pm,j}\\
B_{\pm,j}\end{array}\right]\underrightarrow{H^{1}(S^{1})}\left[\begin{array}{c}
1\\
0\end{array}\right]\textrm{ as }j\rightarrow\pm\infty.\]

\end{prop}

\section{Proof of the main result: continuum case\label{sec:Proof-main-continuum}}

In this section, we prove Theorem \ref{thm:main continuum}.

\textbf{Proof of part (a):} We start with the reduced scattering data
\[
\tilde{S}=(r;\xi_{1},\dots,\xi_{m};C_{1},\dots,C_{m})\]
 where $r$ and $\xi r$ are in $H^{1}(\RR)$. This corresponds to
unique scattering data \[
S=(a,b;\xi_{1},\dots,\xi_{m};C_{1}^{\prime},\dots,C_{m}^{\prime}).\]
 Proposition \ref{prop: proof main cc} below tells us that there
is a unique pulse with scattering data $S$.

\textbf{Proof of part (b):} This follows from part (a), in the special
case where, in the notation of Section \ref{sec:Derivation-right-continuum},
$Q_{+,\rho}=r_{0}$, because in this case we have $r_{t}=0$ for all
$t\geq\rho$. This implies that the pulse vanishes for times $t\geq\rho$,
as desired. See Appendix \ref{app:non-simple} for the case of non-simple
poles.

\textbf{Proof of part (c):} This follows from part (b). Simply note
that the scattering data in this case has the form $S=(A,e^{-2i\rho\xi}B;\xi_{1},\dots,\xi_{m};C_{1},\dots,C_{m}).$
One can check that the scattering data for the time reversed pulse
$\overleftarrow{q}(t)=-q(-t)^{*}$ is \[
\overleftarrow{S}=(A,-e^{2i\rho\xi}B^{*};\tilde{\xi}_{1},\dots,\tilde{\xi}_{m};\tilde{C}_{1},\dots,\tilde{C}_{m}).\]
 Notice that $e^{2i(T-\rho)\xi}w^{2i\rho\xi}B^{*}$ is analytic in
the upper half plane, which implies that the time reversed pulse ends
at time $T-\rho$, as desired. We omit some details about the bound
state data which need to be worked out.

\begin{lem}
\label{lem:1 cc}Let $S=(a,b;\xi_{1},\dots,\xi_{m};C_{1},\dots,C_{m})$
be scattering data such that $r=\frac{b}{a}$ and $\xi r$ are in
$H^{1}(\RR)$. For each $t\in\RR$ there exist unique functions $A_{-,t}-1$,
$B_{-,t}$, $A_{+,t}^{*}$, and $B_{+,t}^{*}$ in $H_{+}^{1}(\RR)$
such that \begin{equation}
\left[\begin{array}{cc}
A_{-,t} & -B_{-,t}^{*}\\
B_{-,t} & A_{-,t}^{*}\end{array}\right]=\left[\begin{array}{cc}
A_{+,t} & -B_{+,t}^{*}\\
B_{+,t} & A_{+,t}^{*}\end{array}\right]\left[\begin{array}{cc}
a & -b^{*}e^{-2i\xi t}\\
be^{2i\xi t} & a^{*}\end{array}\right],\label{eq:lem 1 cc}\end{equation}
 and \begin{eqnarray}
\Pi_{-}\frac{1}{a}A_{-,t} & = & -\Pi_{-}Q_{+,t}B_{+,t}^{*}\label{eq:lem 1Q1 cc}\\
\Pi_{-}\frac{1}{a}B_{-,t} & = & \Pi_{-}Q_{+,t}A_{+,t}^{*}\label{eq:lem 1Q2 cc}\end{eqnarray}
 where \begin{eqnarray*}
Q_{+,t}(w) & = & \sum_{k=1}^{m}\frac{C_{k}e^{2i\xi_{k}t}}{\xi-\xi_{k}}.\end{eqnarray*}

\end{lem}
\begin{proof}
We first prove uniqueness. From Section \ref{sec:Derivation-right-continuum}
we see that $B_{+,t}^{*}$ must be the unique solution to equation
(\ref{eq:drmd mar}) where \[
r_{t}=\Pi_{-}re^{2i\xi t}-Q_{+,t}.\]
Using equation (\ref{eq:drm3}) we can determine\[
A_{+,t}=1+\Pi_{-}r_{t}B_{+,t}^{*}.\]
Finally, $A_{-,t}$ and $B_{-,t}$ can be computed from $A_{+,t}$
and $B_{+,t}$ using the matrix equation (\ref{eq:lem 1 cc}). This
proves uniqueness.

To prove existence, we just need to show that $A_{-,t}-1$, $B_{-,t}$,
$A_{+,t}^{*}-1$, and $B_{+,t}^{*}$ as defined in the previous paragraph
are all in $H_{+}^{1}(\RR)$, and satisfy equations (\ref{eq:lem 1 cc}),
(\ref{eq:lem 1Q1 cc}) and (\ref{eq:lem 1Q2 cc}). By construction,
we know that $A_{+,t}^{*}-1$ and $B_{+,t}^{*}$ are in $H_{+}^{1}(\RR)$.
Also, by construction, we know that equations (\ref{eq:drm1}), (\ref{eq:drm2}),
(\ref{eq:drm3p}), and (\ref{eq:drm4p}) hold. Comparing these equations
immediately gives (\ref{eq:lem 1Q1 cc}) and (\ref{eq:lem 1Q2 cc}).
So we just need to show that $A_{-,t}$ and $B_{-,t}$ are analytic
in the upper half plane. But this actually follows from equations
(\ref{eq:lem 1Q1 cc}) and (\ref{eq:lem 1Q2 cc}). Indeed, equation
(\ref{eq:lem 1Q1 cc}) implies that $\frac{1}{a}A_{-,t}+Q_{+,t}B_{+,t}^{*}$
is analytic in the upper half plane. Since $\frac{1}{a}$ and $Q_{+,t}$
have the same poles and multiplicities, and since $B_{+,t}^{*}$ is
analytic in the upper half plane, it follows that $A_{-,t}$ must
also be analytic in the upper half plane. Similarly, equation (\ref{eq:lem 1Q2 cc})
proves the analyticity of $B_{-,t}$.
\end{proof}
\begin{lem}
\label{lem:2 cc}Let $A_{\pm,t}$ and $B_{\pm,t}$ be the functions
given in Lemma \ref{lem:1 cc}. Then for each $\xi\in\RR$, $A_{\pm,t}(\xi)$
and $B_{\pm,t}(\xi)$ are differentiable with respect to $t$, and
we have \begin{equation}
\partial_{t}\left[\begin{array}{c}
A_{+,t}(\xi)\\
B_{+,t}(\xi)\end{array}\right]=\left[\begin{array}{cc}
0 & \gamma_{t}\\
-\gamma_{t}^{*} & 2i\xi\end{array}\right]\left[\begin{array}{c}
A_{+,t}(\xi)\\
B_{+,t}(\xi)\end{array}\right]\label{eq:is 100}\end{equation}
 where \[
\gamma_{t}=\frac{1}{\pi}\FF(B_{+,t}^{*})(0^{+}).\]

\end{lem}
\begin{proof}
Let us first show that for each $\xi\in\RR$, $A_{+,t}(\xi)$ and
$B_{+,t}(\xi)$ are differentiable with respect to $t$. By Proposition
\ref{pro:bd 4}, it is enough to show that $A_{+,t}-1$ and $B_{+,t}$
are differentiable as curves in the Banach space $H^{1}(\RR)$. Let
us prove this for $B_{+,t}$ using the fact that it is a solution
to the Marchenko equation (\ref{eq:drm mar}). By Proposition \ref{pro:bd 3}
and Lemma \ref{lem:bd 1}, it is sufficient to show that $(1+\pi_{+}r_{t}^{*}\pi_{-}r_{t})$
and $\pi_{+}r_{t}^{*}$ are differentiable as curves in $\LL(H^{1}(\RR),H^{1}(\RR))$
and $H^{1}(\RR)$, respectively. But this follows easily from Proposition
\ref{pro:bd 5} and Fact \ref{fac:Sob 4}. Finally, the differentiability
of $A_{+,t}$ is apparent using equation (\ref{eq:drm3}).

So we know that $A_{+,t}$ and $B_{+,t}$ have $t$-derivatives $\dot{A}_{+,t}$
and $\dot{B}_{+,t}$ in $H^{1}(\RR)$. Differentiating equations (\ref{eq:drm3})
and (\ref{eq:drm4}) with respect to $t$ gives \begin{eqnarray}
\dot{A}_{+,t} & = & \Pi_{-}2i\xi r_{t}B_{+,t}^{*}+\Pi_{-}r_{t}\dot{B}_{+,t}^{*}\label{eq:is 101}\\
\dot{B}_{+,t} & = & -\Pi_{-}2i\xi r_{t}A_{+,t}^{*}-\Pi_{-}r_{t}\dot{A}_{+,t}^{*}.\label{eq:is 102}\end{eqnarray}
 This system has a unique solution for $\dot{A}_{+,t}$ and $\dot{B}_{+,t}$
since it can be combined into the single Marchenko type equation \[
(1+\Pi_{-}r_{t}\Pi_{+}r_{t}^{*})\dot{B}_{+,t}=-\Pi_{-}2i\xi r_{t}A_{+,t}^{*}+\Pi_{-}r_{t}\Pi_{+}2i\xi r_{t}^{*}B_{+,t}\]
 (see Section \ref{sub:The-Marchenko-Equation}). Plugging \ref{eq:is 100}
into \ref{eq:is 101} and \ref{eq:is 102} gives \begin{eqnarray*}
\gamma_{t}B_{+,t} & = & \Pi_{-}2i\xi r_{t}B_{+,t}^{*}-\gamma_{t}\Pi_{-}r_{t}A_{+,t}^{*}-\Pi_{-}2i\xi r_{t}B_{+,t}^{*}\\
-\gamma_{t}^{*}A_{+,t}+2i\xi B_{+,t} & = & -\Pi_{-}2i\xi r_{t}A_{+,t}^{*}-\gamma_{t}^{*}\Pi_{-}r_{t}B_{+,t}^{*}.\end{eqnarray*}
Using (\ref{eq:drm3}) and (\ref{eq:drm4}), these equations become\begin{eqnarray*}
0 & = & 0\\
2i\xi B_{+,t} & = & -\Pi_{-}2i\xi r_{t}A_{+,t}^{*}+\gamma_{t}^{*}.\end{eqnarray*}
The second of these equations is satisfied if we set \begin{eqnarray*}
-\gamma_{t}^{*} & = & -2i\xi B_{+,t}-\Pi_{-}2i\xi r_{t}A_{+,t}^{*}\\
 & = & 2i\xi\Pi_{-}r_{t}A_{+,t}^{*}-\Pi_{-}2i\xi r_{t}A_{+,t}^{*}\\
 & = & \frac{1}{\pi}\FF(r_{t}A_{+,t}^{*})(0)\\
 & = & -\frac{1}{\pi}\hat{B}_{+,t}(0^{-})\end{eqnarray*}
(see Lemma\ref{fac:proj}).
\end{proof}
\begin{lem}
\label{lem:3 cc}The functions $A_{+,t}$ and $B_{+,t}$ given in
Lemma \ref{lem:1 cc} satisfy \[
\left[\begin{array}{c}
A_{+,t}-1\\
B_{+,t}\end{array}\right]\underrightarrow{H^{1}(\RR)}\left[\begin{array}{c}
0\\
0\end{array}\right]\textrm{ as }t\rightarrow+\infty.\]

\end{lem}
\begin{proof}
Since $B_{+,t}^{*}$ is the solution to \[
(1+\Pi_{+}r_{t}^{*}\Pi_{-}r_{t})B_{+,t}^{*}=-\Pi_{+}r_{t}^{*},\]
 we can use \ref{pro:TME 2} to estimate its $H^{1}$-norm by \[
\left\Vert B_{+,t}^{*}\right\Vert _{H^{1}}\leq\left\Vert \Pi_{+}r_{t}^{*}\right\Vert _{H^{1}}.\]
 For $t\geq0$ this is \[
\left\Vert B_{+,t}^{*}\right\Vert _{H^{1}}\leq\left\Vert \Pi_{+}e^{-2i\xi t}r_{0}^{*}\right\Vert _{H^{1}}.\]
 Clearly the expression on the right tends to zero, which implies
that \[
\lim_{t\rightarrow+\infty}\left\Vert B_{+,t}^{*}\right\Vert _{H^{1}}=0.\]
 By equation (\ref{fac:Sob 4}) we have (for $t\geq0$)\begin{eqnarray*}
\left\Vert A_{+,t}-1\right\Vert _{H^{1}} & = & \left\Vert \Pi_{-}r_{t}B_{+,t}^{*}\right\Vert _{H^{1}}\\
 & = & \left\Vert \Pi_{-}e^{2i\xi t}r_{0}B_{+,t}^{*}\right\Vert _{H^{1}}\\
 & \leq & 2\left\Vert r_{0}\right\Vert _{H^{1}}\left\Vert B_{+,t}\right\Vert _{H^{1}},\end{eqnarray*}
 which implies that \[
\lim_{t\rightarrow+\infty}\left\Vert A_{+,t}-1\right\Vert _{H^{1}}=0,\]
 as desired. 
\end{proof}
\begin{lem}
\label{lem:5 cc}The functions $A_{+,t}$ and $B_{+,t}$ given in
Lemma \ref{lem:1 cc} satisfy \begin{equation}
|A_{+,t}|^{2}+|B_{+,t}|^{2}=1\textrm{ on }\RR.\label{eq:lem 5b cc}\end{equation}
 
\end{lem}
\begin{proof}
The differential equation in Lemma (\ref{lem:2 cc}) implies that
$|A_{+,t}|^{2}+|B_{+,t}|^{2}$ is independent of $t$. Therefore,
this result follows from Lemma \ref{lem:3 cc}.
\end{proof}
Each of the above lemmas has analogue for the functions $A_{-,t}$
and $B_{-,t}$, which we omit. The above lemmas together with their
analogues, give the following 

\begin{prop}
\label{prop: proof main cc}Let $S=(a,b;\xi_{1},\dots,\xi_{m};C_{1},\dots,C_{m})$
be scattering data such that $r=\frac{b}{a}$ and $\xi r$ are in
$H^{1}(\RR)$. For each $t\in\RR$ there exist unique functions $A_{-,t}-1$,
$B_{-,t}$, $A_{+,t}^{*}-1$, and $B_{+,t}^{*}$ in $H_{+}^{1}(\RR)$
such that \begin{equation}
\left[\begin{array}{cc}
a & -b^{*}e^{-2i\xi t}\\
be^{2i\xi t} & a^{*}\end{array}\right]=\left[\begin{array}{cc}
A_{+,t}^{*} & B_{+,t}^{*}\\
-B_{+,t} & A_{+,t}\end{array}\right]\left[\begin{array}{cc}
A_{-,t} & -B_{-,t}^{*}\\
B_{-,t} & A_{-,t}^{*}\end{array}\right],\label{eq:pm1 cc}\end{equation}
 and \begin{eqnarray}
\Pi_{-}\frac{1}{a}A_{-,t} & = & -\Pi_{-}Q_{+,t}B_{+,t}^{*}\label{eq:pm2 cc}\\
\Pi_{-}\frac{1}{a}B_{-,t} & = & \Pi_{-}Q_{+,t}A_{+,t}^{*}\label{eq:pm3 cc}\end{eqnarray}
 where \begin{eqnarray}
Q_{+,t}(\xi) & = & \sum_{k=1}^{m}\frac{C_{k}e^{2i\xi_{k}t}}{\xi-\xi_{k}}.\label{eq:pm4 cc}\end{eqnarray}
Furthermore, these functions satisfy \begin{equation}
\partial_{t}\left[\begin{array}{c}
A_{\pm,t}(\xi)\\
B_{\pm,t}(\xi)\end{array}\right]=\left[\begin{array}{cc}
0 & q(t)\\
-q(t)^{*} & 2i\xi\end{array}\right]\left[\begin{array}{c}
A_{\pm,t}(\xi)\\
B_{\pm,t}(\xi)\end{array}\right]\label{eq:is 100.1}\end{equation}
 and \[
\left[\begin{array}{c}
A_{\pm,t}-1\\
B_{\pm,t}\end{array}\right]\underrightarrow{H^{1}(\RR)}\left[\begin{array}{c}
0\\
0\end{array}\right]\textrm{ as }t\rightarrow\pm\infty,\]
where \[
q(t)=\frac{1}{\pi}\FF(B_{+,t}^{*})(0^{+})\]
 and \[
-q(t)^{*}=\frac{1}{\pi}\FF(B_{-,t})(0^{+}).\]
 
\end{prop}

\section{Applying the discrete algorithm to continuum scattering data\label{sec:Applying-the-discrete}}

Let $S=(r;\xi_{1},\dots,\xi_{m};C_{1},\dots,C_{m})$ be reduced continuum
scattering data with $r,\xi r\in H^{1}(\RR)$. The main theorem tells
us that there is a unique potential $q:\RR\rightarrow\CC$ corresponding
to $S$. However, to practically compute $q(t)$ for some $t\in\RR$,
one needs to somehow discretize the Marchenko equation\[
(1+\Pi_{+}r_{t}^{*}\Pi_{-}r_{t})h_{t}=-\Pi_{+}r_{t}.\]
 Let us describe one method of doing this.

Choose a time step $\Delta$. We can replace $r_{j\Delta}$ by the
periodic function \[
\tilde{r}_{j\Delta}(\xi)=\frac{\Delta}{\pi}\sum_{n=-\infty}^{-1}\hat{r}_{j\Delta}(2n\Delta)e^{2i\xi n\Delta},\]
 which is an approximation to $r_{j\Delta}$ in a neighborhood of
zero. Let us consider $\tilde{r}_{j\Delta}$ as a function of $w=e^{2i\xi\Delta}$,
so $\tilde{r}_{j\Delta}$ is in $H^{1}(S^{1})$. There is a unique
solution $\tilde{h}_{j\Delta}\in H^{1}(S^{1})$ to the discretized
Marchenko equation \[
(1+\Pi_{+}\tilde{r}_{j\Delta}^{*}\Pi_{-}\tilde{r}_{j\Delta})\tilde{h}_{j\Delta}=-\Pi_{+}\tilde{r}_{t}^{*}.\]
After solving this equation, we can approximate the potential by \begin{equation}
q(j\Delta)=\frac{1}{\pi}\FF(h_{j\Delta})(0^{+})\approx\Delta^{-1}\hat{\tilde{h}}_{j\Delta}(1).\label{eq:approx 1}\end{equation}

This, of course, greatly resembles the inverse scattering theory for
hard pulses! Let us make the correspondence explicit. We have \[
r_{j\Delta}(\xi)=r(\xi)e^{2i\xi j\Delta}-\sum_{k=1}^{m}\frac{C_{k}e^{2i\xi_{k}j\Delta}}{\xi-\xi_{k}},\]
 which implies that \[
\hat{r}_{j\Delta}(2n\Delta)=\hat{r}(2(n-j)\Delta)-2\pi i\sum_{k=1}^{m}C_{k}e^{2i\xi_{k}(j-n)\Delta}.\]
 Therefore,\begin{eqnarray*}
\tilde{r}_{j\Delta}(w) & = & \frac{\Delta}{\pi}\sum_{n=-\infty}^{-1}\hat{r}(2(n-j)\Delta)w^{n}-2\Delta i\sum_{k=1}^{m}C_{k}\sum_{n=-\infty}^{-1}e^{2i\xi_{k}(j-n)\Delta}w^{n}\\
 & = & \frac{\Delta}{\pi}\sum_{n=-\infty}^{-1}\hat{r}(2(n-j)\Delta)w^{n}-2\Delta i\sum_{k=1}^{m}C_{k}e^{2i\xi_{k}j\Delta}\sum_{n=-\infty}^{-1}\left(\frac{w}{e^{2i\xi_{k}n\Delta}}\right)^{n}\\
 & = & \Pi_{-}\tilde{r}(w)w^{j-1}-\sum_{k=1}^{m}\frac{c_{k}w_{k}^{j}}{w-w_{k}},\end{eqnarray*}
where \[
\tilde{r}(w)=\frac{\Delta}{\pi}w\sum_{n=-\infty}^{\infty}\hat{r}(2n\Delta)w^{n},\]
\[
w_{k}=e^{2i\xi_{k}\Delta},\]
 and \[
c_{k}=2\Delta iw_{k}C_{k}.\]
So we see that replacing the reduced scattering data $S$ by the reduced
discrete scattering data \[
\tilde{S}=(\tilde{r};w_{1},\dots,w_{k};c_{1},\dots,c_{k})\]
 is algorithmically equivalent to discretizing the Marchenko equation
in the above manner. In the discrete algorithm, we set \[
\mu_{j}=\frac{\gamma_{j}}{|\gamma_{j}|}\arctan|\gamma_{j}|\]
 where $\gamma_{j}=\hat{\tilde{h}}_{j\Delta}(1)$. This leads to the
approximation \[
q(j\Delta)\approx\Delta^{-1}\mu_{j},\]
 which is very close to the right hand side of (\ref{eq:approx 1})
when $\gamma_{j}$ is small.

In light of the above discussion, it may seem as though the discrete
Marchenko equation is nothing more than a simple discretization of
the continuum Marchenko equation. Such a discretization has been discussed
in the literature, for example see \cite{Frangos Jaggard}. However,
we need to consider the following subtlety, which explains why the
discrete theory is needed to obtain good results. Discretizing the
left and right continuum Marchenko equations, separately, in the above
sense, will not produce the correct pulse. Instead, one should first
replace the scattering data by discrete scattering data, in the above
manner, and then derive the data for the left equation. This will
guarantee that the resulting hard pulse has the correct scattering
data. In particular, the reflection coefficient corresponding to the
discrete potential will be a very good approximation to the original
reflection coefficient in a neighborhood of zero.

\chapter{Pulses with finite rephasing time and applications\label{cha:Pulses-with-finite}}

In most NMR applications the designed pulses have a fixed rephasing
time, $\rho$. In this chapter we derive a simple, SLR-type algorithm
for generating hard pulses with finitely many rephasing time steps.
We then describe several applications in NMR pulse design.

\section{A recursive algorithm for pulses of finite rephasing time}

Part (b) of Theorem \ref{thm:main discrete} tells us that designing
a hard pulse with a fixed number of rephasing time steps, $\rho$,
amounts to specifying a function, $r_{0}$, which is meromorphic in
the unit disk and vanishing at the origin. Once $r=w^{-\rho}r_{0}$
has been specified, one can, of course, generate the pulse using the
recursion described in Section \ref{sec:Discrete-Theory}. However,
there is a simpler, more direct recursive algorithm which can be used
in the case of finite rephasing time. This algorithm, which we derive
below, resembles the SLR algorithm (see \cite{PLNM}).

We are dealing with a potential of the form \[
q(t)=\sum_{j=-\infty}^{\rho-1}\mu_{j}\delta(t-j\Delta).\]
It is easy to check that\[
r_{0}=\lim_{j\rightarrow\infty}\frac{w^{\rho-j}B_{-,j}}{A_{-,j}}=\frac{B_{-,\rho}}{A_{-,\rho}}.\]
So let us set \[
R_{-,j}=\frac{B_{-,j}}{A_{-,j}}.\]
Notice that $R_{-,j}$ is a meromorphic function on the unit disk
which vanishes at the origin. The recursion (\ref{eq:drmd6}) induces
the following recursion on $R_{-,j}$:\[
R_{-,j}=\frac{1-\gamma_{j}w^{-1}R_{-,j+1}}{\gamma_{j}^{*}+w^{-1}R_{-,j+1}}.\]
Since $R_{-,j}$ vanishes at the origin, we must have \[
\gamma_{j}=\left(\left.\frac{R_{-,j+1}}{w}\right|_{w=0}\right)^{-1}.\]
Thus we can reconstruct the potential from the initial data $R_{-,\rho}=r_{0}$. 

For example, we can specify $r_{0}$ as the ratio of two polynomials:\[
r_{0}=R_{-,\rho}=\frac{P_{\rho}(w)}{Q_{\rho}(w)},\]
 with $P_{\rho}(w)=0$. Then, the recursion is simply\[
R_{-,j}=\frac{P_{j}(w)}{Q_{j}(w)}=\frac{Q_{j+1}-\gamma_{j}w^{-1}P_{j+1}}{\gamma_{j}^{*}Q_{j+1}+w^{-1}P_{j+1}},\]
 where $\gamma_{j}=\frac{\hat{Q}_{j+1}(0)}{\hat{P}_{j+1}(1)}.$ 

This very much resembles the SLR recursion. Notice, however, that
there is no need to choose polynomials $A$ and $B$ satisfying $|A|^{2}+|B|^{2}=1$
on the unit circle. As a result, the resulting pulses will generally
have infinite duration.

\section{Equiripple pulse design}

\subsection{\label{sub:IST equiripple}The IST method}

In this section, we describe a method for designing a pulse with a
fixed rephasing time $\rho<\infty$, which gives a profile which uniformly
approximates some real ideal magnetization profile. Suppose we are
given an ideal profile \[
\Mb_{\textrm{ideal}}(z)=\left[\begin{array}{c}
\frac{2r_{\textrm{ideal}}}{1+|r_{\textrm{ideal}}|^{2}}\\
0\\
\frac{1-|r_{\textrm{ideal}}|^{2}}{1+|r_{\textrm{ideal}}|^{2}}\end{array}\right](\frac{z}{2}),\]
 where $r_{\textrm{ideal}}:\RR\rightarrow\RR$ is a real reflection
coefficient. According to Theorem \ref{thm:main continuum}(b), we
should uniformly approximate $r_{\textrm{ideal}}$ by a real reflection
coefficient $r=e^{-2i\xi\rho}r_{0}$ where $r_{0}$ has a meromorphic
extension to the upper half plane and $\lim_{|\xi|\rightarrow\infty}r_{0}(\xi)=0$.
In fact, we should design $r_{0}$ to be analytic in the upper half
plane, so that the resulting pulse has minimum energy (see Corollary
\ref{cor:energy}). So the Fourier transform of $r$ should be supported
on the half ray $[-2\rho,\infty)$. Since $r$ is to be real, $\FF(r)$
should actually be supported on the symmetric interval $[-2\rho,2\rho].$
Therefore the problem reduces to uniformly approximating a real function
$r_{\textrm{ideal}}$ by a function whose Fourier transform is supported
on a given interval $[-2\rho,2\rho]$. To practically implement this
procedure it is best to work in the discrete theory, and use an algorithm
such as the Remez algorithm. 

Let us focus on the most typical example where $r_{\textrm{ideal}}=\chi_{[-1,1]}$,
which corresponds to a single slice selective $90^{\circ}$ pulse.
To use the Remez algorithm, the user would specify the time step $\Delta$,
and three of the following parameters:

(i) The rephasing time: $\rho$;

(ii) The transition width: $\tau$;

(iii) The in-slice ripple: $\delta_{1,\textrm{IST}}$;

(iv) The out-of-slice ripple: $\delta_{2,\textrm{IST}}$.

\noindent The unspecified of these four parameters can then be determined
using, for example, the parameter relations given in \cite{Digital}.
The Remez algorithm then produces a periodic function (of period $\frac{2\pi}{\Delta})$
which approximates $r_{\textrm{ideal}}$ with a maximum error $\delta_{1,\textrm{IST}}$
inside the interval $[-1,1]$ and a maximum error of $\delta_{2,\textrm{IST}}$
outside of the interval $[-1-\tau,1+\tau].$ The algorithm does not
attempt to control the function in the transition region $[-1-\tau,-1]\cup[1,1+\tau]$.
See Section \ref{sub:Comparison-SLR-IST} for plots of pulses obtained
using this method.

\subsection{\label{sub:SLR equiripple}The SLR method}

The SLR method is a procedure for designing a pulse with duration
$T$ such that the resulting flip angle profile approximates some
ideal flip angle profile. The duration is controlled by specifying
$B$ from part (c) of Theorem \ref{thm:main continuum}. The Fourier
transform of $B$ must be supported on $[0,2T]\subset\RR$, where
$T$ is the desired pulse duration. This function is designed so that
$1-2|B|^{2}$ approximates the $z$-component $M_{z}$ of the desired
magnetization profile (or the cosine of the flip angle profile). One
can then compute $A$ to be analytic and non-vanishing in the upper
half plane with $|A|^{2}=1-|B|^{2}$ on $\RR$. The reflection coefficient
is given by \[
r(\xi)=e^{-2\rho\xi i}\frac{B(\xi)}{A(\xi)}.\]
 By specifying the rephasing time, $\rho$, and the zeros of $A$
in the upper half plane, one has some limited control on the phase
of the transverse magnetization.

Let us focus on the case of a selective $90^{\circ}$ pulse where
$r_{\textrm{ideal}}=\chi_{[-1,1]}$ and $|B_{\textrm{ideal}}|=\frac{\sqrt{2}}{2}\chi_{[-1,1]}$.
Again, for practical purposes, it is best to work in the discrete
theory so that $B$ is a polynomial. To design this polynomial, the
Remez algorithm can be used with the following parameters:

(i) The rephasing time: $\rho$;

(ii) The transition width: $\tau$;

(iii) The in-slice ripple: $\delta_{1,\textrm{SLR}}$;

(iv) The out-of-slice ripple: $\delta_{2,\textrm{SLR}}$.

\noindent As before, three of these parameters are specified by the
user, and the fourth is determined by the parameter relations from
\cite{Digital}.

\subsection{Comparison of the SLR and IST methods for selective $90^{\circ}$
pulses\label{sub:Comparison-SLR-IST}}

In this section we compare the SLR and IST methods from Sections \ref{sub:IST equiripple}
and \ref{sub:SLR equiripple}. We make the comparison for various
values of the rephasing time, $\rho$, the transition width, $\tau$,
and the out-of-slice ripple, $\delta_{2,\textrm{trans}}$. Here $\delta_{2,\textrm{trans}}$
represents the maximum magnitude of the transverse magnetization,
$\Mb_{x}+i\Mb_{y}$, for out-of-slice frequencies (recall that $\Mb_{x}+i\Mb_{y}$
is ideally zero out-of-slice). One can check that $\delta_{2,\textrm{trans}}$
is related to $\delta_{2,\textrm{IST}}$ and $\delta_{2,\textrm{SLR}}$
by \begin{eqnarray}
\delta_{2,\textrm{trans}} & = & \frac{2\delta_{2,\textrm{IST}}}{1+\delta_{2,\textrm{IST}}^{2}}\label{eq:del 1}\\
\delta_{2,\textrm{trans}} & = & 2\delta_{2,\textrm{SLR}}\sqrt{1-\delta_{2,\textrm{SLR}}^{2}}.\label{eq:del 2}\end{eqnarray}
Given $\tau$, $\rho$, and $\delta_{2,\textrm{trans}}$, the in-slice
ripple $\delta_{1,\textrm{long}}$ is determined. This is defined
to be the maximum error in the longitudinal magnetization, $\Mb_{z}$,
for in-slice frequencies (recall that $\Mb_{z}$ is ideally zero in-slice).
One can check that $\delta_{1,\textrm{long}}$ is related to $\delta_{1,\textrm{IST}}$
and $\delta_{1,\textrm{SLR}}$ by \begin{eqnarray}
\delta_{1,\textrm{long}} & = & \frac{\delta_{1,\textrm{IST}}-\frac{1}{2}\delta_{1,\textrm{IST}}^{2}}{1-(\delta_{1,\textrm{IST}}-\frac{1}{2}\delta_{1,\textrm{IST}}^{2})}\label{eq:del 3}\\
\delta_{1,\textrm{long}} & = & 2\sqrt{2}\delta_{1,\textrm{SLR}}+2\delta_{1,\textrm{SLR}}^{2}.\label{eq:del 4}\end{eqnarray}

In Figures \ref{cap:90-3-01-20}, \ref{cap:90-2-10-20}, and \ref{cap:90-1-05-50}
we compare SLR and IST pulses for various values of the transition
width, $\tau$, the rephasing time, $\rho$, and the out-of-slice
ripple $\delta_{2,\textrm{trans}}$. We see that, in each case, the
inverse scattering pulse produces a better profile. Of course, the
IST pulses are somewhat longer in duration. In many applications,
however, the extra duration causes no problem, because the important
duration is often the duration of the portion of the pulse following
the peak (see \cite{Magland}).

\begin{figure}

\caption{\label{cap:90-3-01-20}Comparison of SLR and IST equiripple pulses:
$\rho=3$, $\delta_{2,\textrm{trans}}=0.01$, $\tau=0.2\times2\pi$. }

\begin{center}\includegraphics[%
  scale=0.7]{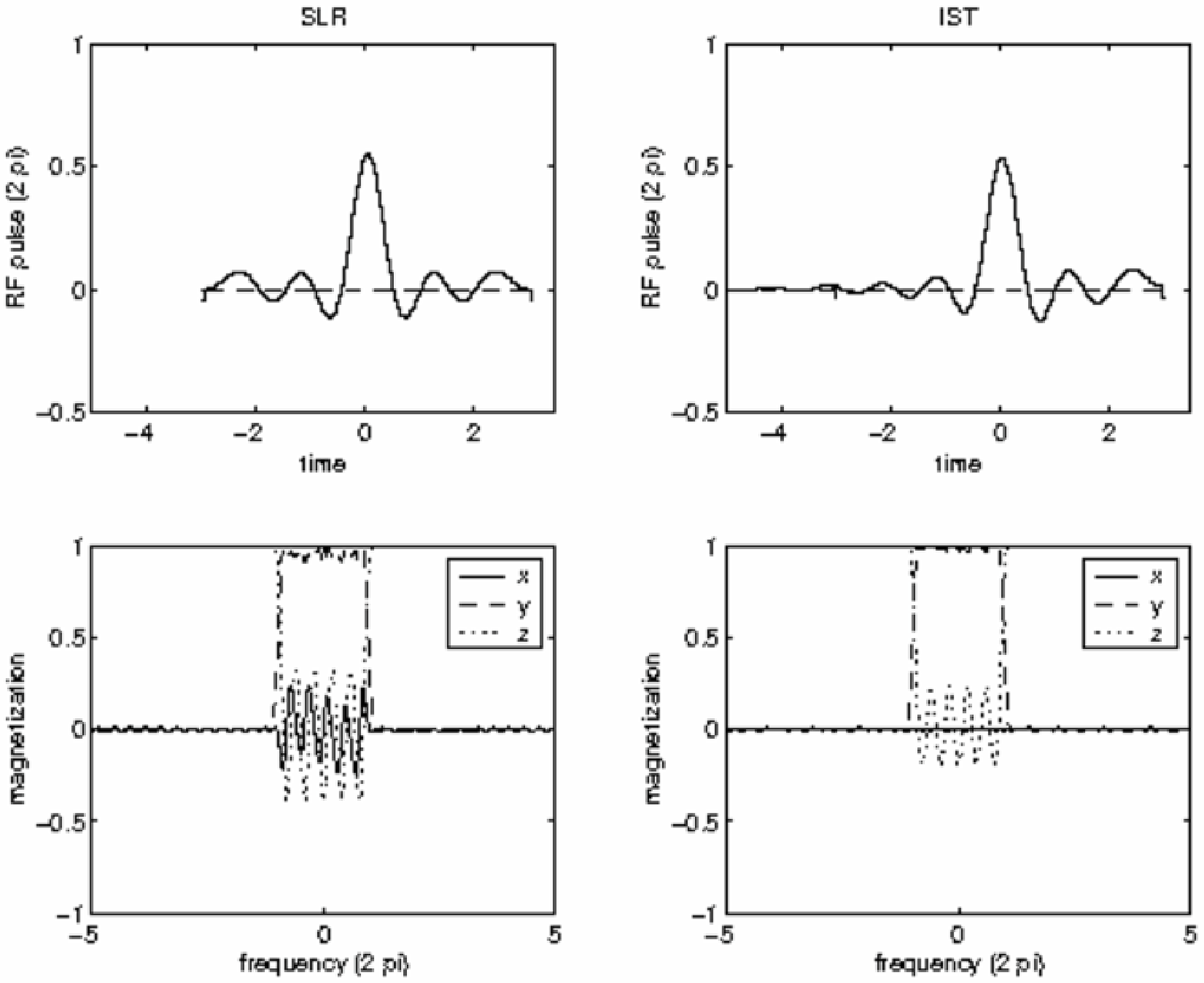}\end{center}
\end{figure}

\begin{figure}

\caption{\label{cap:90-2-10-20}Comparison of SLR and IST equiripple pulses
$\rho=2$, $\delta_{2,\textrm{trans}}=0.1$, $\tau=0.2\times2\pi$. }

\begin{center}\includegraphics[%
  scale=0.7]{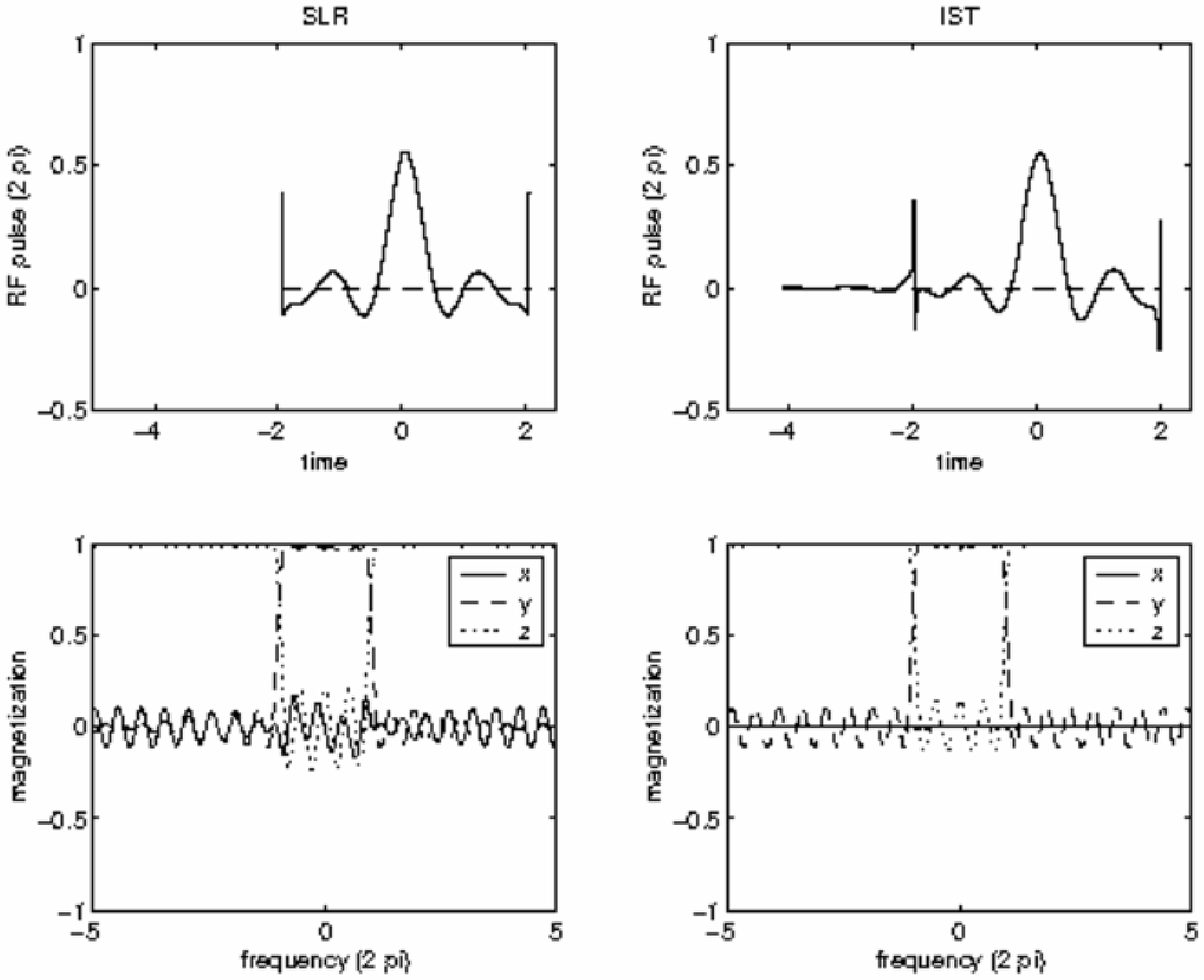}\end{center}
\end{figure}

\begin{figure}

\caption{\label{cap:90-1-05-50}Comparison of SLR and IST equiripple pulses:
$\rho=1$, $\delta_{2,\textrm{trans}}=0.05$, $\tau=0.5\times2\pi$. }

\begin{center}\includegraphics[%
  scale=0.7]{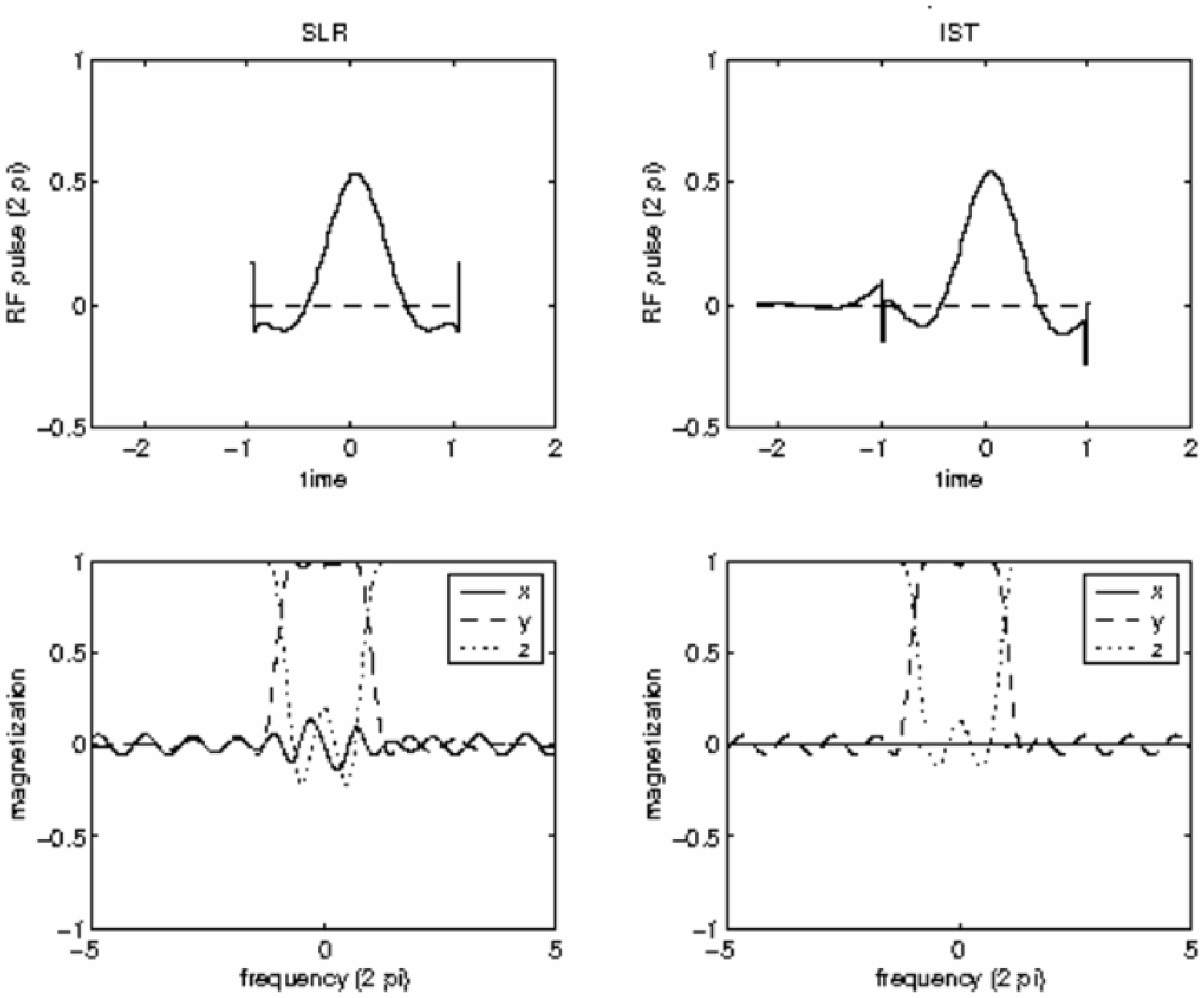}\end{center}
\end{figure}

\section{Self refocused pulse design}

Suppose we want to design a pulse with zero rephasing time ($\rho=0$).
Such a pulse is called a \textit{self refocused} pulse. According
to part (c) of Theorem \ref{thm:main continuum}, we should approximate
the ideal reflection coefficient, $r_{\textrm{ideal}}$, by a function
$r:\RR\rightarrow\CC$ which has a meromorphic extension to the upper
half plane with $\lim_{|z|\rightarrow\infty}r(z)=0$. For example,
$r$ could be a rational function with numerator degree strictly smaller
than denominator degree. Of course, there are many ways to approximate
$r_{\textrm{ideal}}$ in this way. The energy formula (\ref{eq:energy cor})
tells us that the energy of the resulting pulse depends on the locations
of the poles of $r$ in the upper half plane. If energy is a major
concern, then $r$ should be designed to have a small number of poles,
close to the real axis. In this way, there is a delicate trade-off
between the energy of the pulse and the accuracy of the approximation.
Another concern is the stability of the pulse under imperfect magnetic
field conditions. For example, in many applications it is desirable
for the pulse to maintain its selectivity when it is scaled by, say,
90\% or 110\%. 

In this section we describe one method for designing relatively low
energy, self refocused $90^{\circ}$ pulses. We need to approximate
$r_{\textrm{ideal}}=\chi_{[-2,2]}$ by a function $r$ which has a
meromorphic extension to the upper half plane. Let us write \[
r=\frac{e^{R}}{1+e^{R}},\]
 where $R:\RR\rightarrow\CC$ has an analytic extension to the upper
half plane. The idea is to choose $R$, so that $|e^{R}|$ is very
large in-slice and very small out of slice. Then $r$ will be close
to $r_{\textrm{ideal}}$. The magnitude of $e^{R}$ is determined
by the real part of $R$, so we should design $\re R$ to be a smooth
function of the form \[
\re R(\xi)=\begin{cases}
k_{1} & \textrm{if }|\xi|<2\\
-k_{2} & \textrm{if }|\xi|>2+2\tau,\end{cases}\]
 where $k_{1}$, $k_{2}$, and $\tau$ are positive numbers. The imaginary
part of $R$ is then chosen so that $R$ has an analytic extension
to the upper half plane. The parameters $k_{1}$ and $k_{2}$ control
the in-slice and out-of-slice errors, respectively, and $\tau$ is
the transition width. One can experiment with different values of
these parameters to obtain a variety of pulses, with different energies.
We plot one of these pulses in figure \ref{cap:sr90}.

\begin{figure}

\caption{\label{cap:sr90}A self refocused pulse and its resulting magnetization.}

\includegraphics[%
  scale=0.7]{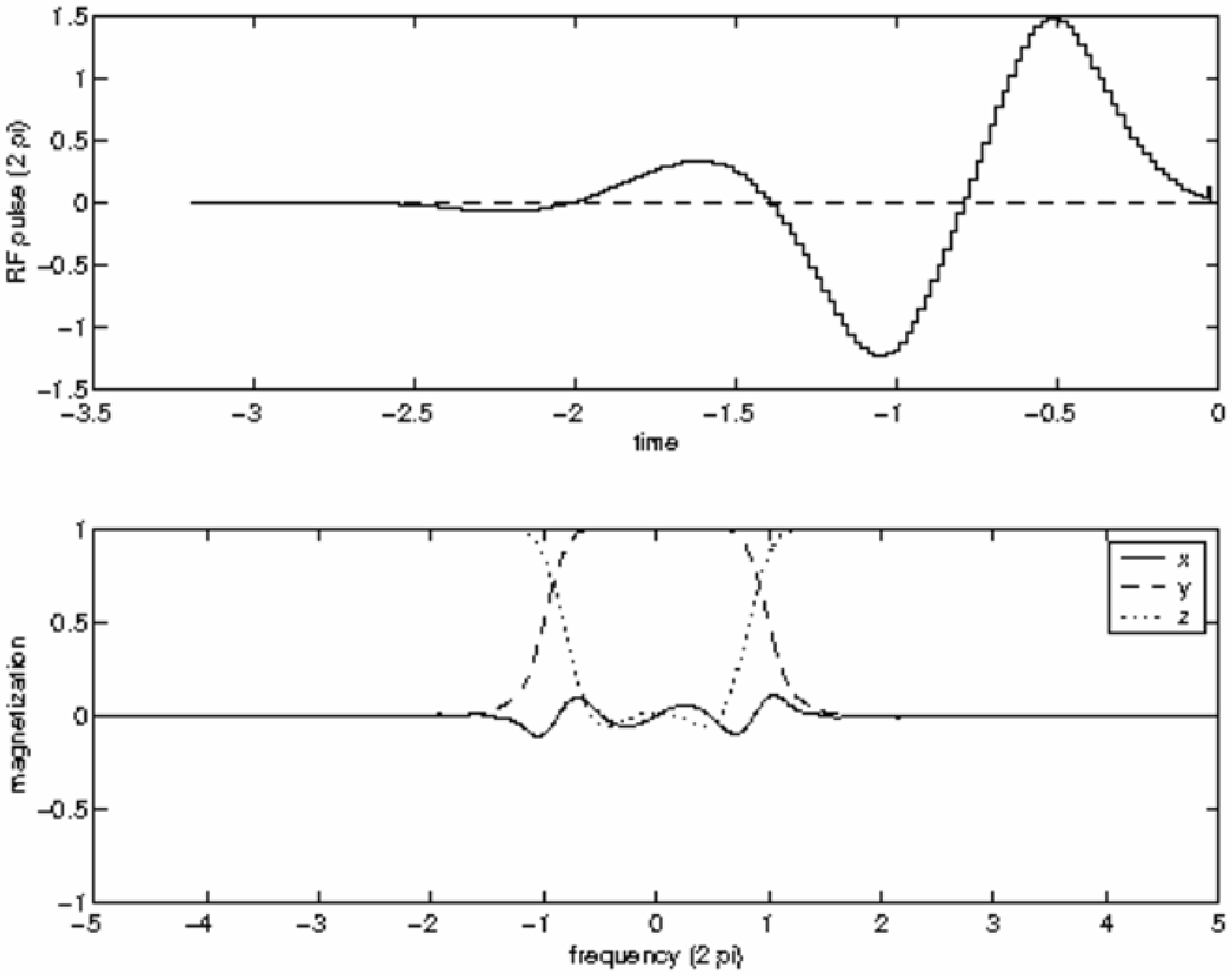}
\end{figure}

Remark: Notice that the $r$ designed in the previous paragraph does
not satisfy $\lim_{|\xi|\rightarrow\infty}r(\xi)=0$. This can be
easily remedied by subtracting an appropriate, very small constant.

\section{Half pulse design}

In some applications it is only necessary to approximate the $x$-component
of the ideal magnetization profile using a self refocused pulse. For
example, see \cite{Half}. Let $\Mb_{\textrm{ideal},x}:\RR\rightarrow\RR$
be an ideal $x$-magnetization profile. We want to find a function
$r$, which has a meromorphic extension to the upper half plane with
$\lim_{|z|\rightarrow\infty}r(z)=0$ such that \[
\frac{2\re r}{1+|r|^{2}}\approx\Mb_{x}.\]
 The trick is to set \[
r=\frac{1-R}{1+R},\]
 where \[
R=\frac{1-r}{1+r}.\]
Then we have \begin{eqnarray*}
\frac{2\re r}{1+|r|^{2}} & = & \frac{r+r^{*}}{1+r^{*}r}\\
 & = & \frac{\frac{1-R}{1+R}+\frac{1-R^{*}}{1+R^{*}}}{1+\frac{1-R^{*}}{1+R^{*}}\frac{1-R}{1+R}}\\
 & = & \frac{(1-R)(1+R^{*})+(1+R)(1-R^{*})}{(1+R)(1+R^{*})+(1-R)(1-R^{*})}\\
 & = & \frac{1-|R|^{2}}{1+|R|^{2}}.\end{eqnarray*}
Therefore, we should design $R$ so that \[
|R|=\sqrt{\frac{1-\Mb_{x}}{1+\Mb_{x}}},\]
 where $\Mb_{x}$ is a perhaps smoothed out version of $\Mb_{\textrm{ideal},x}$.
The above calculations lead to the following

\begin{prop}
Let $\Mb_{x}:\RR\rightarrow\RR$ be an $x$-magnetization profile
with sufficient smoothness and decay, and assume that $|\Mb_{x}(z)|<1$
for all $z\in\RR$. Then there exist infinitely many self refocused
pulses $\omega:\RR\rightarrow\CC$ such that $\TT\omega=\Mb$. These
pulses are parameterized by the poles of $\frac{1-r}{1+r}$ in the
upper half plane, where $r$ is the reflection coefficient. 
\end{prop}
\begin{rem}
The special case of no poles is often a minimum energy pulse. For
example, this is the case whenever $\Mb_{x}$ is non-negative. Indeed
then we have $|R|\leq1$, which implies, by the maximum modulus principle,
that $R$ is never equal to $-1$ in the upper half plane, which is
equivalent to saying that $r=\frac{1-R}{1+R}$ has no poles in the
upper half plane.

Figure \ref{cap:half} shows a pulse designed using this method.

\begin{figure}

\caption{\label{cap:half}A half pulse and its resulting magnetization}

\begin{center}\includegraphics[%
  scale=0.7]{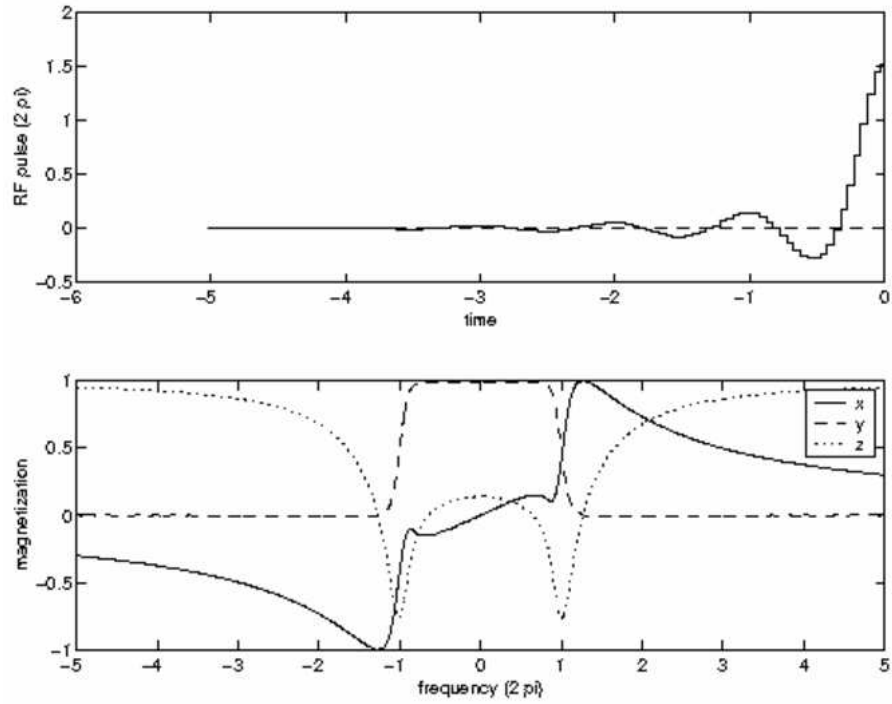}\end{center}
\end{figure}

\end{rem}

\chapter{Conclusion}

We have seen that the hard pulse approximation (the approximation
of an RF-pulse by a sum of $\delta$-functions) leads to a discrete
scattering theory which is completely analogous to the standard continuum
scattering theory for the ZS-system. We introduced the DIST algorithm,
a recursive algorithm for solving the full discrete inverse scattering
problem relating to NMR pulse design, and we explained how this algorithm
could be used to efficiently approximate the continuum inverse scattering
transform. In the past, numerical techniques have been used to approximate
the solutions to the Marchenko equations. In this thesis we have provided
a more exact method of pulse design, which involves replacing the
standard scattering data by, so called, discrete scattering data.

The case of pulses with finite rephasing time is particularly simple,
and useful in practice. We described some new applications to NMR
pulse design. Specifically we explained how to maintain control on
the phase of the magnetization profile while producing equiripple
pulses, self-refocused pulses, and half pulses.

\appendix

\chapter{Notation and background material}

\subsection{Functions on the real line and the unit circle}

In this thesis, we work with complex functions on the real line, $\RR$,
and the unit circle, $S^{1}$. If $f:\RR\rightarrow\CC$ is a function
on the real line, then we say that $f$ has an analytic (meromorphic)
extension to the upper half plane, $\HH$, if there exists an analytic
(meromorphic) function $\tilde{f}$ on $\HH$ such that \[
\lim_{y\rightarrow0^{+}}\tilde{f}(x+iy)=f(x)\]
 for almost every $x\in\RR$. Such an extension is necessarily unique.
In this case we abuse notation and let $f$ also denote this extension
$\tilde{f}$. Similarly, if $f:S^{1}\rightarrow\CC$ is a function
on the unit circle, then we say that $f$ has an analytic (meromorphic)
extension to the unit disk, $\DD$, if there exists an analytic (meromorphic)
function $\tilde{f}$ on $\DD$ such that the radial limits exist
almost everywhere, and coincide almost everywhere with $f$. Again,
we let $f$ denote both the function on $S^{1}$ and its extension
to the unit disk.

When they are not being used to represent complex numbers, the symbols
$\xi$ and $w$ will denote the identity functions on $\RR$ and $S^{1}$,
respectively. For example, if $r:\RR\rightarrow\CC$ is a function
on the real line, then so is $\xi r$. Furthermore, if $r$ has a
meromorphic extension to the upper half plane, then so does $\xi r$.

Let $\alpha^{*}$ denote the complex conjugate of the complex number
$\alpha$. If $f$ is a function on the real line, we let $f^{*}$
denote the complex conjugate of $f$. If $f$ has an analytic (meromorphic)
extension to the upper half plane, then we consider $f^{*}$ also
as an analytic (meromorphic) function on the lower half plane given
by \[
f^{*}(\xi)=f(\xi^{*})^{*}.\]
 We follow similar conventions if $f$ is a function on $S^{1}$.
Specifically, if $f$ has an analytic (meromorphic) extension to $\DD$,
then $f^{*}$ is considered as an analytic (meromorphic) function
on $\CCh\setminus\DD$ given by \[
f^{*}(w)=f(\frac{1}{w^{*}})^{*}.\]
Here, $\CCh$ is the Riemann sphere.

\subsection{The Fourier transform}

For $k>0$, we let $L^{k}(\RR)$ denote the space of equivalence classes
of measurable functions $f:\RR\rightarrow\CC$ for which the quantity
\[
\left\Vert f\right\Vert _{L^{k}}:=(\int_{-\infty}^{\infty}|f(z)|^{k}dz)^{\frac{1}{k}}\]
 is finite. The functions $f$ and $g$ in $L^{k}(\RR)$ are considered
to be equivalent if $\left\Vert f-g\right\Vert _{L^{k}}=0$. The space
$L^{2}(\RR)$ is a Hilbert space with respect to the inner product
\[
\left\langle f,g\right\rangle :=\int_{-\infty}^{\infty}f(z)g^{*}(z)dz.\]
We let $\FF:L^{2}(\RR)\rightarrow L^{2}(\RR)$ denote the \textit{Fourier
transform}, and we write $\FF(f)=\hat{f}$. For integrable $f$, the
Fourier transform takes the form\[
\hat{f}(t)=\FF(f)(t):=\int_{-\infty}^{\infty}f(z)e^{-itz}dz.\]
The Fourier transform is a unitary map, meaning it is a linear isomorphism
which preserves the inner product, up to a factor of $2\pi$. We let
$\FF^{-1}$ denote the inverse of $\FF$. For integrable $f$ we have
\[
\check{f}(z)=\FF^{-1}(f)(z)=\frac{1}{2\pi}\int_{-\infty}^{\infty}f(t)e^{itz}dt.\]

Similarly, we let $L^{k}(S^{1})$ denote the space of measurable functions
$f:S^{1}\rightarrow\CC$ for which the quantity \[
\left\Vert f\right\Vert _{L^{k}}:=(\frac{1}{2\pi}\int_{0}^{2\pi}|f(e^{i\theta})|^{k}d\theta)^{\frac{1}{k}}\]
 is finite. The functions $f$ and $g$ in $L^{k}(S^{1})$ are considered
to be equivalent if $\left\Vert f-g\right\Vert _{L^{k}}=0$. The space
$L^{2}(S^{1})$ is a Hilbert space with respect to the inner product
\[
\left\langle f,g\right\rangle :=\frac{1}{2\pi}\int_{0}^{2\pi}f(e^{i\theta})g^{*}(e^{i\theta})dz.\]
We let $\FF:L^{2}(S^{1})\rightarrow\hat{L}^{2}(S^{1})$ denote the
\textit{discrete Fourier transform}, and we write $\FF(f)=\hat{f}$.
Here $\hat{L}^{2}(S^{1})$ is the Hilbert space of sequences $\hat{f}:\ZZ\rightarrow\CC$
for which \[
\left\Vert \hat{f}\right\Vert _{\hat{L}^{2}}:=\sum_{j=-\infty}^{\infty}|\hat{f}(j)|^{2}\]
 is finite. The inner product on $\hat{L}^{2}(S^{1})$ is given by
\[
\left\langle \hat{f},\hat{g}\right\rangle :=\sum_{j=-\infty}^{\infty}\hat{f}(j)\hat{g}(j)^{*}.\]
Explicitly, we have\[
\hat{f}(j)=\FF(f)(j):=\frac{1}{2\pi}\int_{0}^{2\pi}f(e^{i\theta})e^{-ij\theta}d\theta.\]
 The discrete Fourier transform is a unitary map, meaning it is a
linear isomorphism which preserves the inner product, up to a factor
of $2\pi$. Again, we let $\FF^{-1}$ denote its inverse. For absolutely
summable sequences $S=\left\{ a_{j}\right\} _{j\in\ZZ}$ we have \[
\FF^{-1}(S)(e^{i\theta})=\sum_{j=-\infty}^{\infty}a_{j}e^{ij\theta}.\]

\subsection{The Sobolev Space $H^{k}$\label{sub:The-Sobolev-Spaces}}

We say that $f:\RR\rightarrow\CC$ has a \textit{weak derivative}
$f':\RR\rightarrow\CC$ if for all smooth functions $\phi:\RR\rightarrow\CC$
with compact support we have \[
\int_{-\infty}^{\infty}f(x)\phi'(x)dx+\int_{-\infty}^{\infty}f'(x)\phi(x)dx=0.\]
The weak derivative of a function on $S^{1}$ is defined similarly.

Let $\Lambda=\RR\textrm{ or }S^{1}$, and let $k$ be a non-negative
integer. The \textit{Sobolev space} $H^{k}(\Lambda)\subset L^{2}(\Lambda)$
consists of those functions which have $k$ weak derivatives in $L^{2}$.
The spaces $H^{k}(\RR)$ and $H^{k}(S^{1})$ are Hilbert spaces with
respect to the inner products\[
\left\langle f,g\right\rangle :=\sum_{j=0}^{k}\int_{-\infty}^{\infty}\partial^{j}f(x)\partial^{j}g^{*}(x)dx\]
 and \[
\left\langle f,g\right\rangle :=\sum_{j=0}^{k}\frac{1}{2\pi}\int_{0}^{2\pi}\partial^{j}f(e^{i\theta})\partial^{j}g^{*}(e^{i\theta})d\theta,\]
 respectively. Notice that $H^{0}(\Omega)=L^{2}(\Omega).$ 

These spaces are particularly simple in the Fourier domain. Let us
define \[
\hat{H}^{k}(\RR):=\left\{ u\in L^{2}(\RR):\;\int_{-\infty}^{\infty}(1+|x|)^{2k}|u(x)|^{2}dx<\infty\right\} \]
 with the inner product \[
\left\langle u,v\right\rangle :=\frac{1}{2\pi}\int_{-\infty}^{\infty}(1+x^{2}+\dots+x^{2k})u(x)v^{*}(x)dx\]
and \[
\hat{H}^{k}(S^{1}):=\left\{ u:\ZZ\rightarrow\CC:\;\sum_{n=-\infty}^{\infty}(1+|n|)^{2k}|u(n)|^{2}<\infty\right\} \]
with the inner product \[
\left\langle u,v\right\rangle :=\sum_{n=-\infty}^{\infty}(1+n^{2}+\dots+n^{2k})u(n)v(n)^{*}.\]

\begin{fact}
\label{fac:Sob 3}Let $\Lambda=\RR\textrm{ or }S^{1}$, and let $f$
be in $H^{1}(\Lambda)$. Then $\left\Vert f\right\Vert _{L^{\infty}}\leq\left\Vert f\right\Vert _{H^{1}}.$
\end{fact}
\begin{proof}
By the Cauchy-Schwartz inequality we have \begin{eqnarray*}
|f(x)|^{2} & = & \left|2\int_{-\infty}^{x}f'(y)f(y)dy\right|\\
 & \leq & 2\int_{-\infty}^{x}|f'(y)||f(y)|dy\\
 & \leq & 2\left\langle f',f\right\rangle \\
 & \leq & 2\left\Vert f\right\Vert _{L^{2}}\left\Vert f'\right\Vert _{L^{2}}\\
 & \leq & \left\Vert f\right\Vert _{L^{2}}^{2}+\left\Vert f'\right\Vert _{L^{2}}^{2}\\
 & = & \left\Vert f\right\Vert _{H^{1}}^{2}.\end{eqnarray*}

\end{proof}
\begin{fact}
\label{fac:Sob 4}Let $\Lambda=\RR\textrm{ or }S^{1}$. If $f$ and
$g$ are in $H^{1}(\Lambda)$, then the product $fg$ is also in $H^{1}(\Lambda)$,
and \[
\left\Vert fg\right\Vert _{H^{1}}\leq2\left\Vert f\right\Vert _{H^{1}}\left\Vert g\right\Vert _{H^{1}}.\]

\end{fact}
\begin{proof}
We have \begin{eqnarray*}
\left\Vert fg\right\Vert _{H^{1}}^{2} & = & \left\Vert fg\right\Vert _{L^{2}}^{2}+\left\Vert (fg)'\right\Vert _{L^{2}}^{2}=\left\Vert fg\right\Vert _{L^{2}}^{2}+\left\Vert f'g+fg'\right\Vert _{L^{2}}^{2}\\
 & \leq & \left\Vert f\right\Vert _{L^{\infty}}^{2}\left\Vert g\right\Vert _{L^{2}}^{2}+2\left\Vert f'g\right\Vert _{L^{2}}^{2}+2\left\Vert fg'\right\Vert _{L^{2}}^{2}\\
 & \leq & \left\Vert f\right\Vert _{L^{\infty}}^{2}\left\Vert g\right\Vert _{L^{2}}^{2}+2\left\Vert f'\right\Vert _{L^{2}}^{2}\left\Vert g\right\Vert _{L^{\infty}}^{2}+2\left\Vert f\right\Vert _{L^{\infty}}^{2}\left\Vert g'\right\Vert _{L^{2}}^{2}\\
 & \leq & 2\left\Vert f\right\Vert _{L^{\infty}}^{2}(\left\Vert g\right\Vert _{L^{2}}^{2}+\left\Vert g'\right\Vert _{L^{2}}^{2})+2\left\Vert f'\right\Vert _{L^{2}}^{2}\left\Vert g\right\Vert _{L^{\infty}}^{2}\\
 & \leq & 4\left\Vert f\right\Vert _{H^{1}}^{2}\left\Vert g\right\Vert _{H^{1}}^{2}.\end{eqnarray*}
The last inequality uses Fact \ref{fac:Sob 3}. 
\end{proof}
\begin{fact}
\label{fac:Sob 1}Let $\Lambda=\RR\textrm{ or }S^{1}$. The Fourier
transform is a unitary map from $H^{k}(\Lambda)$ onto $\hat{H}^{k}(\Lambda)$.
\end{fact}

\subsection{Self-adjoint operators}

A bounded operator $A:H\ra H$ on a complex Hilbert space is called
\textit{self-adjoint if} \[
\left\langle Av,w\right\rangle =\left\langle v,Aw\right\rangle \;\;\;\forall v,w\in H.\]
It is easy to verify that for such an operator $\left\langle Av,v\right\rangle $
is real for every $v\in H$. We say that $A$ is \textit{positive}
if \[
\left\langle Av,v\right\rangle >0\]
 for all nonzero $v\in H$.

The following fact can be found in \cite{Lax}.

\begin{fact}
\label{pro:norm self-adjoint}If $A:H\ra H$ is a bounded, self-adjoint
operator on a complex Hilbert space, then its operator norm is given
by \[
\left\Vert A\right\Vert =\sup_{|v|=1}|\left\langle Av,v\right\rangle |.\]

\end{fact}
The \textit{adjoint} $A^{*}$ of a bounded operator $A:H\ra H$ on
a complex Hilbert space is defined by \[
\left\langle Av,w\right\rangle =\left\langle v,A^{*}w\right\rangle \;\;\;\forall v,w\in H.\]
 The following two facts are easy to verify.

\begin{fact}
\label{fac:adjoint pair}If $A$ and $A^{*}$ are bounded operators
on a complex Hilbert space, then the operator $A^{*}A$ is self-adjoint
and positive. 
\end{fact}
$\;$

\begin{fact}
Let $A:X\ra X$ be an operator on a Banach space. If $\left\Vert A\right\Vert =\alpha<1$,
then $1+A$ is invertible.
\end{fact}
$\;$

\begin{lem}
\label{lem:minus gam2}Suppose that $A:H\rightarrow H$ is bounded,
self-adjoint and positive. Then \[
\left\Vert A-\frac{\gamma}{2}\right\Vert \leq\frac{\gamma}{2},\]
 whenever $\gamma\geq\left\Vert A\right\Vert $.
\end{lem}
\begin{proof}
This follows from Fact \ref{pro:norm self-adjoint} and the estimate\[
\left|\left\langle v,(A-\frac{\gamma}{2})v\right\rangle \right|=\left|\left\langle v,Av\right\rangle -\frac{\gamma}{2}\left\langle v,v\right\rangle \right|\leq\frac{\gamma}{2}\left\langle v,v\right\rangle ,\]
 which holds since \[
0\leq\left\langle v,Av\right\rangle \leq\left\Vert A\right\Vert \left\langle v,v\right\rangle \leq\gamma\left\langle v,v\right\rangle .\]

\end{proof}
\begin{prop}
\label{pro:1+A invertible} Suppose that $A:H\rightarrow H$ is a
bounded, positive self-adjoint operator on a complex Hilbert space.
Then $1+A$ is invertible, and \[
\left\Vert (1+A)^{-1}\right\Vert \leq1.\]

\end{prop}
\begin{proof}
The following method was used by Epstein in \cite{Epstein}. The trick
is to set $\gamma=\left\Vert A\right\Vert $ and write \[
1+A=\frac{2+\gamma}{2}(1+B),\]
 where \[
B=\frac{2}{2+\gamma}(A-\frac{\gamma}{2}).\]
Lemma \ref{lem:minus gam2} tells us that \[
\left\Vert B\right\Vert \leq\frac{2}{2+\gamma}\cdot\frac{\gamma}{2}=\frac{\gamma}{2+\gamma}<1,\]
 and so $1+B$ is invertible. Clearly, $1+A$ is also invertible.

Given $v\in H$, we know that \[
|(1+B)v|\geq(1-\frac{\gamma}{2+\gamma})|v|,\]
 and hence \begin{eqnarray*}
|(1+A)v| & = & \frac{2+\gamma}{2}|(1+B)v|\\
 & \geq & \frac{2+\gamma}{2}(1-\frac{\gamma}{2+\gamma})|v|\\
 & = & |v|,\end{eqnarray*}
 which proves that \[
\left\Vert (1+A)^{-1}\right\Vert \leq1.\]

\end{proof}

\chapter{The error from softening a pulse\label{sub:The-error-from-softening-a-pulse}}

As mentioned in the introduction, it is necessary, in practice, to
replace a given hard pulse \[
\Omega_{\textrm{hard}}(t)=\sum_{j=-\infty}^{\infty}\omega_{j}\delta(t-j\Delta)\]
 by a softened version \[
\Omega_{\textrm{soft}}(t)=\sum_{j=-\infty}^{\infty}\frac{\omega_{j}}{\Delta}\chi_{[0,\Delta)}(t-j\Delta).\]
 In this section we estimate the difference between the magnetization
profiles $\Mb_{\textrm{hard}}$ and $\Mb_{\textrm{soft}}$ resulting
from $\Omega_{\textrm{hard}}$ and $\Omega_{\textrm{soft}}$, respectively.

Fix a frequency $z\in\RR$. Let $M_{-,\textrm{hard}}(z;\Delta j)$
and $M_{-,\textrm{soft}}(z;\Delta j)$ denote the magnetizations at
time $t=\Delta j$ (or the $j^{\textrm{th}}$ time step). These are
normalized by \[
\lim_{j\rightarrow-\infty}M_{-,\textrm{hard}}(z;\Delta j)=\lim_{j\rightarrow-\infty}M_{-,\textrm{soft}}(z;\Delta j)=\left[\begin{array}{c}
0\\
0\\
1\end{array}\right].\]
 We are interested in the difference between $M_{-,\textrm{hard}}(z;j\Delta)$
and $M_{-,\textrm{soft}}(z;j\Delta)$ for large $j$. 

Let us focus on the error introduced at the $j^{\textrm{th}}$ time
step. Without loss of generality, we can assume that $\omega_{j}$
is real and positive. For the hard pulse, $M_{-,\textrm{hard}}(z;j+1)$
is obtained from $M_{-,\textrm{hard}}(z;j)$ by a rotation of $\omega_{j}$
radians around the $y$-axis, followed by a rotation of $\Delta z$
radians around the $z$-axis. On the other hand, one can check that
$M_{-,\textrm{soft}}(z;j+1)$ is obtained from $M_{-,\textrm{soft}}(z;j)$
by a single rotation of $\gamma=\sqrt{\omega_{j}^{2}+(\Delta z)^{2}}$
radians around the $\left[\begin{array}{c}
0\\
\omega_{j}/\gamma\\
\Delta z/\gamma\end{array}\right]$-axis. The difference between these rotations is a rotation of the
sphere (around some axis) of some number $\theta_{j}$ of radians.
If we can estimate $\theta_{j}$, then we can estimate the maximum
spherical error introduced at the $j^{\textrm{th}}$ time step. 

Let $R_{y}$ be rotation by $\omega_{j}$ radians around the $y$-axis,
let $R_{z}$ be rotation by $\Delta z$ radians around the $z$-axis,
and let $R_{yz}$ be rotation by $\gamma$ radians around the $\left[\begin{array}{c}
0\\
\omega_{j}/\gamma\\
\Delta z/\gamma\end{array}\right]$-axis, where \[
\gamma=\sqrt{\omega_{j}^{2}+(\Delta z)^{2}}.\]
The question is: How well does $R_{yz}$ approximate $R_{z}R_{y}$? 

These rotations are represented in $SU_{2}$ by \begin{eqnarray*}
R_{y} & = & \left[\begin{array}{cc}
\cos\frac{\omega_{j}}{2} & \sin\frac{\omega_{j}}{2}\\
-\sin\frac{\omega_{j}}{2} & \cos\frac{\omega_{j}}{2}\end{array}\right]\\
R_{z} & = & \left[\begin{array}{cc}
\cos\frac{\Delta z}{2}+i\sin\frac{\Delta z}{2} & 0\\
0 & \cos\frac{\Delta z}{2}-i\sin\frac{\Delta z}{2}\end{array}\right]\\
R_{yz} & = & \left[\begin{array}{cc}
\cos\frac{\gamma}{2}+\frac{\Delta z}{\gamma}i\sin\frac{\gamma}{2} & \frac{\omega_{j}}{\gamma}\sin\frac{\gamma}{2}\\
-\frac{\omega_{j}}{\gamma}\sin\frac{\gamma}{2} & \cos\frac{\gamma}{2}-\frac{\Delta z}{\gamma}i\sin\frac{\gamma}{2}\end{array}\right].\end{eqnarray*}
Using a symbolic mathematics computer program, we compute the real
part of the upper left component of the matrix $R_{yz}^{-1}R_{z}R_{y}\in SU_{2}$
to be \begin{eqnarray*}
\phi(\omega_{j},\Delta z): & = & \cos(\frac{\omega_{j}}{2})\cos(\frac{\Delta z}{2})\cos(\frac{\gamma}{2})+\frac{\omega_{j}}{\gamma}\sin(\frac{\omega_{j}}{2})\cos(\frac{\Delta z}{2})\sin(\frac{\gamma}{2})\\
 &  & +\frac{\Delta z}{\gamma}\cos(\frac{\omega_{j}}{2})\sin(\frac{\Delta z}{2})\sin(\frac{\gamma}{2}),\end{eqnarray*}
 which implies that $R_{yz}^{-1}R_{z}R_{y}$ is a rotation of $2\cos^{-1}\phi(\omega_{j},\Delta z)$
radians around some axis. From the Taylor expansion \[
\left(\cos^{-1}\phi(\alpha,\beta)\right)^{2}=\frac{1}{16}\omega_{j}^{2}(\Delta z)^{2}-\frac{1}{288}\omega_{j}^{4}(\Delta z)^{2}-\frac{1}{288}\omega_{j}^{2}(\Delta z)^{4}+\dots,\]
 we see that, for reasonably small $\omega_{j}$ and $\Delta z$,
$R_{yz}$ is extremely close to $R_{z}R_{y}$ composed with a rotation
of $\frac{\Delta}{2}|\omega_{j}z|$ radians. In fact, numerical evidence
suggests that $2\cos^{-1}\phi(\omega_{j},\Delta z)\leq\frac{\Delta}{2}|\omega_{j}z|$
for all $\omega$ and $\Delta z$. 

\begin{rem}
The above conclusion would be the same if we replaced the axis of
rotation of $R_{y}$ by some other axis orthogonal to the $z$-axis.
\end{rem}
Applying the above result, the maximum error between $\Mb_{\textrm{hard}}(z)$
and $\Mb_{\textrm{soft}}(z)$ introduced at the $j^{th}$ step is
conjectured to be a rotation of at most $\frac{\Delta}{2}|\omega_{j}z|$
radians, which would indicate a total error of at most\[
E=\frac{\Delta|z|}{2}\sum_{j=-\infty}^{\infty}|\omega_{j}|=\frac{\Delta|z|}{2}\int_{-\infty}^{\infty}|\Omega_{\textrm{soft}}(t)|dt\]
 radians on the sphere. Thus we see that by reducing the time step,
$\Delta$, the error introduced by softening a pulse can be made as
small as we like over a fixed frequency interval.

\chapter{The case of non-simple zeros\label{app:non-simple}}

In this section we consider the case where the poles of $a$ in the
upper half plane (or unit disk) are not necessarily simple. We simply
outline the idea, and leave most of the details to the interested
reader. 

In the case of simple zeros, the bound state data for the continuum
scattering transform is encapsulated in the rational function \[
Q_{+,0}(\xi)=\sum_{k=1}^{m}\frac{C_{k}}{\xi-\xi_{k}}.\]
For each $t\in\RR$, the rational functions \begin{eqnarray*}
Q_{+,t}(\xi) & = & \sum_{k=1}^{m}\frac{C_{k}e^{2i\xi_{k}t}}{\xi-\xi_{k}}\\
Q_{-,t}(\xi) & = & \sum_{k=1}^{m}\frac{\tilde{C}_{k}e^{-2i\xi_{k}t}}{\xi-\xi_{k}}\end{eqnarray*}
 can be uniquely determined from $Q_{+,0}$ using the following properties:

\begin{enumerate}
\item $Q_{+,t}$ has the form $Q_{+,t}(\xi)=\sum_{k=1}^{m}\frac{C_{k,t}}{\xi-\xi_{k}}$; 
\item $Q_{-,t}$ has the form $Q_{-,t}(\xi)=\sum_{k=1}^{m}\frac{\tilde{C}_{k,t}}{\xi-\xi_{k}}$; 
\item $Q_{+,t}=\Pi_{-}Q_{+,s}e^{2i\xi(t-s)}$ for all $t>s$;
\item $Q_{-,t}=\Pi_{-}Q_{-,s}e^{-2i\xi(t-s)}$ for all $t<s$;
\item The function $aQ_{-,t}Q_{+,t}+\frac{1}{a}$ is analytic in the upper
half plane.
\end{enumerate}
The fifth property can be verified by observing that the residue of
$aQ_{-,t}Q_{+,t}+\frac{1}{a}$ at $\xi_{k}$ is \[
C_{k}\tilde{C}_{k}a^{\prime}(\xi_{k})+\frac{1}{a^{\prime}(\xi_{k})}=0.\]

Let us now handle the case where $a$ has non-simple zeros. Let $\xi_{1},\dots,\xi_{m}$
be complex numbers in the upper half plane, and let $d_{1},\dots,d_{m}$
be positive integers (representing the multiplicities). In this case,
we specify the bound state data by defining the rational function\[
Q_{+,0}(\xi)=\sum_{k=1}^{m}\frac{P_{k}(\xi)}{(\xi-\xi_{k})^{d_{k}}},\]
 where $P_{k}$ is a polynomial of degree $d_{k}-1$ which does not
vanish at $\xi_{k}$. For each $t\in\RR$, the rational functions
$Q_{+,t}$ and $Q_{-,t}$ are defined by the following properties:

\begin{enumerate}
\item $Q_{+,t}$ has the form $Q_{+,t}(\xi)=\sum_{k=1}^{m}\frac{P_{k,t}(\xi)}{\xi-\xi_{k}}$,
where $P_{k,t}$ is a polynomial of degree $k-1$; 
\item $Q_{-,t}$ has the form $Q_{-,t}(\xi)=\sum_{k=1}^{m}\frac{\tilde{P}_{k,t}}{\xi-\xi_{k}},$where
$\tilde{P}_{k,t}$ is a polynomial of degree $k-1$; 
\item $Q_{+,t}=\Pi_{-}Q_{+,s}e^{2i\xi(t-s)}$ for all $t>s$;
\item $Q_{-,t}=\Pi_{-}Q_{-,s}e^{-2i\xi(t-s)}$ for all $t<s$;
\item The function $aQ_{-,t}Q_{+,t}+\frac{1}{a}$ is analytic in the upper
half plane.
\end{enumerate}
Once these functions have been determined, one can define the right
and left Marchenko equations using\begin{eqnarray*}
r_{t} & = & \Pi_{-}re^{2i\xi t}-Q_{+,t}\\
s_{t} & = & \Pi_{-}se^{-2i\xi t}-Q_{-,t}.\end{eqnarray*}

A similar method can be used in the discrete case.

\chapter{Explicit implementation of DIST}

Let us write down some explicit formulas so that the DIST recursion
can easily be implemented on a computer. We leave the derivations
to the reader. 

Given discrete scattering data \[
S=(a,b;w_{1},\dots,w_{m};c_{1}^{\prime},\dots,c_{m}^{\prime}),\]
we set \begin{eqnarray*}
r & = & \frac{b}{a}\\
s & = & -\frac{b^{*}}{a}\end{eqnarray*}
 and \begin{eqnarray*}
c_{k} & = & \frac{c_{k}^{\prime}}{a^{\prime}(w_{k})}\\
\tilde{c}_{k} & = & -\frac{(c_{k}^{\prime})^{-1}w_{k}^{-1}}{a^{\prime}(w_{k})}.\end{eqnarray*}

\begin{itemize}
\item Step 1: Define the following sequences:\begin{eqnarray*}
f(n) & = & \hat{r}(n)-\sum_{k=1}^{m}c_{k}w_{k}^{-n}\\
g(n) & = & \hat{s}(n)-\sum_{k=1}^{m}\tilde{c}_{k}w_{k}^{-n-1}.\end{eqnarray*}

\item Step 2: Choose $M_{+}>>0$ and set \[
\left[\begin{array}{c}
K_{+,M_{+}}\\
L_{+,M_{+}}\end{array}\right]=\left[\begin{array}{c}
1\\
0\end{array}\right].\]
Then define the polynomials $K_{+,j}$ and $L_{+,j}$ (for $j=0,1,\dots,M_{+}$)
recursively using\[
\left[\begin{array}{c}
K_{+,j-1}\\
L_{+,j-1}\end{array}\right]=\left[\begin{array}{cc}
1 & -\gamma_{j}^{*}\\
\gamma_{j}w & w\end{array}\right]\left[\begin{array}{c}
K_{+,j}\\
L_{+,j}\end{array}\right]\]
 and \[
-\gamma_{j}^{*}=\frac{\sum_{n=0}^{\infty}f(-j-n)\hat{K}_{+,j}(n)}{\hat{K}_{+,j}(n)-\sum_{n=0}^{\infty}f(-j-n)\hat{L}_{+,j}(n)}.\]

\item Step 3: Choose $M_{-}>>0$ and set \[
\left[\begin{array}{c}
K_{-,-M_{-}}\\
L_{-,-M_{-}}\end{array}\right]=\left[\begin{array}{c}
1\\
0\end{array}\right].\]
Then define the polynomials $K_{-,j}$ and $L_{-,j}$ (for $j=-M_{-},\dots,-1,0$)
recursively using \[
\left[\begin{array}{c}
K_{+,j+1}\\
L_{+,j+1}\end{array}\right]=\left[\begin{array}{cc}
1 & \gamma_{j}\\
-\gamma_{j}^{*}w & w\end{array}\right]\left[\begin{array}{c}
K_{+,j}\\
L_{+,j}\end{array}\right]\]
 and\[
\gamma_{j}=\frac{\sum_{n=0}^{\infty}g(j-n)\hat{K}_{-,j}(n)}{\hat{K}_{-,j}(n)-\sum_{n=0}^{\infty}g(j-n)\hat{L}_{-,j}(n)}.\]

\item Step 4: Set \[
\omega(t)=\sum_{j=-M_{-}}^{M_{+}}\omega_{j}\delta(t-j\Delta),\]
 where \[
\omega_{j}=2i\frac{\gamma_{j}^{*}}{|\gamma_{j}|}\arctan|\gamma_{j}|.\]

\end{itemize}
\begin{rem}
If the left and right values of $\gamma_{0}$ are inconsistent, then
$M_{+}$ and $M_{-}$ should be increased. Initially, these integers
should be chosen so that \[
I_{+}=\sum_{n=M_{+}}^{\infty}|f(n)|\]
 and \[
I_{-}=\sum_{n=M_{-}}^{\infty}|g(n)|\]
 are small.
\end{rem}

\chapter{Scattering on Lie groups}

\section{\label{sub:The-scattering-transform-on-Lie-groups}The scattering
transform on Lie groups}

The selective excitation transform fits into a more general framework.
Let $G$ be a Lie group acting on a space $X$, with a special point
$x_{0}\in X$. Fix an element $J\in\Lg$ in the Lie algebra of $G$
such that the one parameter subgroup $\left\{ e^{Jz}:z\in\RR\right\} \subset G$
is isomorphic to $S^{1}$. Also, fix a subspace $\Lk\subset\Lg$.
Let $\omega:\RR\rightarrow\Lk\subset\Lg$ be a function with sufficient
smoothness and decay, so that there is a solution $v_{-}:\RR\times\RR\rightarrow X$
to the equation \begin{equation}
\partial_{t}v_{-}(z;t)=\left(\omega(t)+zJ\right)v_{-}(z;t)\label{eq:st 1}\end{equation}
normalized by \begin{equation}
\lim_{t\rightarrow-\infty}e^{-Jzt}v_{-}(z;t)=x_{0}.\label{eq:st 2}\end{equation}
 For now, we assume that the solution $v_{-}$ is unique, and that
\[
(\TT_{G,J,\Lk,X,x_{0}}\omega)(z):=\lim_{t\rightarrow\infty}e^{-Jzt}v_{-}(z;t)\in X\]
 exists for all $z$. The operator $\TT_{G,J,\Lk,X,x_{0}}$ is called
the \textit{scattering transform}. It maps $\Lk$-valued functions
of time into $X$-valued functions of frequency. 

If we take $G=SO_{3}(\RR)$, $X=\RR^{3}$, $x_{0}=\left[\begin{array}{c}
0\\
0\\
1\end{array}\right]$, $J=\left[\begin{array}{ccc}
0 & 1 & 0\\
-1 & 0 & 0\\
0 & 0 & 0\end{array}\right]\in so_{3}(\RR)$, and \[
\Lk=\left\{ \left[\begin{array}{ccc}
0 & 0 & -\im x\\
0 & 0 & \re x\\
\im x & -\re x & 0\end{array}\right]:\; x\in\CC\right\} ,\]
 then $\TT_{G,J,\Lk,X,x_{0}}$ coincides with the selective excitation
transform $\TT$ defined above. For other examples of scattering transforms,
see Section (\ref{sub:Three-fundamental-scattering-transforms}).

\section{The scattering equation\label{sub:The-scattering-equation}}

Let us consider the scattering transform in the special case where
$X=G$, $x_{0}=1\in G$, and $G$ acts on itself by left multiplication.
To simplify notation we write \[
\TT_{G,J,\Lk}:=\TT_{G,J,\Lk,G,1}.\]
 This scattering transform is universal in the following sense:

\begin{prop}
\label{pro:universal}Let $G$, $J$, $\Lk$, $X$ and $x_{0}$ be
as in Section \ref{sub:The-scattering-transform-on-Lie-groups}. Then
\[
\TT_{G,J,\Lk,X,x_{0}}\omega=(\TT_{G,J,\Lk}\omega)(x_{0}).\]

\end{prop}
We omit the proof, which is a straightforward application of the definitions.

In this universal case, we consider two solutions $v_{\pm}:\RR\times\RR\rightarrow G$
to the equation \begin{equation}
\partial_{t}v_{\pm}(z;t)=\left(\omega(t)+zJ\right)v_{\pm}(z;t)\label{eq:se 1}\end{equation}
normalized by \begin{equation}
\lim_{t\rightarrow\pm\infty}e^{-Jzt}v_{\pm}(z;t)=1\in G.\label{eq:se 2}\end{equation}
 The function \begin{equation}
S_{\omega}(z):=v_{+}(z;t)^{-1}v_{-}(z;t)\label{eq:se 3}\end{equation}
 is independent of $t$ because given any two solutions $v_{1}$ and
$v_{2}$ to equation (\ref{eq:se 1}) we have \begin{eqnarray*}
\partial_{t}v_{1}^{-1}v_{2} & = & -v_{1}^{-1}(\partial_{t}v_{1})v_{1}^{-1}v_{2}+v_{1}^{-1}\partial_{t}v_{2}\\
 & = & -v_{1}^{-1}\left(\omega(t)+zJ\right)v_{2}+v_{1}^{-1}\left(\omega(t)+zJ\right)v_{2}\\
 & = & 0.\end{eqnarray*}
This calculation also demonstrates the uniqueness of $v_{-}$ and
$v_{+}$. Letting $t$ tend to $\infty$, we find that \[
S_{\omega}(z)=\lim_{t\rightarrow\infty}e^{-Jzt}v_{-}(z;t)=(\TT_{G,J,\Lk}\omega)(z)\]
 so that $S_{\omega}$ is the scattering transformation $\TT_{G,J,\Lk}$
applied to $\omega$. By Proposition \ref{pro:universal} we have
\[
(\TT_{G,J,\Lk,X,x_{0}}\omega)(z)=S_{\omega}(z)\cdot x_{0}\]
 whenever $G$ acts on a space $X$.

It is useful to define \[
u_{\pm}(z;t):=v_{\pm}(z;t)e^{-Jzt}.\]
 In terms of this notation equations (\ref{eq:se 1}) through (\ref{eq:se 3})
become \begin{equation}
\partial_{t}u_{\pm}(z;t)=\omega(t)u_{\pm}(z;t)+z\left(Ju_{\pm}(z;t)-u_{\pm}(z;t)J\right)\label{eq:se 4}\end{equation}
\begin{equation}
\lim_{t\rightarrow\pm\infty}u_{\pm}(z;t)=1\in G,\label{eq:se 5}\end{equation}
 and \begin{equation}
S_{\omega}(z)=e^{-Jzt}v_{+}(z;t)^{-1}u_{-}(z;t)e^{Jzt}.\label{eq:se 6}\end{equation}
Equation (\ref{eq:se 6}) corresponds to the scattering equation (\ref{eq:ct6}).

\section{Three scattering transforms\label{sub:Three-fundamental-scattering-transforms}}

Let us fix $X=\CCh$ to be the Riemann Sphere with $x_{0}=0$ and
let $e^{J\RR}\cong S^{1}$ be the group of rotations around the origin.
There are three basic examples of scattering transforms (see Section
\ref{sub:The-scattering-transform-on-Lie-groups}):

\begin{itemize}
\item Euclidean: $G=G_{E}=Aut(\CC)$ is the group of invertible affine transformations
\[
G_{E}=\left\{ z\mapsto\lambda z+c:\;\lambda,c\in\CC,\;\lambda\neq0\right\} .\]

\item Spherical: $G=G_{S}\cong SO_{3}\RR$ is the group of rigid rotations
of the sphere.
\item Hyperbolic: $G=G_{H}\cong PSL_{2}\RR$ is the group of conformal automorphisms
of the unit disk.
\end{itemize}
The second example corresponds to the selective excitation transform.
Let us show that the first example (Euclidean) corresponds to the
inverse Fourier transform.

We can represent \[
G_{E}=\left\{ \left[\begin{array}{cc}
\lambda & x\\
0 & 1\end{array}\right]:\;\lambda\in\CC^{*},x\in\CC\right\} .\]
 It is natural to choose $\Lk=\left[\begin{array}{cc}
0 & \CC\\
0 & 0\end{array}\right]$. Let us write \[
\omega(t)=\left[\begin{array}{cc}
0 & \omega(t)\\
0 & 0\end{array}\right]\]
 and \[
u_{-}(z;t)=\left[\begin{array}{cc}
1 & u_{-}(z;t)\\
0 & 1\end{array}\right].\]
 Then equation (\ref{eq:se 4}) is \[
\partial_{t}\left[\begin{array}{cc}
1 & u_{-}(z;t)\\
0 & 1\end{array}\right]=\left[\begin{array}{cc}
0 & \omega(t)\\
0 & 0\end{array}\right]\left[\begin{array}{cc}
1 & u_{-}(z;t)\\
0 & 1\end{array}\right]+z\left[\begin{array}{cc}
0 & iu_{-}(z;t)\\
0 & 0\end{array}\right]\]
 or \[
\partial_{t}u_{-}(z;t)=\omega(t)+izu_{-}(z;t).\]
 The solution normalized at $-\infty$ is \[
u_{-}(z;t)=\int_{-\infty}^{t}\omega(s)e^{i(t-s)z}ds,\]
 and so \[
S_{\omega}(z)\cdot x_{0}=\int_{-\infty}^{\infty}\omega(s)e^{-isz}ds.\]
Therefore $T_{G,J,\Lk,X,x_{0}}$ is the inverse Fourier transform.

\section{The discrete scattering transform on Lie groups}

\label{sub:The-discrete-scattering-transform-on-Lie-groups}Let $G$,
$J$, $X$, and $x_{0}\in X$ be as in Section \ref{sub:The-scattering-transform-on-Lie-groups}.
That is, $G$ is a Lie group acting on $X$, and $J\in\Lg$ is an
element of the Lie algebra such that the one-parameter subgroup $e^{J\RR}\subset G$
is isomorphic to $S^{1}$. This time we consider functions of the
form $\Omega:\ZZ\rightarrow K\subset G$, where $K$ is some subset
of $G$. Let us identify $S^{1}$ with the subgroup $e^{J\RR}\subset G$,
and set $\Omega_{j}=\Omega(j)$. Suppose there is a solution $V_{-}:S^{1}\times\ZZ\rightarrow X$
to the recursion 

\[
V_{-}(w;j+1)=w\Omega_{j}V_{-}(w;j)\]
 normalized by\[
\lim_{j\rightarrow-\infty}w^{-j}V_{-}(w;j)=x_{0}.\]
 Again, we assume that the solution is unique, and that \[
(\TT_{G,J,K,X,x_{0}}^{\textrm{disc}}\Omega)(w):=\lim_{j\rightarrow+\infty}w^{-j}V_{-}(w;j)\in X\]
 exists for every $w$. The transform $\TT_{G,J,K,X,x_{0}}^{\textrm{disc}}$
is called the \textit{discrete scattering transform}. It maps sequences
in $K$ to loops in $X$. 

If we take $G=SO_{3}(\RR)$, $X=\RR^{3}$, $x_{0}=\left[\begin{array}{c}
0\\
0\\
1\end{array}\right]$, $J=\left[\begin{array}{ccc}
0 & 1 & 0\\
-1 & 0 & 0\\
0 & 0 & 0\end{array}\right]\in so_{3}(\RR)$, and \[
K=\exp\left\{ \left[\begin{array}{ccc}
0 & 0 & -\im x\\
0 & 0 & \re x\\
\im x & -\re x & 0\end{array}\right]:\; x\in\CC\right\} ,\]
 then $\TT_{G,J,K,X,x_{0}}^{\textrm{disc}}$ coincides with the discrete
selective excitation transform $\TT^{\textrm{disc}}$ defined above.
Notice that $K$ consists of the rotations around axes orthogonal
to the $z$-axis. For other examples of discrete scattering transforms,
see Section (\ref{sec:Three-discrete-scattering}).

\section{The discrete scattering equation\label{sub:The-discrete-scattering-equation}}

Let us consider the discrete scattering transform in the special case
where $X=G$, $x_{0}=1$ is the identity element, and $G$ acts on
$X$ by left multiplication. To simplify notation we write \[
\TT_{G,J,K}^{\textrm{disc}}:=\TT_{G,J,K,G,1}^{\textrm{disc}}.\]
 This discrete scattering transform is universal in the following
sense:

\begin{prop}
Let $G$, $J$, $K$, $X$ and $x_{0}$ be as in Section \ref{sub:The-discrete-scattering-transform-on-Lie-groups}.
Then \[
\TT_{G,J,K,X,x_{0}}^{\textrm{disc}}\Omega=(\TT_{G,J,K}^{\textrm{disc}}\Omega)(x_{0}).\]

\end{prop}
We omit the proof, which is a straightforward application of the definitions.

In this universal case, we consider two solutions $V_{\pm}:S^{1}\times\ZZ\rightarrow G$
to the recursion \begin{equation}
V_{\pm}(w;j+1)=w\Omega_{j}V_{\pm}(w;j)\label{eq:sed 1}\end{equation}
 normalized by\begin{equation}
\lim_{j\rightarrow\pm\infty}w^{-j}V_{\pm}(w;j)=1\in G.\label{eq:sed 2}\end{equation}
 The function \begin{equation}
S_{\Omega}^{\textrm{disc}}(w):=V_{+}(w;j)^{-1}V_{-}(w;j)\label{eq:sed 3}\end{equation}
 is independent of j because given any two solutions $V_{1}$ and
$V_{2}$ to the recursion (\ref{eq:sed 1}), we have\begin{eqnarray*}
V_{1}(w;j+1)^{-1}V_{2}(w;j+1) & = & \left(V_{1}(w;j)^{-1}\Omega_{j}^{-1}w^{-1}\right)\left(w\Omega_{j}V_{2}(w;j)\right)\\
 & = & V_{1}(w;j)^{-1}V_{2}(w;j).\end{eqnarray*}
 This calculation, in particular, shows that $V_{-}$ and $V_{+}$
are unique. Letting $j$ tend to $\infty$, we find that \[
S_{\Omega}^{\textrm{disc}}(w)=\lim_{j\rightarrow\infty}w^{-j}V_{-}(w;j)=(\TT_{G,J,\Lk}^{\textrm{disc}}\Omega)(w)\]
 so that $S_{\Omega}^{\textrm{disc}}$ is the discrete scattering
transformation $\TT_{G,J,\Lk}^{\textrm{disc}}$ applied to $\Omega$.

It is useful to define \[
U_{\pm}(w;j):=V_{\pm}(w;j)w^{-j}.\]
 In terms of this notation, equations (\ref{eq:sed 1}) through (\ref{eq:sed 3})
become \begin{equation}
U_{\pm}(w;j+1)=w\Omega_{j}U_{\pm}(w;j)w^{-1}\label{eq:sed 4}\end{equation}
\begin{equation}
\lim_{j\rightarrow\pm\infty}U_{\pm}(w;j)=1\in G,\label{eq:sed 5}\end{equation}
 and \begin{equation}
S_{\Omega}^{\textrm{disc}}(w)=w^{-j}U_{+}(w;j)^{-1}U_{-}(w;j)w^{j}.\label{eq:sed 6}\end{equation}
Equation (\ref{eq:sed 6}) corresponds to the scattering equation
(\ref{eq:dt7}).

\section{Three discrete scattering transforms\label{sec:Three-discrete-scattering}}

As in Section \ref{sub:Three-fundamental-scattering-transforms} ,
we fix $X=\CCh$ to be the Riemann Sphere with $x_{0}=0$ and let
$e^{J\RR}\cong S^{1}$ be the group of rotations around the origin.
There are three basic examples of discrete scattering transforms (see
Section \ref{sub:The-discrete-scattering-transform-on-Lie-groups}):

\begin{itemize}
\item Euclidean: $G=G_{E}=Aut(\CC)$ is the group of invertible affine transformations
\[
G_{E}=\left\{ z\mapsto\lambda z+c:\;\lambda,c\in\CC,\;\lambda\neq0\right\} .\]

\item Spherical: $G=G_{S}\cong SO_{3}\RR$ is the group of rigid rotations
of the sphere.
\item Hyperbolic: $G=G_{H}\cong PSL_{2}\RR$ is the group of conformal automorphisms
of the unit disk.
\end{itemize}
The second example corresponds to the discrete selective excitation
transform. Let us show that the first example (Euclidean) corresponds
to the discrete inverse Fourier transform.

We can represent \[
G_{E}=\left\{ \left[\begin{array}{cc}
\lambda & x\\
0 & 1\end{array}\right]:\;\lambda\in\CC^{*},x\in\CC\right\} .\]
 It is natural to choose $K=\left[\begin{array}{cc}
1 & \CC\\
0 & 1\end{array}\right]$. Let us write \[
\Omega_{j}=\left[\begin{array}{cc}
1 & \omega_{j}\\
0 & 1\end{array}\right]\]
 and \[
U_{-}(w;j)=\left[\begin{array}{cc}
1 & U_{-}(w;j)\\
0 & 1\end{array}\right].\]
 Then equation (\ref{eq:sed 4}) is \[
\left[\begin{array}{cc}
1 & U_{-}(w;j+1)\\
0 & 1\end{array}\right]=\left[\begin{array}{cc}
w & 0\\
0 & 1\end{array}\right]\left[\begin{array}{cc}
1 & \omega_{j}\\
0 & 1\end{array}\right]\left[\begin{array}{cc}
1 & U_{-}(w;j)\\
0 & 1\end{array}\right]\left[\begin{array}{cc}
w^{-1} & 0\\
0 & 1\end{array}\right]\]
 or \[
U_{-}(w;j+1)=w(U_{-}(w;j)+\omega_{j}).\]
 The solution normalized at $-\infty$ is \[
U_{-}(w;j)=\sum_{k=-\infty}^{j-1}\omega_{k}w^{(j-k)},\]
 and so \[
(T_{G,J,K,X,x_{0}}^{\textrm{disc}}\Omega)(w)=S_{\omega}^{\textrm{disc}}(w)\cdot x_{0}=\sum_{k=-\infty}^{\infty}\omega_{k}w^{-k}.\]
Therefore $T_{G,J,K,X,x_{0}}^{\textrm{disc}}$ is the discrete inverse
Fourier transform.

\end{document}